\documentclass[a4paper, 11pt, reqno]{amsart}

\usepackage{microtype}
\usepackage{fullpage}

\usepackage{amssymb}
\usepackage{mathtools}
\usepackage{enumitem}

\newsavebox{\foobox}
\newcommand{\slantbox}[2][0]{\mbox{%
        \sbox{\foobox}{#2}%
        \hskip\wd\foobox
        \pdfsave
        \pdfsetmatrix{1 0 #1 1}%
        \llap{\usebox{\foobox}}%
        \pdfrestore
}}
\newcommand\unslant[2][-.25]{\slantbox[#1]{$#2$}}

\newcommand{\B}{\mathfrak{B}}
\newcommand{\BB}{\mathbb{B}}
\newcommand{\CC}{\mathbb{C}}
\newcommand{\dd}{\,\mathrm{d}}
\newcommand{\dom}{\mathfrak{D}}
\newcommand{\ee}{\mathrm{e}}
\newcommand{\eucmet}{\unslant\delta}

\newcommand{\FT}{{\mathcal{F}}}
\newcommand{\g}{\unslant g}
\newcommand{\HH}{\mathfrak{H}}
\newcommand{\I}{\operatorname{I}}
\newcommand{\ii}{\mathrm{i}}

\newcommand{\metdet}[1]{\sqrt{|\,{#1}\,|}\,}
\newcommand{\minkmet}{\unslant\eta}
\newcommand{\NN}{\mathbb{N}}
\newcommand{\od}{\mathrm{d}}
\newcommand{\odd}{\mathrm{odd}}
\newcommand{\pd}{\partial}
\newcommand{\rad}{\mathrm{rad}}
\newcommand{\ran}{\operatorname{ran}}
\renewcommand{\restriction}{\mathord{\upharpoonright}}
\newcommand{\RR}{\mathbb{R}}
\renewcommand{\Re}{\operatorname{Re}}
\newcommand{\supp}{\operatorname{supp}}

\newcommand{\X}{\mathfrak{X}}

\theoremstyle{definition}
\newtheorem{definition}{Definition}[section]

\theoremstyle{remark}
\newtheorem{remark}{Remark}[section]

\theoremstyle{plain}
\newtheorem{theorem}{Theorem}[section]

\theoremstyle{plain}
\newtheorem{proposition}{Proposition}[section]

\theoremstyle{plain}
\newtheorem{lemma}{Lemma}[section]

\numberwithin{equation}{section}

\usepackage{hyperref}
    \hypersetup{
        bookmarksdepth = 3,
               linktoc = page,
            colorlinks = true,
             linkcolor = blue,
             citecolor = green,
    }

\usepackage{cleveref}
    \crefname{equation}{eq.}{eqs.}
    \Crefname{equation}{Eq.}{Eqs.}

\title{A Globally Stable Self-Similar Blowup Profile in Energy Supercritical Yang-Mills Theory}

\author{Roland Donninger}
\address{University of Vienna, Faculty of Mathematics, Oskar-Morgenstern-Platz 1, 1090 Vienna, Austria}
\email{roland.donninger@univie.ac.at}

\author{Matthias Ostermann}
\address{University of Vienna, Faculty of Mathematics, Oskar-Morgenstern-Platz 1, 1090 Vienna, Austria}
\email{matthias.ostermann@univie.ac.at}

\begin{document}
\begin{abstract}
This paper is concerned with the Cauchy problem for an energy-supercritical nonlinear wave equation in odd space dimensions that arises in equivariant Yang-Mills theory. In each dimension, there is a self-similar finite-time blowup solution to this equation known in closed form. It will be proved that this profile is stable in the whole space under small perturbations of the initial data. The blowup analysis is based on a recently developed coordinate system called \textit{hyperboloidal similarity coordinates} and depends crucially on growth estimates for the free wave evolution, which will be constructed systematically for odd space dimensions in the first part of this paper. This allows to develop a nonlinear stability theory beyond the singularity.
\end{abstract}
\maketitle
\tableofcontents
\newpage
\section{Introduction}
In this work, we study blowup formation from smooth data in the Cauchy problem for the equivariant Yang-Mills equation in all odd supercritical dimensions and develop a nonlinear perturbation theory for an explicit self-similar blowup profile in the whole space.
\subsection{Motivation and related research}
We are interested in a spherically symmetric model in Yang-Mills theory, i.e. the gauge theory that describes the strong interaction within the Standard Model of Particle Physics. More precisely, let us consider Lie-algebra valued one-forms $A_{\mu}: \RR^{1,n} \rightarrow \mathfrak{so}(n)$ over $(1+n)$-dimensional Minkowski space-time as a model for the underlying \emph{gauge potential}. The corresponding action functional is given by
\begin{equation*}
\int_{\RR^{1,n}} \mathrm{tr} \left( F_{\mu\nu}F^{\mu\nu} \right) \,,
\end{equation*}
where the connection two-form $F_{\mu\nu} = \pd_{\mu} A_{\nu} - \pd_{\nu} A_{\mu} + [A_{\mu},A_{\nu}] : \RR^{1,n} \rightarrow \mathfrak{so}(n)$ is the antisymmetric \emph{field strength}. Here and in the following, we use summation convention and let Greek indices $\mu,\nu$ range over $\{0,1,\ldots,n\}$ and Latin indices $i,j,k,\ell$ range over $\{1,\ldots,n\}$. The Euler-Lagrange equations associated to this action are provided by the Yang-Mills equation
\begin{equation}
\pd_{\mu} F^{\mu\nu} + [A_{\mu},F^{\mu\nu}] = 0 \,,
\end{equation}
see \cite{MR3475668}, \cite{MR1923684}. This equation is gauge invariant and assuming spherical symmetry \cite{MR682374}
\begin{equation}
\label{temporalgauge}
{A_{\mu}}^{ij}(t,x) = - \big( x^i {\eucmet_{\mu}}^{j} - x^{j} {\eucmet_{\mu}}^i \big) |x|^{-2} \psi(t,|x|)
\end{equation}
with a field $\psi: \RR\times[0,\infty) \rightarrow \RR$ fixes the gauge. In this case, we find that
\begin{align*}
{F_{00}}^{ij} &= 0 \,, \\
{F_{0k}}^{ij} &= - {F_{k0}}^{ij} = - r^{-2} \big( x^i \eucmet_k^{j} - x^{j} \eucmet_k^i \big) \psi_{t} \,, \\
{F_{k\ell}}^{ij} &= -r^{-2} \big( \eucmet_k^i \eucmet_\ell^{j} - \eucmet_k^{j} \eucmet_\ell^i \big) \big( 2\psi + \psi^{2} \big) \\&\quad-
r^{-2}\big( x^i x_k \eucmet_\ell^{j} - x^i x_\ell \eucmet_k^{j} + x^{j} x_\ell \eucmet_k^i - x^{j} x_k \eucmet_\ell^i \big) \big( r^{-1} \psi_{r} - r^{-2}(2\psi + \psi^{2}) \big) \,,
\end{align*}
which yields
\begin{align*}
\pd^\mu {F_{\mu 0}}^{ij} &= 0 \,, \\
\pd^\mu {F_{\mu \ell}}^{ij} &= r^{-2} \big( x^i \eucmet^{j}_\ell - x^{j} \eucmet^i_\ell \big) \big( \psi_{tt} - \psi_{rr} - \frac{n-3}{r} \psi_{r} + \frac{n-2}{r^{2}} (2\psi + \psi^{2}) \big) \,, \\
[A^\mu, F_{\mu0}]^{ij} &= 0 \,, \\
[A^\mu, F_{\mu\ell}]^{ij} &= r^{-2} \big( x^i \eucmet_\ell^{j} - x^{j} \eucmet^i_\ell \big) \frac{n-2}{r^{2}} \big( 2\psi^{2} + \psi^3 \big) \,,
\end{align*}
where we denote $r = |x|$ for brevity. The Euler-Lagrange equations reduce to the \emph{equivariant Yang-Mills equation}
\begin{equation}
\label{YMEq}
\left( - \pd_{t}^{2} + \pd_{r}^{2} + \frac{n-3}{r} \pd_{r} \right) \psi(t,r) - \frac{n-2}{r^{2}} \psi(t,r)\big( \psi(t,r)+1 \big)\big( \psi(t,r)+2 \big) = 0 \,,
\end{equation}
which is a radial semilinear wave equation. This model admits a positive definite conserved energy
\begin{equation}
E_\psi(t) = \int_{0}^\infty \left( \pd_{t}\psi(t,r)^{2} + \pd_{r}\psi(t,r)^{2} + \frac{n-2}{2} r^{-2} \psi(t,r)^{2} \big( \psi(t,r) + 1 \big)^{2} \right) r^{n-3} \dd r \,,
\end{equation}
which may be produced by testing the Yang-Mills equation with $\pd_{t} \psi(t,r)$ over the whole space and integrating by parts. Moreover, \Cref{YMEq} is invariant under the scaling $\psi^{\lambda}(t,r) = \psi(t/\lambda,r/\lambda)$, for $\lambda>0$, and we have for the energy
\begin{equation*}
E_{\psi^{\lambda}}(t) = \lambda^{n-4} E_\psi(t/\lambda) \,.
\end{equation*}
Thus, the equation is subcritical if $n\leq 3$, critical if $n=4$ and supercritical for $n\geq 5$. According to classification into criticality class \cite{MR1826256}, the global behaviour of solutions to the Cauchy problem for \Cref{YMEq} is expected to depend on the space dimension $n$.
\par\medskip
In the physical dimension $n=3$, the Yang-Mills equation exhibits no singularities. An early global existence result in the temporal gauge was shown by D. Eardley and V. Moncrief \cite{MR649158}, \cite{MR649159} for data in Sobolev spaces without symmetry or smallness assumptions. An alternative approach of S. Klainerman and M. Machedon \cite{MR1338675} allowed them to lower regularity and prove this result in the energy space. Local well-posedness exclusively in the temporal gauge for small data below the energy norm was given by T. Tao \cite{MR1964470}. For related results in the Lorenz gauge see e.g. A. Tesfahun \cite{MR3385623}, \cite{MR3439099}, S. Selberg and A. Tesfahun \cite{MR3519539}. A novel proof for local and global well-posedness in the energy norm that pursues a robust approach to the gauge choice was presented in a work by S.-J. Oh \cite{MR3190112}, \cite{MR3357182}. Lately, the techniques of S.-J. Oh and D. Tataru \cite{MR3907955} allowed them to give an alternative proof of local well-posedness in the temporal gauge without smallness assumptions.
\par\medskip
In the critical dimension $n=4$, S. Klainerman and D. Tataru \cite{MR1626261} considered local well-posedness in the Coulomb gauge at optimal regularity. Global well-posedness in the Coulomb gauge for small energy data was proved by J. Krieger and D. Tataru \cite{MR3664812}. Earlier, J. Sterbenz \cite{MR2325100} proved global existence for small data in Besov spaces. Recently, S.-J. Oh and D. Tataru \cite{MR3907955} obtained in combination with their previous paper \cite{MR4113787} local well-posedness in the temporal gauge at optimal regularity. Moreover, the Yang-Mills equation develops singularities in dimension $n=4$ which is closely connected to the existence of a static finite-energy solution called the \emph{instanton}. For the equivariant Yang-Mills equation, R. C\^{o}te, C. Kenig and F. Merle \cite{MR2443303} proved global existence and scattering, either to zero or to a rescaling of the instanton, for data with energy less or equal to the energy of the instanton. Based on numerical studies of radially symmetric solutions, existence of blowup was conjectured by P. Bizo\'{n} \cite{MR1923684} and presented numerically in \cite{MR1923684}, \cite{MR1888849}. The blowup rate for radial solutions was derived by P. Bizo\'{n}, Y. Ovchinnikov and I. Sigal \cite{MR2069700}. The proof for the existence of blowup solutions to the equivariant Yang-Mills equation is due to J. Krieger, W. Schlag and D. Tataru \cite{MR2522426} and P. Rapha\"{e}l and I. Rodnianski \cite{MR2929728}, the latter also obtained the stable blowup rate. One of the latest results for the $(4+1)$-dimensional Yang-Mills equation concerns a proof of the \emph{Threshold Conjecture} and the \emph{Dichotomy Theorem} in the significant line of papers \cite{MR4518477}, \cite{MR4113787}, \cite{MR3907955}, \cite{MR4298746}, \cite{MR3923343} by S.-J. Oh and D. Tataru.
\par\medskip
The supercritical setting $n\geq 5$, which shall concern us, is much less understood. Global existence holds for small data, see A. Stefanov \cite{MR2735817} for $n=5$ in the Coulomb gauge and \cite{MR3087010} for results concerning $n\geq 6$. Local well-posedness at optimal regularity in the temporal gauge obtained in \cite{MR3907955} holds indeed in all dimenions $n\geq 5$ as well. However, the occurrence of finite time blowup obstructs global existence for \emph{large} data, as has been demonstrated in \cite{MR1622539}. Apart from the relevance for our understanding of large data evolution in supercritical evolution equations, singularity formation from smooth data is also of physical interest because for $n=5$, the Yang-Mills equation enjoys the same scaling behaviour as Einstein's equations in $(3+1)$ dimensions and is thus proposed to serve as a toy model for gravitational collapse in General Relativity, see \cite{MR1923684}, \cite{MR3475668}. Breakdown of solutions has been demonstrated by P. Bizo\'{n} \cite{MR1923684} for $n=5$ who constructed a countable family of self-similar solutions. The ground state of this family is known in closed form and given by
\begin{equation}
\label{BizonBlowup}
\psi_{T}(t,r) = f_{0}(\tfrac{r}{T-t})\,, \quad f_{0}(\rho) = -\frac{8\rho^{2}}{5 + 3\rho^{2}} \,, \quad (n=5)
\end{equation}
where the \emph{blowup time} $T\in\RR$ is a real parameter due to time translation invariance. By now, P. Biernat and P. Bizo\'{n} \cite{MR3355819} discovered that this profile is also present in all supercritical dimensions,
\begin{equation}
\label{BiernatBizonBlowup}
\psi_{T}(t,r) = f_{0}(\tfrac{r}{T-t})\,, \quad f_{0}(\rho) = -\frac{a(n) \rho^{2}}{b(n) + \rho^{2}} \,, \quad (n\geq 5)
\end{equation}
where
\begin{equation}
\label{BiernatBizonBlowupCoeff}
a(n) = 2\left( 1 + \sqrt{\frac{n-4}{3(n-2)}} \right) \,, \qquad
b(n) = \frac{1}{3} \left( 2(n-4) + \sqrt{3(n-2)(n-4)} \right) \,.
\end{equation}
These solutions are smooth for all $t<T$ but blow up in finite time at $t=T$. Our present paper ties in with \cite{MR3278903}, where the first author proved the nonlinear stability of $\psi_{T}$ in the backward light cone of the singularity, which stresses the relevance of this solution for our understanding of the dynamics of the Yang-Mills equation. The proof assumed \emph{mode stability} of the profile, which had been established numerically in \cite{bizon2005convergence} and later proved to hold in \cite{MR3475668}. Most recently, I. Glogi\'{c} \cite{MR4469070} extended this result to all odd dimensions by proving that the profile \eqref{BiernatBizonBlowup} is stable in lightcones in the sense that the Cauchy evolution of initial data close to that profile leads to finite time blowup described by $\psi_{T}$. His work also resolves rigorously the higher dimensional mode stability problem with improved methods.
\par\medskip
At this stage, two major questions arise. Firstly, the present results leave open how such solutions behave outside the backward light cone. Do they still approach the blowup profile or are there other effects to be expected? Secondly, although $\psi_{T}$ is only defined for times $t<T$, it has a natural extension
\begin{equation*}
\label{BizonBlowupExtended}
\psi_{T}^{*}(t,r) = - \frac{a(n) r^{2}}{b(n)(T-t)^{2} + r^{2}}
\end{equation*}
beyond the blowup, which is a smooth solution to \Cref{YMEq} in dimensions $n\geq 5$ for all $(t,r)\in\RR\times(0,\infty)$ with the boundary condition $\psi_{T}^{*}(t,0) = 0$ for all $t<T$. The blowup of $\psi_{T}$ at $(T,0)$ does not change this boundary condition at later times $t>T$. After the blowup we have $\lim_{t\to\infty} \psi_{T}^{*}(t,r) = 0$ for all $r>0$. So, is it possible to continue blowup solutions beyond the singularity in a well defined way?
\par\medskip
To tackle these questions, we implement the programme suggested by P. Biernat, R. Donninger and B. Sch\"{o}rkhuber \cite{MR4338226}, who treated the analogous questions for three-dimensional wave maps. There, the key insight was the construction of a novel coordinate system called \emph{hyperboloidal similarity coordinates} that is adapted to self-similarity and covers a larger portion of space-time than the so-called \emph{standard similarity coordinates} that were applied in \cite{MR3278903}, see the discussion below.
\subsection{Statement of the main result}
\begin{figure}
\centering
\includegraphics[width=0.8\textwidth]{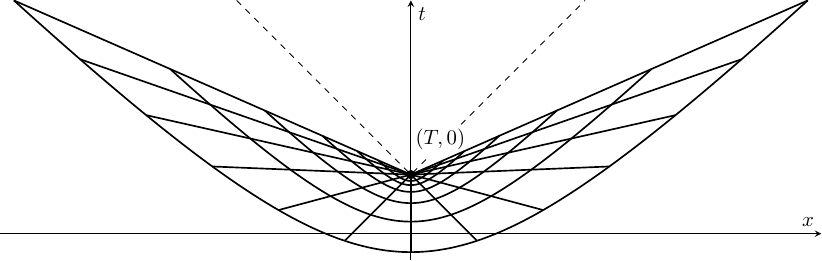}
\caption{Hyperboloidal similarity coordinates on $\RR^{1,1}$. The hyperboloids correspond to the level sets of $s$. The radial lines are level sets of $y$. These coordinates cover the whole complement of the future light cone at $(T,0)$.}
\label{Fig_HSC}
\end{figure}
\begin{figure}
\centering
\includegraphics[width=0.8\textwidth]{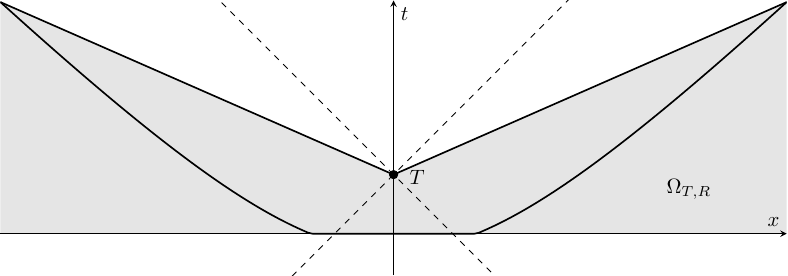}
\caption{The solutions constructed in \Cref{YMglobThm} are smoothly defined on the grey shaded region $\Omega_{T,R}$. The solid lines have slope $b$ and may be adjusted to come arbitrarily close to the dashed boundary of the future light cone at $(T,0)$. Together with the curved lines they bound the region $\Omega_{T,R} \setminus \eta_{T}\big( [s_{0},\infty)\times\BB_{R} \big)$.}
\label{Fig_OmegaTb}
\end{figure}
Introduction of the new field $\widehat{u}(t,r) = r^{-2} \psi(t,r)$ removes the singularity in the nonlinearity of \Cref{YMEq} and we infer for $n=d-2$
\begin{equation*}
\left( - \pd_{t}^{2} + \pd_{r}^{2} + \frac{d-1}{r} \pd_{r} \right) \widehat{u}(t,r) = (d-4)\big(r^{2} \widehat{u}(t,r)^3 + 3\widehat{u}(t,r)^{2} \big) \,.
\end{equation*}
Via $u(t,x) = \widehat{u}(t,|x|)$ for $(t,x)\in\RR^{1,d}$ we have the radial wave equation
\begin{equation}
\label{YMwave}
- \pd_{t}^{2} u(t,x) + \Delta_{x} u(t,x) = (d-4)(|x|^{2}u(t,x)^3 + 3u(t,x)^{2})
\end{equation}
with a regular nonlinearity in effectively $d=n+2$ space dimensions. Note that by $u_{T}^{*}(t,x) = \widehat{u}_{T}^{*}(t,|x|) = |x|^{-2}\psi_{T}^{*}(t,|x|)$ the corresponding blowup solution \eqref{BizonBlowupExtended} reads explicitly
\begin{equation}
\label{selfsimilarblowupextended}
u_{T}^{*}(t,x) = - \frac{a_d}{b_d(T-t)^{2} + |x|^{2}} \,,
\end{equation}
where $a_d = a(d-2)$ and $b_d = b(d-2)$ as given in \Cref{BiernatBizonBlowupCoeff}, and is defined for all $(t,x) \in \RR^{1,d} \setminus \{ (T,0)\}$. We present our main result to the questions posed above in the following theorem and discuss its content below.
\begin{theorem}
\label{YMglobThm}
Let $d\geq 7$ be an odd integer. Fix \textnormal{hyperboloidal similarity coordinates}
\begin{equation*}
\eta_{T}(s,y) = \big( T+\ee^{-s}h(y),\ee^{-s}y \big) \,, \quad h(y) = \sqrt{2+|y|^{2}} - 2 \,, \quad (s,y)\in\RR^{1,d} \,, \quad T\in\RR\,,
\end{equation*}
see \Cref{Fig_HSC}. For $R\geq \frac{1}{2}$ consider the region
\begin{equation*}
\Omega_{T,R} = \big\{ (t,x)\in\RR^{1,d} \mid 0\leq t < T + b|x| \big\} \,, \quad b = \frac{h(R)}{R} \in [-1,1) \,,
\end{equation*}
see \Cref{Fig_OmegaTb}. Let $k\in\NN$, $k\geq \frac{d+1}{2}$. There are positive constants $M_{0},\delta_{0},\varepsilon,\omega_{0} > 0$ such that for any $0<\delta\leq\delta_{0}$ and $M\geq M_{0}$ the following holds.
\begin{enumerate}[itemsep=1em,topsep=1em]
\item For any pair of radial functions $(f,g)\in C^\infty(\RR^d)\times C^\infty(\RR^d)$ with support in $\BB_\varepsilon^d\times\BB_\varepsilon^d$ and with
\begin{equation*}
\| (f,g) \|_{H^{k+\frac{d+1}{2}}(\RR^d)\times H^{k+\frac{d-1}{2}}(\RR^d)} \leq \frac{\delta}{M^{2}}
\end{equation*}
there exists a $T\in [1-\delta/M,1+\delta/M]$ and a unique solution $u\in C^\infty(\Omega_{T,R})$ to
\begin{equation*}
\begingroup
\renewcommand{\arraystretch}{1.2}
\left\{
\begin{array}{rcll}
(-\pd_{t}^{2} + \Delta_{x})u(t,x) &=& (d-4)(|x|^{2}u(t,x)^3 + 3u(t,x)^{2}) \,, & (t,x) \in \Omega_{T,R} \,, \\
u(0,x) &=& u_{1}^{*}(0,x) + f(x) \,, & x\in\RR^d \,,  \\
(\pd_{0}u)(0,x) &=& (\pd_{0}u_{1}^{*})(0,x) + g(x) \,, & x\in\RR^d \,,
\end{array}
\right.
\endgroup
\end{equation*}
such that $u = u_{1}^{*}$ in the domain $\Omega_{T,R} \setminus \eta_{T}\big( [s_{0},\infty) \times \BB^{d}_{R} \big)$, where $s_{0} = \log\big(-\tfrac{h(0)}{1+2\varepsilon}\big)$, and
\begin{align*}
\ee^{-2s} \left\| (u-u_{T}^{*})\circ\eta_{T}(s,\,.\,) \right\|_{H^{k}(\BB^{d}_{R})} &\leq \delta \ee^{-\omega_{0} s} \,, \\
\ee^{-2s} \left\| \pd_s (u-u_{T}^{*})\circ\eta_{T}(s,\,.\,) \right\|_{H^{k-1}(\BB^{d}_{R})} &\leq \delta \ee^{-\omega_{0} s} \,,
\end{align*}
for all $s\geq 0$.
\item In particular, within the domain $\eta_{T}\big( [s_{0},\infty) \times \BB^{d}_{R} \big) \subset \RR^{1,d}$, the solution $u$ blows up and converges to $u_{T}^{*}$ uniformly along the hyperboloids
\begin{equation*}
\mathrm{Hyp}^{1,d}_{T,R}(t) = \eta_{T}\big( \{s(t)\}\times\BB^{d}_{R} \big) \,, \qquad s(t) = \log\big(-\tfrac{h(0)}{T-t}\big) \,,
\end{equation*}
through $(t,0)\in\RR^{1,d}$ in the sense that
\begin{equation*}
\sup_{(t',x')\in\mathrm{Hyp}^{1,d}_{T,R}(t)} \frac{|u(t',x')-u_{T}^{*}(t',x')|}{|u_{T}^{*}(t',x')|} \lesssim (T-t)^{\omega_{0}}
\end{equation*}
for all $0\leq t < T$.
\end{enumerate}
\end{theorem}
\begin{remark}
\Cref{YMglobThm} is a large data result for a nonlinear wave equation in all odd energy-supercritical dimensions $d\geq 7$. It gives full information about the Cauchy evolution for data that evolve near the blowup solution $u_{1}^{*}$ at $t=0$ and shows that solutions remain smooth also outside the backward light cone of $(T,0)$ up to all times $t<T$. This result was known only inside the backward light cone. In this novel hyperboloidal formulation we are even able to go beyond the blowup time $t=T$. Events that are causally separated from the blowup $(T,0)$ do not form singularities, a result that could not be ruled out by the results of \cite{MR4469070}, \cite{MR3278903}. The module $b\in(0,1)$ allows to approach the future of the blowup arbitrarily close.
\end{remark}
\begin{remark}
The blowup occurs generically via $u_{T}^{*}$ in the full range of Sobolev spaces for large equivariant data that evolve \emph{closely} from $u_{1}^{*}$. This improves our knowledge beyond the topology used in \cite{MR3278903}.
\end{remark}
\begin{remark}
The size $\varepsilon>0$ of the support of the initial data is dictated by the local existence time of the above Cauchy problem and imposed for a geometric reason. Namely, our solution to the Cauchy problem is constructed in a hyperboloidal formulation of the equation which requires initial data on some initial hyperboloid. Local existence theory yields a local solution in a truncated light cone. By finite speed of propagation, this local solution is patched together with the blowup solution $u_{1}^{*}$ outside the domain of influence of the support of initial data, see \Cref{Fig_LocExis}. Now, the important observation is that the resulting local solution is defined in a region whose shape allows to fit an initial hyperboloid $\eta_{T}(s_{0},\RR^d)$ inside, along which we obtain initial data for the hyperboloidal evolution.
\end{remark}
\begin{remark}
The comparatively high degree of regularity $k+\frac{d+1}{2}$, where $k\geq \frac{d+1}{2}$, is owed to an algebra property of Sobolev spaces that is used to control the nonlinearity, and Sobolev embedding that is used in the construction of initial data on the hyperboloid $\eta_{T}(s_{0},\BB^{d}_{R})$.
\end{remark}
\begin{remark}
The factors $\ee^{-2s}$ appear naturally and reflect the fact that $\| (u_{1}^{*}\circ\eta_{T})(s,\,.\,) \|_{H^{k}(\BB^{d}_{R})} \simeq \ee^{2s}$ for all $s$.
\end{remark}
\begin{remark}
Our result does not depend on the specific choice of the \emph{height function} $h$. In fact, we only require some elementary conditions such as radiality and monotonicity. However, our choice leads to convenient descriptions of the domains.
\end{remark}
\subsection{Notation and coordinate systems}
The sets of natural numbers, real numbers, complex numbers are denoted by $\NN$, $\RR$, $\CC$, respectively. Open balls of radius $R>0$ centered about $x\in\RR^d$ are denoted by $\BB^{d}_{R}(x)$. If $x = 0$ we simply put $\BB^{d}_{R} \coloneqq \BB^{d}_{R}(0)$ and $\BB_{R} \coloneqq \BB_{R}^1$. Elements associated to two-component product spaces are preferably denoted by boldface symbols $\mathbf{f}$ and their components are usually denoted by $f_{1}$, $f_2$ then.
\subsubsection{Inequalities}
In order to deal with constants in inequalities we define on $\RR_{>0}$ the relation $x \lesssim y$ if there exists a constant $C>0$ such that $x \leq Cy$. As usual, $x\gtrsim y$ if $y\lesssim x$, as well as $x\simeq y$ if both $x\lesssim y$ and $x\gtrsim y$. We say $x_{\lambda} \lesssim y_{\lambda}$ for all $\lambda$, if there is a constant $C>0$, independent of $\lambda$, such that the inequality $x_{\lambda} \leq C y_{\lambda}$ holds uniformly for all $\lambda$.
\subsubsection{Derivatives}
The partial derivative with respect to the $j$-th slot of a function $f$ is denoted by $\pd_{j} f$. We use the common multi-index notation for higher partial derivatives, i.e. when $\alpha = (\alpha_{1},\ldots,a_d)\in\NN_{0}^d$ is a multi-index of length $|\alpha| := \alpha_{1}+\ldots+a_d$, then $\pd^{\alpha} f := \pd^{\alpha_{1}}_{1}\ldots\pd^{\alpha_{d}}_{d}f$. We find it convenient to denote the gradient, or in general the Jacobi matrix, of a function $f$ by $\pd f$. When dealing with functions $u$ in Minkowski space $\RR^{1,d}$, we denote the partial derivative with respect to the slot of the time variable by $\pd_{0} u$, e.g. $(\pd_{0}u)(t,x) = \pd_{t}u(t,x)$ or $(\pd_{0}v)(s,y) = \pd_{s}v(s,y)$.
\subsubsection{Function spaces}
We say $f\in C^\infty(\RR^d)$ is radial if there exists an even function $\widehat{f}\in C^\infty(\RR)$ such that $f(x) = \widehat{f}(|x|)$ for all $x\in\RR^d$. On a ball we consider the space $C^\infty_\rad(\overline{\BB^{d}_{R}})$ of radial smooth functions whose derivatives are continuous up to the boundary. We denote by $\mathcal{S}(\RR^d)$ the Schwartz space on $\RR^d$ and introduce conveniently the Fourier transform by
\begin{equation}
\label{FourierTransform}
\FT f(\xi) = \int_{\RR^d} \ee^{-\ii \xi\cdot x} f(x) \dd x
\end{equation}
for all $f\in \mathcal{S}(\RR^d)$. If $k\in\NN_{0}$ and $\Omega\subseteq\RR^d$ is open and bounded we define for $f\in C^\infty(\overline{\Omega})$ the classical Sobolev norm $\| \,.\, \|_{H^{k}(\Omega)}$ and homogeneous Sobolev norm $\| \,.\, \|_{\dot{H}^{k}(\Omega)}$ by
\begin{equation*}
\| f \|_{H^{k}(\Omega)} = \sum_{|\alpha| \leq k } \| \pd^\alpha f \|_{L^{2}(\Omega)} \,, \qquad
\| f \|_{\dot{H}^{k}(\Omega)} = \sum_{|\alpha| = k } \| \pd^\alpha f \|_{L^{2}(\Omega)} \,,
\end{equation*}
respectively. The Sobolev norms on $\RR^d$ may be equivalently introduced with the Fourier transform via
\begin{equation*}
\| f \|_{H^s(\RR^d)} = \big\| \langle\,.\,\rangle^s \FT f \big\|_{L^{2}(\RR^d)}\,,  \qquad \| f \|_{\dot{H}^s(\RR^d)} = \big\| |\,.\,|^s \FT f \big\|_{L^{2}(\RR^d)}  \,,
\end{equation*}
where $s\geq 0$ and $\langle x\rangle = \sqrt{1+|x|^{2}}$ denotes the \emph{Japanese bracket}. The corresponding Sobolev spaces $H^{k}(\Omega)$ are the completions of $C^\infty(\overline{\Omega})$ with respect to $\| \,.\, \|_{H^{k}(\Omega)}$. Moreover, we have the closed radial subspace $H^{k}_\mathrm{rad}(\BB^{d}_{R})$ that is the completion of $C^\infty_\rad(\overline{\BB^{d}_{R}})$ in the Sobolev norm. On intervals $\BB_{R}\subseteq \RR$ that are symmetric around the origin we will frequently consider the space $C^\infty_\odd(\overline{\BB_{R}})$ of odd functions. The completion of this space with respect to the Sobolev norm $\| \,.\, \|_{H^{k}(\BB_{R})}$ is given by $H^{k}_\mathrm{odd}(\BB_{R})$.
\subsubsection{Operators}
We denote by $L: \dom(L) \subset X \rightarrow Y$ a linear operator $L$, bounded or unbounded, defined on a domain $\dom(L)\subset X$ between Banach spaces $X,Y$. The space of bounded linear operators from $X$ to $Y$ is denoted by $\mathfrak{L}(X,Y)$ with $\mathfrak{L}(X) = \mathfrak{L}(X,X)$. For a closed operator $L$ on a Banach space $X$ the resolvent set is given by $\varrho(L) = \{ z\in\CC \mid z\I - L : \dom(L)\subset X \rightarrow X \text{ is a bijection}  \}$ and the spectrum is given by $\sigma(L) = \CC\setminus\varrho(L)$. By the closed graph theorem we have $(z\I - L)^{-1} \in \mathfrak{L}(X)$ for all $z\in\varrho(L)$ and we define the resolvent operator $R_L: \varrho(L) \rightarrow \mathfrak{L}(X)$ by $R_L(z) = (z\I - L)^{-1}$. We have further the point spectrum $\sigma_\mathrm{p}(L) = \{ \lambda\in\sigma(L) \mid \ker( \lambda\I - L) \neq \{ 0 \} \}$. For more details we refer to \cite{MR3243083}, \cite{teschl1998topicsFA}. Consider the case where $M,N\subset X$ are two closed complementary subspaces, i.e. $X=M\oplus N$. In this situation there corresponds a projection onto $M$ along $N$ which is the bounded linear operator $P\in\mathfrak{L}(X)$ defined by $Px = x_M$ for the unique decomposition $x=x_M+x_N$ with $x_M\in M$, $x_N\in N$ for all $x\in X$. If $L: \dom(L)\subset X \rightarrow X$ is an operator such that for all $x\in \dom(L)$ we have $Px\in \dom(L)$ and $LPx\in M$, $L(\I-P)x\in N$, then we can define the part of $L$ in $M$ as the operator $L\restriction_M : \dom(L\restriction_M) \subset M \rightarrow M$ with domain $\dom(L\restriction_M) = \dom(L)\cap M$ by $L\restriction_M x = Lx \in M$. Note that the hypothesis is equivalent to $LP$ extending $PL$, i.e. $PL \subset LP$, see \cite[p. 171 f.]{MR1335452}.
\subsubsection{Coordinate systems}
Self-similar blowup in finite time $T$ is conveniently studied in coordinates that are compatible with self-similarity, meaning that the fraction $r/(T-t)$ is independent of the new time variable in such coordinates. One such novel coordinate system are so-called \emph{hyperboloidal similarity coordinates} $(s,y)$ which have been introduced in \cite{MR4338226} as a crucial ingredient. They are defined through a change of coordinates
\begin{equation*}
(t,x) = \big( T + \ee^{-s}h(y), \ee^{-s}y \big) \,.
\end{equation*}
Their properties are extensively studied in \Cref{geomHSC}. These coordinates are capable of covering the whole exterior region of the future light cone of the blowup event $(T,0)$ and thus make it possible to study the stability of blowup solutions on the whole space, rather than only in the backward light cone.
Unless $h$ is specified, hyperboloidal similarity coordinates can be viewed as a general family of coordinate systems. Note, for example, for the choice $h = -1$ we retrieve \emph{standard similarity coordinates} $(\tau,\xi)$ via
\begin{equation*}
(t,x) = (T - \ee^{-\tau},\ee^{-\tau} \xi) \,.
\end{equation*}
They only cover the portion $\{(t,x)\in\RR^{1,d} \mid t<T\}$ of Minkowski space-time and were previously applied to study stable blowup in the Yang-Mills equation \cite{MR4469070}, \cite{MR3278903}, also see other earlier works of the first author.
\subsection{Strategy of the proof}
One of the keys in our proof is to get a solid understanding of the wave evolution in hyperboloidal similarity coordinates. After a change from Cartesian coordinates $(t,x)$ to hyperboloidal similarity coordinates $(s,y)$ via
\begin{equation*}
(t,x) = \eta_{T}(s,y) = \big( T + \ee^{-s} h(y) , \ee^{-s}y \big) \,, \qquad h(y) = \sqrt{2+|y|^{2}} - 2 \,,
\end{equation*}
and a linearization around the blowup solution $u_{1}^{*}$ we consider our Yang-Mills equation \eqref{YMwave} in terms of the variable for the rescaled field
\begin{equation*}
\Phi_{1}(s) = \ee^{-2s}((u-u_{T}^{*})\circ\eta_{T})(s,\,.\,) \,, \quad \Phi_2(s) = \ee^{-2s} \pd_s ((u-u_{T}^{*})\circ\eta_{T})(s,\,.\,) \,,
\end{equation*}
and arrive at an autonomous first order hyperboloidal formulation
\begin{equation*}
\pd_s \mathbf{\Phi}(s) = ( \mathbf{L}_d - 2 \mathbf{I} + \mathbf{L}_V ) \mathbf{\Phi}(s) + \mathbf{N}(\mathbf{\Phi}(s)) \eqqcolon
\mathbf{L} \mathbf{\Phi}(s) + \mathbf{N}(\mathbf{\Phi}(s))
\end{equation*}
on the radial Sobolev space
\begin{equation*}
\HH^{k}_\rad(\BB^{d}_{R}) = H^{k}_\rad(\BB^{d}_{R}) \times H^{k-1}_\rad(\BB^{d}_{R}) \,, \qquad k \geq \frac{d-1}{2} \,.
\end{equation*}
Here, $\mathbf{L}_d$ corresponds to the \emph{free} wave evolution, $\mathbf{L}_V$ represents a multiplicative potential term arising from the linearization and $\mathbf{N}$ is the nonlinearity. We refer to $\mathbf{L} \coloneqq \mathbf{L}_d - 2 \mathbf{I} + \mathbf{L}_V$ as the \emph{linearized} wave evolution and to the full equation as the \emph{nonlinear} wave evolution and regard it as a perturbed linear system. In this setting, the aim of the bulk of this paper is launching a perturbation theory for the equivariant Yang-Mills equation around the blowup $u_{1}^{*}$. This will be achieved in a functional analytic setting that employs semigroup methods, non-selfadjoint spectral theory and ideas from infinite-dimensional dynamical systems, similarly to the route in the recent work \cite{MR4338226}. For this approach it is essential to control the linear part of this equation, which leads us to the need of a thorough understanding of the free wave evolution in hyperboloidal similarity coordinates. In \Cref{MasterSection} we present a systematic method to control the free radial wave evolution in hyperboloidal similarity coordinates throughout all odd space dimensions.
\begin{itemize}[itemsep = 1em, topsep = 1em]
\item\textbf{\emph{Control of the \emph{free} wave evolution.}} We focus on the free wave equation
\begin{equation*}
\big( - \pd_{t}^{2} + \Delta_{x} \big) u(t,x) = 0
\end{equation*}
on $\RR^{1,d}$ and assume the radially symmetric ansatz $u(t,x) = \widehat{u}(t,|x|)$. In terms of hyperboloidal similarity coordinates we can formulate it as a first order system
\begin{equation*}
\pd_s \mathbf{v}(s, \,.\,) = \mathbf{L}_d \mathbf{v}(s,\,.\,)
\end{equation*}
for the variable $\mathbf{v}(s, \,.\,) = [\, v(s,\,.\,), \pd_s v(s,\,.\,) \,]$ with $v = u \circ \eta_{T}$ and the free radial wave evolution operator $\mathbf{L}_d$ which is a spatial differential operator.
\begin{itemize}[itemsep=1em, topsep=1em, leftmargin=0pt]
\item\emph{Energy norms.} The first and foremost issue is to come up with a suitable norm that allows us to talk about stability concepts. Such a norm is required to incorporate enough regularity and exhibit good growth estimates for the free wave evolution. We handle this by adapting previous strategies on a new ground, see \cite{MR2881965}, \cite{MR2839272}, \cite{MR3278903}, \cite{MR3152726}, \cite{MR3218816}, \cite{MR3537340}, \cite{MR3742520}, \cite{MR4338226} and start reviewing the situation in dimension $d=1$ from \cite{MR4338226}. There, the closest such quantity that we can start with is an $\dot{H}^1(\BB_{R}) \times L^{2}(\BB_{R})$-energy for $\mathbf{v}(s,\,.\,)$ derived from the free wave equation. This rests on the observation that the one-dimensional free wave equation factorizes into transport equations for the half-waves $v_\pm = u_\pm\circ\eta_{T}$ where $u_\pm(t,x) = (\pd_{t} \pm \pd_{x})u(t,x)$. Moreover, the subtle fact that there exist vector fields $D_\pm$ that commute with the half-wave equation satisfied by $v_\pm$ allow us to upgrade the energy estimates to $H^{k}(\BB_{R})\times H^{k-1}(\BB_{R})$. Based on them we construct Hilbert spaces for our analysis.
\item\emph{Half-wave evolution.} By studying the half-wave equations combined into a first order system $\pd_s \mathbf{v}_\mp(s) = \mathbf{L}_\mp \mathbf{v}_\mp(s)$ via $\mathbf{v}_\mp = [\, v_-, v_+\,]$ on the previously defined spaces, we infer from the Lumer-Philips Theorem a strongly continuous operator semigroup $\mathbf{S}_\mp$ that propagates the solution for the half-wave variable.
\item\emph{One-dimensional wave evolution.}
The transition between the variables $\mathbf{v}_\mp$, $\mathbf{v}$ is governed by operators $\mathbf{A}$, $\mathbf{A}^{\!\!\times}$ that extend to mutually inverse bounded linear operators $\mathbf{A}_\mp$, $\mathbf{A}_\mp^{-1}$. Now we are ready to solve the abstract Cauchy problem for the free one-dimensional wave equation on the half-line in hyperboloidal similarity coordinates with the rescaled similar operator semigroup $\mathbf{S}_{1}(s) = \ee^{-s} \mathbf{A}_\mp^{-1} \mathbf{S}_\mp(s) \mathbf{A}_\mp$. This operator semigroup decays sharply like $\ee^{-s/2}$ in the full range of Sobolev norms.
\item\emph{The descent method.} A strategy that has turned out successful for solving the free wave equation in higher dimensions works via the reduction to the one-dimensional wave equation. The underlying transformations manifest themselves as commutation relations for the radial Laplace operator $\Delta_{x} u(t,x) = \pd_{r}^{2} \widehat{u} (t,r) + \frac{d-1}{r} \pd_{r} \widehat{u}(t,r)$ and are well-known, provided that the underlying space dimension $d$ is odd, see e.g. \cite{MR2597943}. More concretely, the radial ansatz in the wave equation from above implies that $\widehat{u}$ solves the free radial wave equation
\begin{equation*}
\big( - \pd_{t}^{2} + \pd_{r}^{2} + \tfrac{d-1}{r} \pd_{r} \big) \widehat{u}(t,r) = 0
\end{equation*}
and we get by virtue of
\begin{equation*}
\widehat{u}_d(t,r) \coloneqq \big( r^{-1} \pd_{r} \big)^{\frac{d-3}{2}} \big( r^{d-2} \widehat{u}(t,r) \big)
\end{equation*}
a solution to the one-dimensional wave equation $\big( -\pd_{t}^{2} + \pd_{r}^{2} \big) \widehat{u}_d(t,r) = 0$ on the half-line. However, if this is transformed directly to hyperboloidal similarity coordinates, the resulting expressions are inaccessibly complicated and make a systematic analysis impossible. The key observation in overcoming this technical difficulty is that $\widehat{u}_d$ arises from repeated applications of transformations of the type
\begin{equation*}
\widehat{u}^\downarrow_d(t,r) \coloneqq r^{-(d-3)} \pd_{r} \big( r^{d-2} \widehat{u}(t,r) \big)
\end{equation*}
that map solutions of the radial wave equation in $d$ dimensions to corresponding solutions in $(d-2)$ dimensions. We will refer to this as the \emph{descent method} and study it thoroughly in hyperboloidal similarity coordinates. We show that this effect is visible as a \emph{descent operator} $\mathbf{D}_d$, which is a boundedly invertible spatial differential operator that satisfies the intertwining identity $\mathbf{D}_d \mathbf{L}_d - \mathbf{D}_d = \mathbf{L}_{1} \mathbf{D}_d$ in a rigorous sense.
\item\emph{Radial wave evolution in odd space dimensions.} Equipped with this, the further idea is to descend the wave evolution $\mathbf{L}_d$ via $\mathbf{D}_d$ to the one-dimensional wave evolution $\mathbf{L}_{1}$, evolve the solution and transport it back with $\mathbf{D}_d^{-1}$ to the original dimension. The mapping properties of the descent operator in combination with the intertwining identity allow us to provide a solution to the wave equation in hyperboloidal similarity coordinates in terms of the operator semigroup $\mathbf{S}_d(s) = \ee^s \mathbf{D}_d^{-1} \mathbf{S}_{1}(s) \mathbf{D}_d$. Combined with the decay estimates in one dimension, this leads to a systematic control of the free radial wave evolution in all odd space dimensions. This is presented as our first main result in \Cref{MScresult} which paves the way for a systematic study of blowup in nonlinear wave equations in the whole space.
\end{itemize}
\item\textbf{\emph{Control of the \emph{linearized} wave evolution.}} The next ingredient in the stability programme is the development of a well-posedness theory for the linearized wave evolution.
\begin{itemize}[leftmargin=0pt, itemsep=1em, topsep=1em]
\item\emph{Spectral analysis of the linearized evolution.} Starting off from our result about the free wave evolution, the Bounded Perturbation Theorem shows that $\mathbf{L}$ is the generator of a strongly continuous semigroup $\mathbf{S}(s)$. In order to prove stability of the blowup solution we need a decay estimate for this operator semigroup. For this, we have to acquire information about the spectrum of its generator $\mathbf{L}$. Semigroup methods allow us to prove that the spectrum is contained in a left-half plane except for finitely many eigenvalues.
\item\emph{The mode stability problem.}
However, locating those unstable eigenvalues confronts us with a highly nontrivial spectral problem, which is equivalent to the so-called \emph{mode stability} of $u_{T}^{*}$. In the setting of standard similarity coordinates, this problem arose in a numerical study by P. Bizo\'{n} and T. Chmaj \cite{bizon2005convergence} for the profile \eqref{BizonBlowup}. It took time to devise rigorous methods for such non-selfadjoint spectral problems until the paper \cite{MR3475668} by O. Costin, R. Donninger, I. Glogi\'{c} and M. Huang appeared. Recently, those techniques have been improved by I. Glogi\'{c} \cite{MR4469070} who thereby solved the mode stability problem for \eqref{BiernatBizonBlowup} in all higher space dimensions. Interestingly, the structural properties of hyperboloidal similarity coordinates let us apply this information and conclude that the only unstable eigenvalue of $\mathbf{L}$ is the isolated eigenvalue $1\in\sigma(\mathbf{L})$ to a single symmetry mode $\mathbf{f}_{1}^{*}$ that arises from time-translation invariance of the Yang-Mills equation.
\item\emph{Co-dimension one stability.} We show how the symmetry mode $\mathbf{f}_{1}^{*}$ is the only source for instabilities in the time evolution by employing the Riesz projection $\mathbf{P}$ to the isolated eigenvalue $1\in\sigma(\mathbf{L})$. The Riesz projection decomposes the space $\HH^{k}_\rad(\BB^d_{R})$ into the direct sum of the $\mathbf{L}$-invariant subspaces $\mathfrak{M}=\ran(\mathbf{P})$ and $\mathfrak{N} = \ran(\mathbf{I}-\mathbf{P})$ which split the linearized wave evolution into parts with the properties $\sigma(\mathbf{L}\restriction_\mathfrak{M}) = \{1\}$ and $\sigma(\mathbf{L}\restriction_\mathfrak{N}) = \sigma(\mathbf{L}) \setminus \{1\}$. We prove that $\mathfrak{M}$ is an unstable subspace, i.e. it is spanned by the symmetry mode $\mathbf{f}_{1}^{*}$ and the semigroup grows exponentially like $\mathbf{S}(s)\mathbf{P} = \ee^s \mathbf{P}$ on it. Once again, the underlying analysis can be traced back to standard similarity coordinates \cite{MR4469070} by exploiting curious coordinate artefact effects. After establishing resolvent estimates we apply the Gearhart-Pr\"uss Theorem and obtain that the semigroup $\mathbf{S}(s)( \mathbf{I} - \mathbf{P})$ decays like $\ee^{-\omega_{0} s}$ for some $\omega_{0}>0$. So, this fully attributes the instability to the space $\mathfrak{M}$ and anticipates exponential decay for data evolving from the complementary subspace $\mathfrak{N}$.
\end{itemize}
\item\textbf{\emph{Control of the \emph{nonlinear} wave evolution.}} Now, we are ready to carry this over to the nonlinear problem by employing Duhamel's formula
\begin{equation*}
\mathbf{\Phi}(s) = \mathbf{S}(s-s_{0})\mathbf{f} + \int_{s_{0}}^s \mathbf{S}(s-s')\mathbf{N}(\mathbf{\Phi}(s')) \dd s'
\end{equation*}
with initial data $\mathbf{f} = \mathbf{\Phi}(s_{0})$.
\begin{itemize}[leftmargin=0pt, itemsep=1em, topsep=1em]
\item\emph{Control of the nonlinearity.} The nonlinearity $\mathbf{N}$ is essentially a polynomial which is controlled by a local Lipschitz bound which is obtained through an algebra property of the underlying Sobolev spaces. The algebra property is available because our norms incorporate enough regularity.
\item\emph{Reformulation as a modified fixed point problem.} 
We consider a modified right-hand side $\mathbf{K}_\mathbf{f}$ in the Duhamel formula, with the initial data $\mathbf{f}$ replaced by $\mathbf{f} - \mathbf{C}_{s_{0}}(\mathbf{\Phi},\mathbf{f})$, where $\mathbf{C}_{s_{0}}(\mathbf{\Phi},\mathbf{f})$ is a correction term. This correction term is chosen in a way so that it suppresses the instability. The effective fixed point equation $\mathbf{\Phi} = \mathbf{K}_\mathbf{f}(\mathbf{\Phi})$ is solved by Banach's fixed point theorem and the solution $\mathbf{\Phi}(s)$ exhibits the linear decay rate $\ee^{-\omega_{0} s}$.
\item\emph{Preparation of initial data.} Before we can actually solve the Yang-Mills equation in the hyperboloidal formulation we need to prepare initial data for the hyperboloidal evolution. By local existence for the classical Cauchy evolution we can evolve a local solution from the data $f,g$ at $t=0$ and piece it together with finite speed of propagation so that it fits on some initial hyperboloid $\eta_{T}(s_{0},\RR^d)$. We evaluate the local solution on this hyperboloid and obtain initial data $\mathbf{U}(f,g,T)$ for the hyperboloidal evolution. This evaluation process is well-defined by Sobolev embedding.
\item\emph{Elimination of the correction term by adjusting the blowup time.} Note that the blowup time $T$ will only enter through the initial data operator $\mathbf{U}(f,g,T)$ in the hyperboloidal formulation of the problem. This allows us to remove the above introduced correction term. Indeed, we consider the solution $\mathbf{\Phi}_{\mathbf{U}(f,g,T)}$ that evolves from the prepared initial data $\mathbf{U}(f,g,T)$ and by adjusting a time $T$ close around $1$ we show with Brouwer's fixed point theorem that the correction term $\mathbf{C}_{s_{0}}(\mathbf{\Phi}_{\mathbf{U}(f,g,T)},\mathbf{U}(f,g,T))$ actually vanishes. Finally, this yields a solution to the full nonlinear problem and tracing its properties back culminates in the proof of \Cref{YMglobThm}.
\end{itemize}
\end{itemize}
\section{The wave equation in hyperboloidal similarity coordinates}
\label{MasterSection}
The goal of this section is to study the wave equation in hyperboloidal similarity coordinates under radial symmetry and, most notably, establish good growth estimates for the free wave evolution. In view of the perturbative non-selfadjoint character of the blowup analysis in the second part of this paper, this is done within the picture of strongly continuous semigroups. To begin with, we transform the wave equation from Cartesian coordinates to hyperboloidal similarity coordinates by considering the Laplace-Beltrami operator associated to the Minkowski metric. Geometric quantities such as the \hyperref[JacobianHSC]{Jacobian}, \hyperref[metHSC]{Minkowski metric} or \hyperref[ChristoffelHSC]{Christoffel symbols} are readily computed and provided in \Cref{geomHSC}. Suppose $u,v\in C^\infty(\RR^{1,d})$ are related via $v = u\circ\eta_{T}$. The wave operator for the Minkowski metric $\minkmet$ in Cartesian coordinates
\begin{equation*}
\Box_{\minkmet} u \coloneqq \frac{1}{\metdet{\minkmet}}\pd_{\mu}\left(\metdet{\minkmet}\minkmet^{\mu\nu}\pd_{\nu} u \right) = \left( -\pd_{0}^{2} + \pd^i\pd_{i} \right) u
\end{equation*}
transforms in hyperboloidal similarity coordinates to the Laplace-Beltrami operator $\left( \Box_{\minkmet} u \right) \circ \eta_{T} = \Box_{\g} v $, which is given by
\begin{equation*}
\Box_{\g} v \coloneqq \frac{1}{\metdet{\g}}\pd_{\mu}\left(\metdet{\g}\g^{\mu\nu}\pd_{\nu} v\right) = \g^{\mu\nu} \pd_{\mu} \pd_{\nu} v - \g^{\mu\nu} \Gamma^{\lambda}{}_{\mu\nu} \pd_{\lambda} v \,.
\end{equation*}
Explicitly we have
\begin{equation}
\label{secondorderwave}
\Box_{\g} v = \g^{00} \left( \pd_{0}^{2} - c^{ij} \pd_{i}\pd_j - c^{i0} \pd_{i}\pd_{0} - c^0 \pd_{0} - c^i \pd_{i} \right) v
\end{equation}
with coefficients
\begin{align*}
\g^{00}(s,y) &= - \ee^{2s} \frac{1 - \pd_{y^{k}}h(y)\pd_{y_k}h(y)}{\big( y^{k}\pd_{y^{k}}h(y) - h(y) \big)^{2}} \,, \\
c^{ij}(y) &= - \frac{ y^{k}\pd_{y^{k}}h(y) - h(y) }{1 - \pd_{y^{k}}h(y)\pd_{y_k}h(y)} \big( y^i \pd_{y_j} h(y) + y^{j} \pd_{y_{i}} h(y) - \big( y^{k}\pd_{y^{k}}h(y) - h(y) \big) \delta^{ij} \big) - y^iy^{j} \,, \\
c^{i0}(y) &= -2y^i - 2 \frac{ y^{k} \pd_{y^{k}} h(y) - h(y) }{1 - \pd_{y^{k}}h(y)\pd_{y_k}h(y)} \pd_{y_{i}} h(y) \,, \\
c^0(y) &= -1 - \frac{y^{k} \pd_{y^{k}}h(y) - h(y)}{1 - \pd_{y^{k}}h(y)\pd_{y_k}h(y)} \pd_{y^{k}} \pd_{y_k} h(y) + \frac{y^iy^{j}}{y^{k} \pd_{y^{k}}h(y) - h(y)}\pd_{y^i}\pd_{y^{j}} h(y) \\&\quad\,\,+
2\frac{y^i\pd_{y_j}h(y)}{1 - \pd_{y^{k}}h(y)\pd_{y_k}h(y)} \pd_{y^i}\pd_{y^{j}} h(y) \,, \\
c^i(y) &= c^{i0}(y) + (c^0(y)+1)y^i \,.
\end{align*}
In this paper, we are going to study the wave evolution as an operator defined on appropriate subspaces of the following product space.
\begin{definition}
Let $R>0$ and $d,k\in\NN$. We define the two-component Sobolev spaces
\begin{equation*}
\HH^{k}(\BB^{d}_{R}) = H^{k}(\BB^{d}_{R}) \times H^{k-1}(\BB^{d}_{R})
\end{equation*}
equipped with the product structure
\begin{equation*}
\| \mathbf{f} \|_{\HH^{k}(\BB^{d}_{R})} = \| f_{1} \|_{H^{k}(\BB^{d}_{R})} + \| f_2 \|_{H^{k-1}(\BB^{d}_{R})} \,.
\end{equation*}
\end{definition}
\subsection{Energy estimates}
Before we come to the wave equation in higher dimensions it is essential to understand the one-dimensional free wave evolution. It is worth to get started with a review of results about it from \cite{MR4338226}. The classical route for solving the wave equation begins with the observation that the wave equation in one space dimension
\begin{equation}
\label{freeWaveEq1}
\big( - \pd_{t}^{2} + \pd_{x}^{2} \big) u(t,x) = 0
\end{equation}
factorizes into a transport equation
\begin{equation*}
\big( \pd_{t} \mp \pd_{x} \big) u_\pm(t,x) = 0
\end{equation*}
for the half-waves
\begin{equation*}
u_\pm(t,x) \coloneqq \big( \pd_{t} \pm  \pd_{x} \big) u(t,x) \,.
\end{equation*}
The half-waves in hyperboloidal similarity coordinates are simply defined by
\begin{equation*}
v_\pm(s,y) \coloneqq (u_\pm\circ\eta_{T})(s,y) \,,
\end{equation*}
and by invoking the transformations \eqref{PDHSC0}, \eqref{PDHSCj} for the partial derivatives, the half-wave equation for $v_\pm$ reads
\begin{equation}
\label{half-wave-HSC}
(1\pm h'(y)) \pd_s v_\pm(s,y) + (y\pm h(y)) \pd_y v_\pm(s,y) = 0 \,.
\end{equation}
Testing with $\ee^{-s} v_\pm(s,y)$ yields the differential energy conservation
\begin{equation*}
\pd_s \big( \ee^{-s} v_\pm(s,y)^{2} (1\pm h'(y)) \big) + \pd_y \big( \ee^{-s} v_\pm(s,y)^{2} (y\pm h(y)) \big) = 0 \,,
\end{equation*}
which implies upon integration over $\BB_{R}$
\begin{equation}
\label{GronwallHSC}
\frac{\od}{\od s} \Big( \ee^{-s} \int_{-R}^R v_\pm(s,y)^{2} (1\pm h'(y)) \dd y \Big) = \Big[ - \ee^{-s} v_\pm(s,y)^{2} (y\pm h(y)) \Big]_{-R}^R \leq 0 \,,
\end{equation}
provided that $R\geq \frac{1}{2}$. This motivates to introduce the following Hilbert space structure.
\begin{definition}
\label{Hilberthpm}
Let $R>0$ and set
\begin{equation}
\label{hpm}
h_\pm: \overline{\BB_{R}} \rightarrow \RR \,, \qquad h_\pm(y) \coloneqq y \pm h(y) \,.
\end{equation}
We define inner products
\begin{equation*}
\big( f_\pm \,\big|\, g_\pm \big)_{L^{2}(\BB_{R},\,h_\pm')} = \int_{-R}^R f_\pm(y) \overline{g_\pm(y)}(1\pm h'(y)) \dd y
\end{equation*}
with induced norms $\| f_\pm \|_{L^{2}(\BB_{R},\,h_\pm')}$, for all $f_\pm,g_\pm\in C^\infty(\overline{\BB_{R}})$, respectively.
\end{definition}
After integrating \Cref{GronwallHSC} we get an energy bound for the half-waves, that is
\begin{equation}
\label{half-wave-energy}
\| v_\pm(s,\,.\,) \|_{L^{2}(\BB_{R},\,h_\pm')} \lesssim \ee^{s/2} \| v_\pm(s_{0},\,.\,) \|_{L^{2}(\BB_{R},\,h_\pm')}
\end{equation}
for all $s\geq s_{0}$ and all $R\geq\frac{1}{2}$. In order to control higher derivatives with this norm via energy conservation, we shall show that derivatives of solutions satisfy the same half-wave equation. This is indeed the case, due to the special structure of the transport equation satisfied by the half-waves. Note that $1\pm h'(y) \simeq 1$ for all $y\in \overline{\BB_{R}}$, so \Cref{half-wave-HSC} is equivalent to the evolution equation
\begin{equation}
\label{half-wave-HSC-timeEvo}
\pd_s v_\pm(s,y) = - \frac{y \pm h(y)}{1 \pm h'(y)} \pd_y v_\pm(s,y) \,,
\end{equation}
and we get for free that $-(y\pm h(y))/(1\pm h'(y))\pd_y v_\pm(s,y)$ solves an equation of the same type as \Cref{half-wave-HSC-timeEvo} again. However, the coefficient $y\pm h(y)$ has a zero at $\pm \frac{1}{2}$ and thus will spoil the equivalence to a Sobolev norm. On the other hand, $1\pm h'(y) \simeq 1$ for all $y\in\overline{\BB_{R}}$ and from the form of \Cref{half-wave-HSC} we read off $\ee^{-s}(y\pm h(y))$ as a solution. Both observations combined with the product rule show that $\ee^s/(1\pm h'(y)) \pd_y v_\pm(s,y)$ is another solution with a controlled coefficient. These effects are not inherent to the wave equation but rather a fact for the transport equation that comes from the half-wave equation in hyperboloidal similarity coordinates.
\begin{definition}
Let $R>0$. We define the vector fields
\begin{equation*}
L_\pm f_\pm(y) = - \frac{y\pm h(y)}{1\pm h'(y)} f_\pm'(y) \,, \qquad D_\pm f_\pm(y) = \frac{1}{1\pm h'(y)} f_\pm'(y) \,,
\end{equation*}
for all $f_\pm\in C^\infty(\overline{\BB_{R}})$, respectively.
\end{definition}
Note that the half-wave equations \eqref{half-wave-HSC} take the form
\begin{equation}
\label{half-wave-Lpm}
\pd_s v_\pm(s,\,.\,) = L_\pm v_\pm(s,\,.\,) \,.
\end{equation}
The point of the vector field $D_\pm$ is that it satisfies a crucial commutator relation
\begin{equation}
\label{commutator}
[D_\pm,L_\pm] = -D_\pm
\end{equation}
by the above observation, and by induction we have $[D_\pm^{j},L_\pm] = -jD_\pm^{j}$ for all $j\in\NN$. We will see soon that the commutator provides a decay inducing term in the half-wave equation. What is more, this vector field controls Sobolev norms.
\begin{lemma}
\label{halfwavenorm}
Let $R>0$ and $k\in\NN_{0}$. Then
\begin{equation*}
\sum_{j=0}^{k} \big\| D_\pm^{j} f_\pm \big\|_{L^{2}(\BB_{R},\, h_\pm')} \simeq \| f_\pm \|_{H^{k}(\BB_{R})}
\end{equation*}
for all $f_\pm\in C^\infty(\overline{\BB_{R}})$, respectively.
\end{lemma}
\begin{proof}
In the case $k=0$ note that $1\pm h'(y) \simeq 1$ for all $ y\in \overline{\BB_{R}}$ implies
\begin{equation*}
\|f_\pm\|_{L^{2}(\BB_{R},\,h_\pm')} \simeq \|f_\pm \|_{L^{2}(\BB_{R})}
\end{equation*}
for all $f_\pm\in C^\infty(\overline{\BB_{R}})$, respectively.
\begin{enumerate}[wide, itemsep=1em, topsep=1em]
\item[``$\lesssim$'':] For $1 \leq j\leq k$ we find by the Leibniz rule
\begin{equation*}
D_\pm^{j} f_\pm = \sum_{i=0}^{j-1} {\varphi_\pm}_{j,i} f_\pm^{(i)} + \left(\frac{1}{1\pm h'}\right)^{j} f_\pm^{(j)}
\end{equation*}
for some ${\varphi_\pm}_{j,i}\in C^\infty(\overline{\BB_{R}})$. Hence $\big\| D_\pm^{j} f_\pm \big\|_{L^{2}(\BB_{R},\, h_\pm')} \lesssim \| f_\pm \|_{H^{j}(\BB_{R})} \lesssim \| f_\pm \|_{H^{k}(\BB_{R})}$.
\item[``$\gtrsim$'':] From rearranging the Leibniz rule above we also get
\begin{equation*}
\big\| f_\pm^{(j)} \big\|_{L^{2}(\BB_{R})} \lesssim \sum_{i=0}^{j-1} \big\| f_\pm^{(i)} \big\|_{L^{2}(\RR)} + \big\| D_\pm^{j} f_\pm \big\|_{L^{2}(\BB_{R},\,h_\pm')} \lesssim \| f_\pm \|_{H^{j-1}(\BB_{R})} + \big\| D_\pm^{j} f_\pm \big\|_{L^{2}(\BB_{R},h_\pm')} \,.
\end{equation*}
From here the bound follows by induction.
\qedhere
\end{enumerate}
\end{proof}
An application of the standard energy estimate then yields together with the commutator relation corresponding estimates in higher regularity. The following result about the control of the wave evolution hints at what we can expect for well-posedness.
\begin{lemma}
\label{half-wave-energy-lemma}
Let $R,s_{0}\in\RR$, $k\in\NN$ such that $R\geq\frac{1}{2}$ and $k\geq 2$.
\begin{enumerate}[itemsep=1em, topsep=1em]
\item Suppose $v_\pm\in C^{k-1}([s_{0},\infty)\times\overline{\BB_{R}})$ solve $\pd_s v_\pm (s,\,.\,) = L_\pm v_\pm(s,\,.\,)$, respectively. Then
\begin{equation*}
\| v_\pm(s,\,.\,) \|_{H^{j-1}(\BB_{R})} \lesssim \ee^{s/2} \| v_\pm(s_{0},\,.\,) \|_{H^{j-1}(\BB_{R})}
\end{equation*}
for all $1\leq j \leq k$ and all $s\geq s_{0}$.
\item Suppose $u\in C^{k}(\RR^{1,1})$ solves $\big( - \pd_{t}^{2} + \pd_{x}^{2} \big) u(t,x) = 0$ with $u(t,0) = 0$. Let $v = u \circ \eta_{T}$. Then
\begin{equation*}
\| v(s,\,.\,) \|_{H^{j}(\BB_{R})} + \| \pd_s v(s,\,.\,) \|_{H^{j-1}(\BB_{R})} \lesssim \ee^{-s/2} \big( \| v(s_{0},\,.\,) \|_{H^{j}(\BB_{R})} + \| \pd_{0} v(s_{0},\,.\,) \|_{H^{j-1}(\BB_{R})} \big)
\end{equation*}
for all $1\leq j \leq k$ and all $s\geq s_{0}$.
\end{enumerate}
\end{lemma}
\begin{proof}
\begin{enumerate}[wide, itemsep=1em, topsep=1em]
\item This is the point where we use the commutator relation \eqref{commutator} to see
\begin{equation*}
\pd_s \big( \ee^{js} D_\pm^{j} v_\pm(s,\,.\,) \big) =  L_\pm \big( \ee^{js} D_\pm^{j} v_\pm(s,\,.\,) \big) \,.
\end{equation*}
That is, given that $v_\pm (s,\,.\,)$ solves the half-wave equation \eqref{half-wave-Lpm}, $\ee^{js} D_\pm^{j}v_\pm(s,\,.\,)$ is again a solution to the half-wave equation, for any $1\leq j\leq k-1$, and we can apply the half-wave energy estimate \eqref{half-wave-energy} to get
\begin{align*}
\sum_{i=0}^{j-1} \| D_\pm^i v_\pm (s,\,.\,)\|_{L^{2}(\BB_{R},\,h_\pm')} &\lesssim \sum_{i=0}^{j-1} \ee^{(1/2 - i)s} \| D_\pm^i v_\pm (s_{0},\,.\,)\|_{L^{2}(\BB_{R},\,h_\pm')} \\&\lesssim
\ee^{s/2} \sum_{i=0}^{j-1} \| D_\pm^i v_\pm (s_{0},\,.\,)\|_{L^{2}(\BB_{R},\,h_\pm')}
\end{align*}
for all $1\leq j \leq k$ and all $s\geq s_{0}$. \Cref{halfwavenorm} yields the estimate.
\item Define the half-waves $u_\pm(t,x) = \big( \pd_{t} \pm \pd_{x} \big) u(t,x)$ and let $v_\pm = u_\pm \circ \eta_{T}$. With \eqref{PDHSC0}, \eqref{PDHSCj} we see the transformations
\begin{align}
\label{v-}
v_-(s,y) &= \frac{\ee^s}{y h'(y)-h(y)} \big( (1 + h'(y))\pd_s v(s,y) + (y + h(y)) \pd_y v(s,y) \big) \,, \\
\label{v+}
v_+(s,y) &= \frac{\ee^s}{y h'(y) - h(y)} \big( (1 - h'(y))\pd_s v(s,y) + (y - h(y)) \pd_y v(s,y) \big) \,,
\end{align}
and
\begin{align}
\label{pdsv}
\pd_s v(s,y) &= \frac{\ee^{-s}}{2} \big( (y-h(y))v_-(s,y) - (y+h(y))v_+(s,y) \big)  \,, \\
\label{pdyv}
\pd_y v(s,y) &=  \frac{\ee^{-s}}{2} \big( -(1-h'(y))v_-(s,y) + (1+h'(y))v_+(s,y) \big) \,,
\end{align}
where the occurring coefficients are smooth with all their derivatives bounded on $\overline{\BB_{R}}$. Hence
\begin{equation*}
\| \pd_s v(s,\,.\,) \|_{H^{j-1}(\BB_{R})} + \| \pd v(s,\,.\,) \|_{H^{j-1}(\BB_{R})} \simeq \ee^{-s} \big( \| v_-(s,\,.\,) \|_{H^{j-1}(\BB_{R})} + \| v_+(s,\,.\,) \|_{H^{j-1}(\BB_{R})} \big) \,.
\end{equation*}
The boundary condition $u(t,0) = 0$ implies $v(s,0) = 0$, so with Cauchy-Schwarz
\begin{equation*}
|v(s,y)| = \left| \int_{0}^y \pd_{y'} v(s,y') \dd y' \right| \lesssim \| \pd v(s,\,.\,) \|_{L^{2}(\BB_{R})} \,,
\end{equation*}
for all $y\in\overline{\BB_{R}}$, hence $\| v(s,\,.\,) \|_{L^{2}(\BB_{R})} \lesssim \| \pd v(s,\,.\,) \|_{L^{2}(\BB_{R})}$. Now $v_\pm$ satisfy the assumption of the first part and we conclude the asserted higher regularity energy estimate.
\qedhere
\end{enumerate}
\end{proof}
We complete this discussion with a presentation of a standard energy estimate in higher dimensions.
\begin{remark}
Incidentally, by considering linear combinations of the \emph{energy}- and \emph{momentum conservation}, one arrives similarly at a standard energy estimate for the radial wave equation
\begin{equation*}
-\pd_{t}^{2} \widehat{u}(t,r) + r^{-(d-1)} \pd_{r} \big( r^{d-1}\pd_{r}\widehat{u}(t,r) \big) = 0
\end{equation*}
in dimensions $d>1$, which follows from inserting the radiality condition $u(t,\,.\,) = \widehat{u}(t,|\,.\,|)$ into the free wave equation, see \Cref{SecRadWave} below. We start with observations in Cartesian coordinates and adopt them in hyperboloidal similarity coordinates. For this purpose, we consider the equation in terms of the vector fields $\pd_{t} \pm \pd_{r}$. Adding a smart zero then gives
\begin{equation*}
\big( \pd_{t} \mp \pd_{r} \big) \big( r^{d-1} ( \pd_{t} \pm \pd_{r} ) \widehat{u}(t,r) \big) \pm (d-1)r^{d-2} \pd_{t} \widehat{u}(t,r) = 0 \,.
\end{equation*}
As in the one-dimensional setting, we introduce the forward and backward half-waves
\begin{equation*}
\widehat{u}_\pm(t,r) \coloneqq \pd_{t} \widehat{u}(t,r) \pm \pd_{r} \widehat{u} (t,r) \,,
\end{equation*}
respectively, so that the radial wave equation becomes
\begin{equation*}
\big( \pd_{t} \mp \pd_{r} \big) \big( \widehat{u}_\pm(t,r) r^{d-1} \big) \pm \tfrac{d-1}{2r} \big( \widehat{u}_+(t,r) r^{d-1} + \widehat{u}_-(t,r) r^{d-1} \big) = 0 \,.
\end{equation*}
This is the starting point for the derivation of an energy estimate for which it is clear how it translates to hyperboloidal similarity coordinates. Testing with $\widehat{u}_\pm(t,r)$ yields
\begin{equation*}
\big( \pd_{t} \mp \pd_{r} \big) \big( \widehat{u}_\pm(t,r)^{2} r^{d-1} \big) \pm \tfrac{d-1}{r} \widehat{u}_+(t,r) \widehat{u}_-(t,r) r^{d-1}  = 0 \,,
\end{equation*}
which is equivalent to the differential energy conservation
\begin{equation}
\label{DifferentialEnergy}
\big( \pd_{t} - \pd_{r} \big) \big( \widehat{u}_+(t,r)^{2}r^{d-1} \big) + \big( \pd_{t} + \pd_{r} \big) \big( \widehat{u}_-(t,r)^{2}r^{d-1} \big) = 0 \,.
\end{equation}
The picture is somehow different in one space dimension, as the equations for the fields $\widehat{u}_\pm$ decouple and make a separate analysis possible. Now we transform the equations from above to hyperboloidal similarity coordinates. The half-waves are
\begin{equation*}
\widehat{v}_\pm(s,\eta) \coloneqq \widehat{u}_\pm(T+\ee^{-s}h(\eta),\ee^{-s}\eta)
\end{equation*}
and the differential energy conservation multiplied by $\ee^{-s}$ implies
\begin{align*}
0 = \pd_s & \Big( \ee^{-ds} \big( (1+h'(\eta))\widehat{v}_+(s,\eta)^{2} + (1-h'(\eta))\widehat{v}_-(s,\eta)^{2} \big) \eta^{d-1} \Big)
\\
+\,\pd_\eta&\Big( \ee^{-ds} \big( (\eta+h(\eta))\widehat{v}_+(s,\eta)^{2} + (\eta-h(\eta))\widehat{v}_-(s,\eta)^{2} \big) \eta^{d-1} \Big) \,.
\end{align*}
This gives integrated over the interval $[0,R]$
\begin{align*}
&\frac{\od}{\od s} \Big( \ee^{-ds} \int_{0}^R \big( (1+h'(\eta))\widehat{v}_+(s,\eta)^{2} + (1-h'(\eta))\widehat{v}_-(s,\eta)^{2} \big) \eta^{d-1} \dd \eta \Big) \\&=
- \Big[ \ee^{-ds} \big( (\eta+h(\eta))\widehat{v}_+(s,\eta)^{2} + (\eta-h(\eta))\widehat{v}_-(s,\eta)^{2} \big) \eta^{d-1} \Big]_{0}^R \leq 0 \,,
\end{align*}
provided that $R\geq \frac{1}{2}$. Because $1\pm h' \simeq 1$ on $\overline{\BB_{R}}$ we have $(1+h')\widehat{v}_+^{2} + (1-h')\widehat{v}_-^{2} \simeq \widehat{v}_+^{2} + \widehat{v}_-^{2}$ so
\begin{equation*}
\int_{0}^R \big( \widehat{v}_+(s,\eta)^{2} + \widehat{v}_-(s,\eta)^{2} \big) \eta^{d-1} \dd \eta \lesssim \ee^{ds} \int_{0}^R \big( \widehat{v}_+(s_{0},\eta)^{2} + \widehat{v}_-(s_{0},\eta)^{2} \big) \eta^{d-1} \dd \eta \,.
\end{equation*}
Taking the transformations \eqref{PDHSC0}, \eqref{PDHSCj} for the partial derivatives into account, we have
\begin{align*}
\widehat{v}_-(s,\eta) &= \frac{\ee^s}{\eta h'(\eta)-h(\eta)} \big( (1-h'(\eta))\pd_s\widehat{v}(s,\eta) + (\eta-h(\eta)) \pd_\eta \widehat{v}(s,\eta) \big) \,, \\
\widehat{v}_+(s,\eta) &= \frac{\ee^s}{\eta h'(\eta)-h(\eta)} \big( (1+h'(\eta))\pd_s\widehat{v}(s,\eta) + (\eta+h(\eta)) \pd_\eta \widehat{v}(s,\eta) \big) \,, \\
\pd_s \widehat{v}(s,\eta) &= \frac{\ee^{-s}}{2} \big( (\eta-h(\eta))\widehat{v}_-(s,\eta) - (\eta+h(\eta))\widehat{v}_+(s,\eta) \big)  \,, \\
\pd_\eta\widehat{v}(s,\eta) &=  \frac{\ee^{-s}}{2} \big( -(1-h'(\eta))\widehat{v}_-(s,\eta) + (1+h'(\eta))\widehat{v}_+(s,\eta) \big) \,.
\end{align*}
The occurring coefficients are smooth with all their derivatives bounded on $\overline{\BB_{R}}$, so we get immediately
\begin{equation*}
\int_{0}^R \left( \pd_s \widehat{v}(s,\eta)^{2} + \pd_\eta \widehat{v}(s,\eta)^{2} \right) \eta^{d-1} \dd \eta \simeq \ee^{-2s} \int_{0}^R \big( \widehat{v}_+(s,\eta)^{2} + \widehat{v}_-(s,\eta)^{2} \big) \eta^{d-1} \dd \eta \,.
\end{equation*}
Using that $\widehat{v}(s,\,.\,)$ is even, this implies precisely the \emph{standard energy estimate}
\begin{equation*}
\| v(s,\,.\,) \|_{\dot{H}^1(\BB^{d}_{R})} + \| \pd_s v(s,\,.\,) \|_{L^{2}(\BB^{d}_{R})} \lesssim \ee^{ \left( \frac{d}{2} - 1 \right) s } \left( \| v(s_{0},\,.\,) \|_{\dot{H}^1(\BB^{d}_{R})} + \| \pd_{0} v(s_{0},\,.\,) \|_{L^{2}(\BB^{d}_{R})}  \right) \,.
\end{equation*}
However, via this approach, it is unclear how to obtain sharp estimates that incorporate more derivatives of the solution $v(s,\,.\,)$ in hyperboloidal similarity coordinates. Additionally to that, the above quantity only provides a seminorm. We can get the above estimate for the $H^1(\BB_R^d)\times L^2 (\BB_R^d)$-norm via a Gr\"{o}nwall-type argument, with the expense that the exponential scaling factor increases by $\ee^{s/2}$, which is not desirable. This is why we resort to the \emph{descent method} below in \Cref{MethodDescent}. For a related discussion on how to derive from a conserved quantity of the wave equation a norm that is suitable for the study of stability of self-similar blowup see \cite[Section 2]{MR2881965}.
\end{remark}
\subsection{Half-waves}
Next, we digress to existence for the half-wave equation and use it later on when we draw the connection to the one-dimensional wave evolution. The symbolic notation $\mathbf{f}_\mp \coloneqq [\, f_-,f_+ \,]$ for tuples of functions $f_\pm$ will come in handy in what follows.
\begin{definition}
Let $R>0$ and $k\in\NN$. We define the \emph{parity operator} $P$ as the extension to the whole space of the bounded linear operator $P: \dom(P) \subset H^{k-1}(\BB_{R}) \rightarrow H^{k-1}(\BB_{R})$, densely defined on $\dom(P) = C^\infty(\overline{\BB_{R}})$ by $P f = f(-\,.\,)$.
\end{definition}
\begin{definition}
Let $R>0$ and $k\in\NN$. The space $\HH^{k-1}_{\mp}(\BB_{R})$ arises as the completion of
\begin{equation*}
\big\{ \mathbf{f}_\mp \in C^\infty(\overline{\BB_{R}})^{2} \mid P f_- = - f_+ \big\}
\end{equation*}
with respect to the norm $\| \,.\, \|_{\HH^{k-1}_\mp(\BB_{R})}$ induced by the inner product
\begin{equation*}
\big( \mathbf{f}_\mp \,\big|\, \mathbf{g}_\mp \big)_{\HH^{k-1}_\mp(\BB_{R})} \coloneqq \sum_{j=0}^{k-1} \left(
\big( D_-^{j} f_- \,\big|\, D_-^{j} g_- \big)_{L^{2}(\BB_{R},\,h_-')} +
\big( D_+^{j} f_+ \,\big|\, D_+^{j} g_+ \big)_{L^{2}(\BB_{R},\, h_+')}
\right) \,.
\end{equation*}
\end{definition}
\begin{remark}
\label{equivHkmp}
We see with \Cref{halfwavenorm} the equivalence
\begin{equation*}
\| \mathbf{f}_\mp \|_{\HH^{k-1}_\mp(\BB_{R})} \simeq \| \mathbf{f}_\mp \|_{H^{k-1}(\BB_{R})^{2}}
\end{equation*}
for all $\mathbf{f}_\mp \in C^\infty(\overline{\BB_{R}})^{2}$ with $P f_- = - f_+$. Hence, the space $\HH^{k-1}_{\mp}(\BB_{R})$ is isomorphic to the closed subspace $\big\{ \mathbf{f}_\mp \in H^{k-1}(\BB_{R})^{2} \mid P f_- = - f_+ \big\}$ and henceforth may be identified with it.
\end{remark}
Note the identity
\begin{equation*}
P L_- f_- = L_+ P f_- = - L_+ f_+
\end{equation*}
for all $\mathbf{f}_\mp \in C^\infty(\overline{\BB_{R}})^{2}$ with $P f_- = - f_+$ in order to give the following definition.
\begin{definition}
Let $R>0$ and $k\in\NN$. We define the unbounded linear operator $\mathbf{L}_\mp: \dom(\mathbf{L}_\mp) \subset \HH^{k-1}_\mp(\BB_{R}) \rightarrow \HH^{k-1}_\mp(\BB_{R})$, densely defined on $\dom(\mathbf{L}_\mp) = \big\{ \mathbf{f}_\mp \in C^\infty(\overline{\BB_{R}})^{2} \mid P f_- = - f_+ \big\}$ by
\begin{equation*}
\mathbf{L}_\mp \mathbf{f}_\mp =
\begin{bmatrix}
L_- f_- \\
L_+ f_+
\end{bmatrix}
\,.
\end{equation*}
\end{definition}
The half-wave equation reads
\begin{equation*}
\pd_s \mathbf{v}_\mp(s, \,.\,) = \mathbf{L}_\mp \mathbf{v}_\mp(s, \,.\,)
\end{equation*}
in terms of the variable $\mathbf{v}_\mp(s, \,.\,) \coloneqq [\, v_-(s,\,.\,) , v_+(s,\,.\,) \,]$. The following proposition gives a solution to the corresponding abstract Cauchy problem.
\begin{proposition}
\label{one-dim-propHSC}
Let $R\geq\frac{1}{2}$ and $k\in\NN$. The operator $\mathbf{L}_\mp$ is closable and its closure $\overline{\mathbf{L}_\mp}: \dom(\overline{\mathbf{L}_\mp}) \subset \HH^{k-1}_\mp(\BB_{R}) \rightarrow \HH^{k-1}_\mp(\BB_{R})$ is the generator of a strongly continuous one-parameter semigroup $\mathbf{S}_\mp : [0,\infty) \rightarrow \mathfrak{L}\big( \HH^{k-1}_\mp(\BB_{R}) \big)$ with the bound
\begin{equation}
\label{Spmbound}
\| \mathbf{S}_\mp(s) \mathbf{f}_\mp \|_{\HH^{k-1}_\mp(\BB_{R})} \lesssim \ee^{s/2} \| \mathbf{f}_\mp \|_{\HH^{k-1}_\mp(\BB_{R})}
\end{equation}
for all $\mathbf{f}_\mp\in \HH^{k-1}_\mp(\BB_{R})$ and all $s\geq 0$.
\end{proposition}
\begin{proof}
We have for each component
\begin{equation*}
\big(L_\pm f_\pm \,\big|\, f_\pm \big)_{L^{2}(\BB_{R},\,h_\pm')} = - \Big[ (y\pm h(y)) |f_\pm(y)|^{2} \Big]_{-R}^R - \overline{\big(L_\pm f_\pm \,\big|\, f_\pm \big)_{L^{2}(\BB_{R},\,h_\pm')}} + \| f_\pm \|_{L^{2}(\BB_{R},\,h_\pm')}^{2} \,,
\end{equation*}
so $\Re\big(L_\pm f_\pm \,\big|\, f_\pm \big)_{L^{2}(\BB_{R},\,h_\pm')} \leq \frac{1}{2} \| f_\pm \|_{L^{2}(\BB_{R},\,h_\pm')}^{2}$ for all $f_\pm\in C^\infty(\overline{\BB_{R}})$, respectively, if $R\geq \frac{1}{2}$. This implies with the commutator relation \eqref{commutator}
\begin{align*}
\Re \big( \mathbf{L}_\mp \mathbf{f}_\mp \mid  \mathbf{f}_\mp \big)_{\HH^{k-1}_\mp(\BB_{R})} &=
\sum_{j=0}^{k-1} \left(  \Re \big( D_-^{j} L_- f_- \,\big|\, D_-^{j} f_- \big)_{L^{2}(\BB_{R},\,h_-')} \right. \\&\qquad+\left.
\Re \big( D_+^{j} L_+ f_+ \,\big|\, D_+^{j} f_+ \big)_{L^{2}(\BB_{R},\,h_+')} \right) \\&=
\sum_{j=0}^{k-1} \left(
\Re \big(L_- D_-^{j} f_- \,\big|\, D_-^{j} f_- \big)_{L^{2}(\BB_{R},\,h_-')} - j \big\| D_-^{j} f_- \big\|_{L^{2}(\BB_{R},\,h_-')}^{2} \right. \\&\qquad+\left.
\Re \big(L_+ D_+^{j} f_+ \,\big|\, D_+^{j} f_+ \big)_{L^{2}(\BB_{R},\,h_+')} - j \big\| D_+^{j} f_+ \big\|_{L^{2}(\BB_{R},\,h_+')}^{2} 
\right) \\&\leq
\frac{1}{2} \| \mathbf{f}_\mp \|_{\HH^{k-1}_{\mp}(\BB_{R})}^{2}
\end{align*}
for all $\mathbf{f}_\mp \in \dom(\mathbf{L}_\mp)$. Hence $\mathbf{L}_\mp - \frac{1}{2}\mathbf{I}$ is dissipative and thus closable. Moreover, let $\mathbf{F}_\mp \in C^\infty(\overline{\BB_{R}})^{2}$ with $P F_- = - F_+$ and consider the equation $(\mathbf{I} - \mathbf{L}_\mp)\mathbf{f}_\mp = \mathbf{F}_\mp$ with $\mathbf{f}_\mp \in \dom(\mathbf{L}_\mp)$. The underlying differential equations
\begin{equation*}
\frac{y \pm h(y)}{1 \pm h'(y) } f_\pm'(y) + f_\pm(y) = F_\pm(y)
\end{equation*}
have a smooth solution $\mathbf{f}_\mp$ with components
\begin{equation*}
f_\pm(y) = \frac{y \mp \tfrac{1}{2}}{y\pm h(y)} \int_{0}^1 \big( 1 \pm h'( \pm \tfrac{1}{2} + y'(y \mp \tfrac{1}{2}) ) \big) F_\pm( \pm \tfrac{1}{2} + y'(y \mp \tfrac{1}{2}) ) \dd y' \,,
\end{equation*}
respectively. Indeed, $y\pm h(y)$ is smooth with a zero of first order at $\pm 1/2$, so smoothness of the coefficient in front of the integral follows from an application of the fundamental theorem of calculus. Thus $f_\pm \in C^\infty(\overline{\BB_{R}})$, respectively, and it is immediate that $P f_- = -f_+$. So $\mathbf{f}_\mp\in\dom(\mathbf{L}_\mp)$ and the range of $\mathbf{I} - \mathbf{L}_\mp$ is dense in $\HH^{k-1}_\mp(\BB_{R})$. Now $\mathbf{L}_\mp - \tfrac{1}{2} \mathbf{I}$ is the generator of a contraction semigroup by the Lumer-Phillips Theorem \cite[p. 83, 3.15 Theorem]{MR1721989} which yields multiplied by $\ee^{s/2}$ the desired semigroup.
\end{proof}
\begin{remark}
\label{SpmClassic}
If $k\geq 3$ set $\mathbf{v}_\mp(s,\,.\,) \coloneqq \mathbf{S}_\mp(s) \mathbf{f}_\mp$ for $\mathbf{f}_\mp\in \dom(\overline{\mathbf{L}_\mp})$ and $s\geq 0$. Then $\mathbf{v_\mp}\in C^1([0,\infty)\times\overline{\BB_{R}})^{2}$ with $ v_-(s,-\,.\,) = v_+(s,\,.\,)$ and we have classical solutions of the half-wave equations $\pd_s v_\pm(s, \,.\,) = L_\pm v_\pm(s, \,.\,)$. Indeed, if $k\geq 3$, by \Cref{equivHkmp} and Sobolev embedding $\HH^{k-1}_\mp(\BB_{R}) \subset H^{k-1}(\BB_{R})^{2} \subset C^{k-2}(\overline{\BB_{R}})^{2}$ holds. Since $\mathbf{v}_\mp(s,\,.\,)\in\dom(\overline{\mathbf{L}_\mp})\subset\HH^{k-1}_\mp(\BB_{R})$, we have $\mathbf{v}_\mp(s,\,.\,) \in C^{k-2}(\overline{\BB_{R}})^{2}$ and by definition of the closure there is a sequence $( {\mathbf{v}_\mp}_n(s,\,.\,) )_{n\in\NN}$ in $\dom(\mathbf{L}_\mp)$ such that ${\mathbf{v}_\mp}_n(s,\,.\,)\to \mathbf{v}_\mp(s,\,.\,)$ and $\mathbf{L}_\mp {\mathbf{v}_\mp}_n(s,\,.\,) \to \overline{\mathbf{L}_\mp}\mathbf{v}_\mp(s,\,.\,)$ in $\HH^{k-1}_\mp(\BB_{R})$ as $n\to\infty$. Thus
\begin{align*}
\left| L_\pm {v_\pm}_n(s,y) + \frac{y\pm h(y)}{1\pm h'(y)} v_\pm'(s,y) \right| &\lesssim
\| {v_\pm'}_n(s,\,.\,) - v_\pm'(s,\,.\,) \|_{L^\infty(\BB_{R})} \\&\leq
\| {v_\pm}_n(s,\,.\,) - v_\pm(s,\,.\,) \|_{H^{k-1}(\BB_{R})} \to 0 \quad (n\to \infty)
\end{align*}
and so the generator acts as a classical differential operator in this case. Hence $\pd_s \mathbf{v}_\mp(s, \,.\,) = \overline{\mathbf{L}_\mp} \mathbf{v}_\mp(s, \,.\,) = \mathbf{L}_\mp \mathbf{v}_\mp(s, \,.\,)$ and the half-wave equation is solved classically.
\end{remark}
\begin{remark}
Let us investigate whether the exponential growth bound \eqref{Spmbound} is sharp via considering \emph{mode solutions} for the half-wave equation in lowest regularity, i.e. suppose there is a $\lambda\in\CC$ such that
\begin{equation*}
\mathbf{S}_\mp(s)\mathbf{f}_\mp = \ee^{\lambda s} \mathbf{f}_\mp
\end{equation*}
for some nontrivial $\mathbf{f}_\mp\in\HH^0_\mp(\BB_{R})$, i.e. $f_\pm \in L^{2}(\BB_{R})$ with $P f_- = -f_+$. The half-wave equation for the mode solution is equivalent to the spectral equation $(\mathbf{L}_\mp - \lambda\mathbf{I})\mathbf{f}_\mp = \mathbf{0}$, which reads
\begin{equation*}
\frac{y \pm h(y)}{1 \pm h'(y) } f_\pm'(y) + \lambda f_\pm(y) = 0 \,,
\end{equation*}
respectively. We can regard $z\pm h(z)$ as an analytic function on a domain containing $\RR$ by choosing the principal branch for the square root in the definition of $h$. Then, we can choose a complex logarithm for $z\pm h(z)$ on a simply connected domain containing $\RR\setminus\{\pm \frac{1}{2} \}$, respectively, and get there a nontrivial analytic solution
\begin{equation*}
f_\pm(y) = c_\pm \exp\left( -\lambda \int^y \frac{1 \pm h'(z)}{z \pm h(z)} \dd z \right) \,,
\end{equation*}
where $c_\pm \in \CC\setminus\{0\}$ with $c_+ = -c_-$. Then $f_\pm$ satisfies $P f_- = -f_+$ and
\begin{equation*}
|f_\pm(y)| \simeq \big| y \mp \tfrac{1}{2} \big|^{-\Re \lambda}
\end{equation*}
for all $y$ in a neighbourhood around $\pm \frac{1}{2}$. The condition $f_\pm\in L^{2}(\BB_{R})$, where $R\geq \frac{1}{2}$, forces $\lambda$ such that a possible pole at $\pm \frac{1}{2}$ is integrable. This is the case if and only if $\Re \lambda < \frac{1}{2}$ which demonstrates that the bound is sharp. Furthermore, observe that the half-wave equation admits constant solutions, so the evolution certainly can never decay exponentially.
\end{remark}
\subsection{The wave equation on the half-line}
The previous results about the half-wave equation let us treat existence of solutions to the wave equation on the half-line in hyperboloidal similarity coordinates. For this we introduce the following subspace.
\begin{definition}
Let $R>0$ and $k\in\NN$. We define the subspace
\begin{equation*}
\HH^{k}_\odd(\BB_{R}) = H^{k}_\odd(\BB_{R}) \times H^{k-1}_\odd(\BB_{R}) \,.
\end{equation*}
\begin{remark}
Note that
\begin{equation*}
H^{k}_\odd(\BB_{R}) = \big\{ f\in H^{k}(\BB_{R}) \mid P f = -f \big\} \,.
\end{equation*}
\end{remark}
\end{definition}
The relations \eqref{v-}, \eqref{v+} and \eqref{pdsv}, \eqref{pdyv} in the proof of \Cref{half-wave-energy-lemma} suggest to study the following operators when asking how the half-wave variable $\mathbf{v}_\mp = [\, v_-,v_+ \,]$ determines the field $\mathbf{v} = [\, v,\pd_{0}v \,]$ and vice versa.
\begin{lemma}
\label{half-wave-operators}
Let $R>0$ and $k\in\NN$. Let $h_\pm$ be the functions from \Cref{Hilberthpm}. The operators $\mathbf{A}: \dom(\mathbf{A}) \subset \HH^{k}_\odd(\BB_{R}) \rightarrow \HH^{k-1}_\mp(\BB_{R})$, densely defined on $\dom(\mathbf{A}) = C^\infty_\odd(\overline{\BB_{R}})^{2}$ by
\begin{equation*}
(\mathbf{A} \mathbf{f})(y) =
\frac{2}{h_-(y)h_+'(y) - h_-'(y)h_+(y)}
\begin{bmatrix}
h_+(y) f_{1}'(y) + h_+'(y)f_2(y) \\
h_-(y)f_{1}'(y) + h_-'(y)f_2(y) 
\end{bmatrix} \,,
\end{equation*}
and $\mathbf{A}^{\!\!\times}: \dom(\mathbf{A}^{\!\!\times}) \subset \HH^{k-1}_\mp(\BB_{R}) \rightarrow \HH^{k}_\odd(\BB_{R})$, densely defined on $\dom(\mathbf{A}^{\!\!\times}) = \big\{ \mathbf{f}_\mp \in C^\infty(\overline{\BB_{R}})^{2} \mid P f_- = -f_+ \big\}$ by
\begin{equation*}
(\mathbf{A}^{\!\!\times}\mathbf{f}_\mp)(y) = \frac{1}{2}
\begin{bmatrix}
\displaystyle{\int_{0}^y} \big( - h_-'(y')f_-(y') +  h_+'(y')f_+(y') \big) \dd y' \\
h_-(y)f_-(y) - h_+(y)f_+(y)
\end{bmatrix}
\,,
\end{equation*}
are both well-defined and bounded with $\ran(\mathbf{A}) = \dom(\mathbf{A}^{\!\!\times})$, $\ran(\mathbf{A}^{\!\!\times}) = \dom(\mathbf{A})$ and satisfy
\begin{equation*}
\mathbf{A} \mathbf{A}^{\!\!\times} \mathbf{f}_\mp = \mathbf{f}_\mp \,, \qquad \mathbf{A}^{\!\!\times} \mathbf{A} \mathbf{f} = \mathbf{f} \,,
\end{equation*}
for all $\mathbf{f} \in \dom(\mathbf{A})$ and all $\mathbf{f}_\mp \in \dom(\mathbf{A}^{\!\!\times})$.
\end{lemma}
\begin{proof}
We have $Ph_\pm = - h_\mp$, so
\begin{align*}
P \Big( \frac{2}{h_-h_+'-h_-'h_+} \big( h_\pm f_{1}' + h_\pm' f_2 \big) \Big) &=
 \frac{2}{h_-h_+'-h_-'h_+} \big( -h_\mp P (f_{1}') + h_\mp' P f_2 \big) \\&=
-\frac{2}{h_-h_+'-h_-'h_+} \big( h_\mp f_{1}' + h_\mp' f_2 \big)
\end{align*}
for all $\mathbf{f}\in\dom(\mathbf{A})$, which shows $\mathbf{A} \mathbf{f} \in \dom(\mathbf{A}^{\!\!\times})$ for all $\mathbf{f}\in\dom(\mathbf{A})$. Also,
\begin{equation*}
P \big( h_-f_- - h_+f_+ \big) = -h_+P f_- - h_+P f_+ =
- \big( h_-f_- - h_+f_+ \big)
\end{equation*}
and
\begin{equation*}
P \big( -h_-'f_- + h_+'f_+ \big) = -h_+' P f_- + h_-' P f_+ =
-h_-'f_- + h_+'f_+
\end{equation*}
for all $\mathbf{f}_\mp \in \dom(\mathbf{A}^{\!\!\times})$. This gives $\mathbf{A}^{\!\!\times} \mathbf{f}_\mp \in \dom(\mathbf{A})$ for all $\mathbf{f}_\mp \in \dom(\mathbf{A}^{\!\!\times})$. Boundedness is clear, also see the proof of \Cref{half-wave-energy-lemma}. A straightforward computation shows $\mathbf{A} \mathbf{A}^{\!\!\times} \mathbf{f}_\mp = \mathbf{f}_\mp$ and $\mathbf{A}^{\!\!\times} \mathbf{A} \mathbf{f} = \mathbf{f}$ for all $\mathbf{f} \in \dom(\mathbf{A})$ and all $\mathbf{f}_\mp \in \dom(\mathbf{A}^{\!\!\times})$.
\end{proof}
The operator $\mathbf{A}$ forms a half-wave and $\mathbf{A}^{\!\!\times}$ reverses it. Let us connect the half-wave equation with the one-dimensional wave equation. Therefore, recall from the beginning of \Cref{MasterSection} if $u,v\in C^\infty(\RR^{1,1})$ are related by $v = u\circ\eta_{T}$ and $\big( -\pd_{t}^{2} + \pd_{x}^{2} \big) u(t,x) = 0$, then
\begin{equation}
\label{wave1inHSC}
\pd_s^{2} v(s,y) = \big( c_{11}(y) \pd_y + c_{12}(y) \pd_y^{2} + c_{20}(y) \pd_s + c_{21}(y) \pd_y \pd_s \big) v(s,y)
\end{equation}
with coefficients
\begin{align*}
c_{11}(y) &= - \frac{h(y)^{2} - y^{2}}{1-h'(y)^{2}} \frac{y h''(y)}{y h'(y)-h(y)} - 2\frac{y - h(y)h'(y)}{1-h'(y)^{2}} \,, \\
c_{12}(y) &= \frac{h(y)^{2} - y^{2}}{1-h'(y)^{2}} \,, \\
c_{20}(y) &= -1 - \frac{h(y)^{2} - y^{2}}{1-h'(y)^{2}} \frac{h''(y)}{y h'(y)-h(y)} \,, \\
c_{21}(y) &= -2 \frac{y-h(y)h'(y)}{1-h'(y)^{2}} \,.
\end{align*}
Observe that $c_{11}, c_{21}$ are odd and $c_{12},c_{20}$ are even smooth functions.
\begin{definition}
Let $R>0$ and $k\in\NN$. Let $\mathbf{A}_\mp \in \mathfrak{L}\big( \HH^{k}_\odd(\BB_{R}), \HH^{k-1}_\mp(\BB_{R}) \big)$ be the boundedly invertible extension of the operator $\mathbf{A}$, see \Cref{half-wave-operators}.
\end{definition}
\begin{definition}
Let $R>0$ and $k\in\NN$. The one-dimensional wave evolution on the half-line is given by the unbounded linear operator $\mathbf{L}_{1}: \dom(\mathbf{L}_{1}) \subset \HH^{k-1}_\odd(\BB_{R}) \rightarrow \HH^{k-1}_\odd(\BB_{R})$, densely defined on $\dom(\mathbf{L}_{1}) = C^\infty_\odd(\BB_{R})^{2}$ by
\begin{equation*}
\mathbf{L}_{1} \mathbf{f} =
\begin{bmatrix}
f_2 \\
c_{11} f_{1}' + c_{12} f_{1}'' + c_{20} f_2 + c_{21} f_2'
\end{bmatrix}
\,.
\end{equation*}
\end{definition}
Then, the wave equation \eqref{wave1inHSC} reads
\begin{equation*}
\pd_s \mathbf{v}(s,\,.\,) = \mathbf{L}_{1} \mathbf{v}(s,\,.\,)
\end{equation*}
in terms of the variable $\mathbf{v} = [\, v,\pd_{0} v \,]$.
\begin{remark}
\label{WaveEq-HalfWave}
A quick computation reveals that the operators $\mathbf{A}^{\!\!\times}\mathbf{L}_\mp\mathbf{A} - \mathbf{I}$ and $\mathbf{L}_{1}$ are in fact equal, i.e. they are both densely defined with domain $C^\infty_\odd(\overline{\BB_{R}})^{2}$ and
\begin{equation*}
\mathbf{A}^{\!\!\times}\mathbf{L}_\mp\mathbf{A} \mathbf{f} - \mathbf{f} = \mathbf{L}_{1} \mathbf{f}
\end{equation*}
for all $\mathbf{f} \in C^\infty_\odd(\overline{\BB_{R}})^{2}$.
\end{remark}
The construction of the solution operator and well-posedness for the wave evolution on the half-line in hyperboloidal similarity coordinates are now an easy consequence.
\begin{proposition}
\label{one-dim-semigrp-prop}
Let $R\geq \frac{1}{2}$, $k\in\NN$. The operator $\mathbf{L}_{1}$ is closable and its closure
\begin{equation}
\label{generatorL1}
\overline{\mathbf{L}_{1}} = \mathbf{A}^{-1}_\mp \overline{\mathbf{L}_\mp}\mathbf{A}_\mp - \mathbf{I} \,, \quad \dom(\overline{\mathbf{L}_{1}}) = \big\{ \mathbf{f} \in \mathfrak{H}^{k}_\odd(\BB_{R}) \mid \mathbf{A}_\mp \mathbf{f} \in \dom(\overline{\mathbf{L}_\mp}) \big\} \,,
\end{equation}
is the generator of the rescaled similar semigroup
\begin{equation}
\label{S1}
\mathbf{S}_{1}(s) = \ee^{-s} \mathbf{A}^{-1}_\mp \mathbf{S}_\mp(s)\mathbf{A}_\mp
\end{equation}
of bounded linear operators $\mathfrak{L}\big( \HH^{k}_\mathrm{odd}(\BB_{R}) \big)$ which satisfies
\begin{equation}
\| \mathbf{S}_{1}(s)\mathbf{f} \|_{\HH^{k}(\BB_{R})} \lesssim \ee^{-s/2} \| \mathbf{f} \|_{\HH^{k}(\BB_{R})}
\end{equation}
for all $\mathbf{f}\in \HH^{k}_\mathrm{odd}(\BB_{R})$ and all $s\geq 0$.
\end{proposition}
\begin{proof}
We see in \Cref{WaveEq-HalfWave} that $\mathbf{A}^{\!\!\times}\mathbf{L}_\mp\mathbf{A} - \mathbf{I}$ and $\mathbf{L}_{1}$ are equal as operators. The operator $\mathbf{L}_{1}$ is closable as a consequence of $\mathbf{L}_\mp$ being closable and $\mathbf{A}$, $\mathbf{A}^{\!\!\times}$ being densely defined bounded linear operators that extend to mutually inverse operators. With this at hand, the closure is easily computed. The rescaled similar semigroup is well-defined by \Cref{one-dim-propHSC} and \Cref{half-wave-operators}, also see \cite[p. 59 f.]{MR1721989} for this standard construction. The norm bound follows from \Cref{Spmbound} together with boundedness of $\mathbf{A}_\mp$, $\mathbf{A}^{-1}_\mp$.
\end{proof}
\begin{remark}
\label{S1Classic}
If $k\geq 3$ set $\mathbf{v}(s,\,.\,) \coloneqq \mathbf{S}_{1}(s) \mathbf{f}$ for $\mathbf{f}\in\dom(\overline{\mathbf{L}_{1}})$ and let $v(s,\,.\,) \coloneqq [\mathbf{v}(s,\,.\,)]_{1}$ for $s\geq 0$. Then $v\in C^{2}([0,\infty)\times\overline{\BB_{R}})$, $v(s,\,.\,)$ is odd and we have a classical solution to the wave equation in hyperboloidal similarity coordinates. Indeed, this follows from the definition of the semigroup $\mathbf{S}_{1}$ with \Cref{SpmClassic} and the form of $\mathbf{A}_\mp$.
\end{remark}
\subsection{Radial wave evolution}
\label{SecRadWave}
The next step in our programme is to understand the radial wave equation in higher dimensions $d>1$. To fix notation, if $\mathbf{f}$ is a radial function then $\mathbf{\widehat{f}}$ denotes its radial representative, that is an even function such that $\mathbf{f}(\,.\,) = \widehat{\mathbf{f}}(|\,.\,|)$. The free radial wave evolution is naturally defined in the following spaces.
\begin{definition}
Let $R>0$ and $d,k\in\NN$. We define the closed subspaces
\begin{equation*}
\HH^{k}_\rad(\BB^{d}_{R}) \coloneqq H^{k}_\rad(\BB^{d}_{R}) \times H^{k-1}_\rad(\BB^{d}_{R}) \,.
\end{equation*}
\end{definition}
Now, let $u,v\in C^\infty(\RR^{1,d})$ be related by $v = u\circ \eta_{T}$. If we suppose that $u(t,\,.\,)$ is radial then there is $\widehat{u}\in C^\infty(\RR^{1,1})$ such that $\widehat{u}(t,\,.\,)$ is even and $u(t,\,.\,) = \widehat{u}(t,|\,.\,|)$. Then $v(s,\,.\,)$ is also radial since
\begin{equation*}
v(s,y) = u(T+\ee^{-s}h(y),\ee^{-s}y) = \widehat{u}(T+\ee^{-s}h(|y|),\ee^{-s}|y|) \eqqcolon \widehat{v}(s,|y|)
\end{equation*}
yields a $\widehat{v}\in C^\infty(\RR^{1,1})$ such that $\widehat{v}(s,\,.\,)$ is even and $v(s,\,.\,) = \widehat{v}(s,|\,.\,|)$. Given that $u$ solves the wave equation
\begin{equation*}
\big( -\pd_{t}^{2} + \Delta_{x} \big) u(t,x) = 0 \,,
\end{equation*}
we get that $\widehat{u}$ solves the radial wave equation
\begin{equation*}
-\pd_{t}^{2} \widehat{u}(t,r) + \pd_{r}^{2} \widehat{u}(t,r) + \frac{d-1}{r} \pd_{r} \widehat{u}(t,r) = 0
\end{equation*}
and from \Cref{secondorderwave} we infer for $\widehat{v}$ the radial wave equation in hyperboloidal similarity coordinates,
\begin{equation}
\label{waveinHSC}
\pd_s^{2} \widehat{v}(s,\eta) = \big( c_{11}^d(\eta) \pd_\eta + c_{12}(\eta) \pd_\eta^{2} + c_{20}^d(\eta) \pd_s + c_{21}(\eta) \pd_\eta\pd_s \big) \widehat{v}(s,\eta) \,,
\end{equation}
with coefficients
\begin{align}
\label{ccoeff}
c_{11}^d(\eta) &= - (d-1)\frac{\eta h'(\eta) - h(\eta)}{1-h'(\eta)^{2}} \frac{h(\eta)}{\eta} - \frac{h(\eta)^{2} - \eta^{2}}{1-h'(\eta)^{2}} \frac{\eta h''(\eta)}{\eta h'(\eta)-h(\eta)} - 2\frac{\eta - h(\eta)h'(\eta)}{1-h'(\eta)^{2}} \,, \\
c_{12}(\eta) &= \frac{h(\eta)^{2} - \eta^{2}}{1-h'(\eta)^{2}} \,, \\
c_{20}^d(\eta) &= -1 - (d-1) \frac{\eta h'(\eta)-h(\eta)}{1-h'(\eta)^{2}} \frac{h'(\eta)}{\eta} - \frac{h(\eta)^{2} - \eta^{2}}{1-h'(\eta)^{2}} \frac{h''(\eta)}{\eta h'(\eta)-h(\eta)} \,, \\
c_{21}(\eta) &= - 2 \frac{\eta - h(\eta)h'(\eta)}{1-h'(\eta)^{2}} \,.
\end{align}
Observe that $c_{21}$ is odd, $(\,.\,) c_{11}^d$, $c_{20}^d$, $c_{12}$ are even and all of them are smooth. In order to formulate the wave equation as a linear first order system we introduce the wave operator in hyperboloidal similarity coordinates.
\begin{definition}
Let $R>0$ and $d,k\in\NN$. We define the \emph{free radial wave evolution} as the unbounded linear operator $\mathbf{L}_d: \dom(\mathbf{L}_d)\subset \HH^{k}_\rad(\BB^{d}_{R}) \rightarrow \HH^{k}_\rad(\BB^{d}_{R})$, densely defined on $\dom(\mathbf{L}_d) = C^\infty_\rad(\overline{\BB^{d}_{R}})^{2}$ by $\mathbf{L}_d \mathbf{f} \coloneqq \widehat{\mathbf{L}}_d \widehat{\mathbf{f}}(|\,.\,|)$, where
\begin{equation*}
\widehat{\mathbf{L}}_d \widehat{\mathbf{f}} \coloneqq
\begin{bmatrix}
\widehat{f}_2 \\
c_{11}^d \widehat{f}_{1}' + c_{12} \widehat{f}_{1}'' + c_{20}^d \widehat{f}_2 + c_{21} \widehat{f}_2'
\end{bmatrix}
\,.
\end{equation*}
\end{definition}
\Cref{waveinHSC} is equivalent to
\begin{equation*}
\pd_s\mathbf{v}(s,\,.\,) = \mathbf{L}_d\mathbf{v}(s,\,.\,)
\end{equation*}
where $\mathbf{v}(s,\,.\,) = [\, v(s,\,.\,), \pd_s v(s,\,.\,) \,]$.
\subsection{Descent method for the radial wave equation}
\label{MethodDescent}
The key insight for transporting the well-posedness result and the growth estimates from \Cref{one-dim-semigrp-prop} to higher dimensions are transformations that map the radial wave equation in higher dimensions to the one-dimensional wave equation. We build up our way with transformations that map between expressions that are of a type like the radial part of the Laplace operator. It suffices to consider only spatial transformations, since it is the purely spatial radial part of the Laplace operator that encodes the space dimension. For the beginning we orient ourself towards the strategy in \cite{MR2881965}, \cite{MR3278903}, \cite{MR3218816}.
\subsubsection{Descent by multiplication} Let $f_{1} \in C^\infty(\RR)$. We begin by considering the expression $\big( \pd_{r}^{2} + \tfrac{d_{1}}{r} \pd_{r} \big) f_{1}(r)$ for some constant $d_{1}\in\RR$. Introduce a function by $r^\alpha f_{1}(r) $ for some $\alpha\in\RR$. Differentiation yields
\begin{equation*}
r^{\alpha}\big( \pd_{r}^{2} + \tfrac{d_{1}}{r} \pd_{r} \big) f_{1}(r) = \big( \pd_{r}^{2} + \tfrac{d_{1} - 2\alpha}{r} \pd_{r} + \tfrac{\alpha( \alpha - d_{1} + 1)}{r^{2}} \big) \big( r^\alpha f_{1}(r) \big)
\end{equation*}
at every $r>0$. By making a choice on $\alpha$ we can decide which term we want to cancel. Since our goal is to transform between radial wave equations, let us try to make the zero order term disappear. This amounts to the choice $\alpha \in \{ 0, d_{1} - 1 \}$, where in the trivial case the transformation is just the identity and the original equation remains invariant. The function given by $r^{d_{1}-1}f_{1}(r)$ satisfies
\begin{equation}
\label{DescMult}
r^{d_{1} - 1} \big( \pd_{r}^{2} + \tfrac{d_{1}}{r} \pd_{r} \big) f_{1}(r) = \big( \pd_{r}^{2} + \tfrac{-d_{1}+2}{r} \pd_{r} \big) \big( r^{d_{1} - 1} f_{1}(r) \big)
\end{equation}
at every $r>0$.
\subsubsection{Descent by differentiation}
Let $f_2 \in C^\infty(\RR)$. Instead we may first differentiate the field and then apply a multiplication by some power. That is, consider the expression $\big( \pd_{r}^{2} + \tfrac{d_2}{r} \pd_{r} \big) f_2(r)$ and differentiate it,
\begin{equation*}
\pd_{r} \big( \pd_{r}^{2} + \tfrac{d_2}{r} \pd_{r} \big) f_2(r) = 
\big( \pd_{r}^{2} + \tfrac{d_2}{r} \pd_{r} - \tfrac{d_2}{r^{2}} \big) \pd_{r} f_2(r) \,.
\end{equation*}
The function given by $r^\beta \pd_{r} f_2(r)$ for some $\beta\in\RR$ satisfies
\begin{equation*}
r^\beta \pd_{r} \big( \pd_{r}^{2} + \tfrac{d_2}{r} \pd_{r} \big) f_2(r) = \big( \pd_{r}^{2} + \tfrac{d_2 - 2\beta}{r} \pd_{r} - \tfrac{\beta( \beta - d_2 + 1) - d_2}{r^{2}} \big) \big( r^\beta \pd_{r} f_2(r) \big)
\end{equation*}
at every $r>0$. Again, we want to preserve the radial character of the equation. So, the zero order term vanishes if $\beta \in \{ -1, d_2 \}$. This provides the identities
\begin{equation}
\label{DescDiff1}
r^{-1} \pd_{r} \big( \pd_{r}^{2} + \tfrac{d_2}{r} \pd_{r} \big) f_2(r) = \big( \pd_{r}^{2} + \tfrac{d_2 + 2}{r} \pd_{r} \big) \big( r^{-1} \pd_{r} f_2(r) \big)
\end{equation}
and
\begin{equation}
\label{DescDiffd2}
r^{d_2} \pd_{r} \big( \pd_{r}^{2} + \tfrac{d_2}{r} \pd_{r} \big) f_2(r) = \big( \pd_{r}^{2} + \tfrac{-d_2}{r} \pd_{r} \big) \big( r^{d_2} \pd_{r} f_2(r) \big)
\end{equation}
at every $r>0$. The first transformation increases the effective dimension by two. The second transformation makes a negative dimension positive and vice versa.
\subsubsection{Combined descent}
We combine both previous transformations and are lead to the \emph{descent method}. If we start with a multiplication and then combine it with differentiation, there are two types of transformations that yield the radial part of the Laplacian in a lower positive dimension. We capture the gist from the discussion above in a lemma.
\begin{lemma}
\label{2descentlemma}
Let $d\in\NN$ with $d \geq 3$. Then we have the commutation relations
\begin{align}
\label{intertwiningodd}
\big( (d-2) + r \pd_{r} \big)\big( -\pd_{t}^{2} + \pd_{r}^{2} + \tfrac{d-1}{r} \pd_{r} \big) \widehat{u}(t,r) &= \big( -\pd_{t}^{2} + \pd_{r}^{2} + \tfrac{d-3}{r} \pd_{r} \big) \big( (d-2) + r\pd_{r} \big) \widehat{u}(t,r) \,, \\
\label{intertwiningthree}
r\big( -\pd_{t}^{2} + \pd_{r}^{2} + \tfrac{2}{r} \pd_{r} \big) \widehat{u}(t,r) &= \big( -\pd_{t}^{2} + \pd_{r}^{2} \big) \big( r\widehat{u}(t,r) \big) \,,
\end{align}
for all $\widehat{u}\in C^\infty(\RR^{1,1})$ such that $\widehat{u}(t,\,.\,)$ is even.
\end{lemma}
\begin{proof}
\Cref{DescMult} for $d_{1} = d - 1$ and \Cref{DescDiffd2} for $d_2 = d - 3$ imply
\begin{equation*}
r^{-(d-3)} \pd_{r} \big( r^{d-2} \widehat{u}(t,r) \big) = \big( (d-2) + r\pd_{r} \big) \widehat{u}(t,r)
\end{equation*}
and we obtain \Cref{intertwiningodd}. If we consider \Cref{DescMult} for $d_{1} = 2$ we obtain \Cref{intertwiningthree}.
\end{proof}
These identities hold regardless whether $\widehat{u}$ solves a radial wave equation or not. Some comments about other combinations are in order.
\begin{remark}
The radial wave equation in an odd space dimension $d$ is mapped to the one-dimensional wave equation via $\big( r^{-1} \pd_{r} \big)^{\frac{d-3}{2}} \big( r^{d-2} \widehat{u}(t,r) \big)$, as is seen by starting with a descent by multiplication \eqref{DescMult} and applying $(d-3)/2$ times a descent by differentiation of the form \eqref{DescDiff1}, also see \cite{MR3742520}. Note that this also follows from repeated application of transformations of the type $r^{-(d-3)}\pd_{r} \big( r^{d-2} \widehat{u}(t,r) \big)$
from \Cref{2descentlemma}, which map solutions of the radial $d$-dimensional wave equation in any given spatial dimension $d$ into a solution in $d-2$ dimensions.
\end{remark}
\begin{remark}
When we start from a radial $d$-dimensional wave equation and apply a descent by differentiation of the form \eqref{DescDiffd2} two times in order to get from the negative effective dimension back to the original dimension, we recover the obvious fact that the radial Laplace operator $r^{-d+1}\pd_{r}(r^{d-1}\pd_{r})$ commutes with the wave equation. At first glance, it seems that we could use this to get the differential energy estimate \eqref{DifferentialEnergy} for the radial Laplace operator and derive from this a higher energy in hyperboloidal similarity coordinates. However, pursuing this idea is not practical because on bounded domains, the $L^{2}$-norm of the Laplacian does not give a full norm. Even worse, the corresponding seminorm obtained in this way contains higher $s$-derivatives, when carried over to hyperboloidal similarity coordinates, which rules it out as a candidate for an energy norm. This might be resolved by substituting higher $s$-derivatives by lower ones with the wave equation but the resulting expressions are too difficult to handle due to their complexity and elude a systematic analysis.
\end{remark}
\subsubsection{Descent operators}
Now it is important to understand the descent method in hyperboloidal similarity coordinates. The transformation given by
\begin{equation*}
\widehat{u}^\downarrow_d(t,r) \coloneqq (d-2) \widehat{u}(t,r) + r\pd_{r} \widehat{u}(t,r)
\end{equation*}
from \Cref{2descentlemma} reads in hyperboloidal similarity coordinates
\begin{align}
\nonumber
\widehat{v}^\downarrow_d(s,\eta) &\coloneqq (d-2) \widehat{v}(s,\eta) -
\frac{\eta h'(\eta)}{\eta h'(\eta) - h(\eta)} \pd_s \widehat{v}(s,\eta) -
\frac{\eta h(\eta)}{\eta h'(\eta) - h(\eta)} \pd_\eta \widehat{v}(s,\eta)  \\&=
\label{v}
\big( (d-2) + c_{1}(\eta) \pd_\eta + c_2(\eta) \pd_s \big) \widehat{v}(s,\eta)
\end{align}
with coefficients
\begin{equation*}
c_{1}(\eta) \coloneqq -\frac{\eta h(\eta)}{\eta h'(\eta) - h(\eta)} \,, \qquad c_2(\eta) \coloneqq -\frac{\eta h'(\eta)}{\eta h'(\eta) - h(\eta)} \,.
\end{equation*}
Also,
\begin{equation*}
\pd_s \widehat{v}^\downarrow_d(s,\eta) = \big( (d-2)\pd_s +
c_{1}(\eta) \pd_\eta\pd_s +
c_2(\eta) \pd_s^{2} \big) \widehat{v}(s,\eta) \,.
\end{equation*}
If we assume for the moment that $u$ solves the wave equation, then \Cref{waveinHSC} allows to substitute the second order $s$-derivatives by lower order derivatives, in particular,
\begin{align}
\label{pdsvtilde}
\begin{split}
\pd_s \widehat{v}^\downarrow_d(s,\eta) = \Big(
&c_2(\eta) c_{11}^d(\eta) \pd_\eta
+c_2(\eta )c_{12}(\eta) \pd_\eta^{2} \\&
+\big(c_2(\eta)c_{20}^d(\eta) + (d-2)\big) \pd_s
+\big(c_2(\eta)c_{21}(\eta) + c_{1}(\eta)\big) \pd_\eta\pd_s
\Big) \widehat{v}(s,\eta) \,.
\end{split}
\end{align}
In the case $d=3$ we have the transformation
\begin{equation*}
\widehat{u}^\downarrow_3(t,r) \coloneqq r\widehat{u}(t,r)
\end{equation*}
at hand that reads in hyperboloidal similarity coordinates
\begin{align}
\label{v3}
\widehat{v}^\downarrow_3(s,\eta) &= \ee^{-s} \eta \widehat{v}(s,\eta) \,, \\
\label{pdsv3}
\pd_s \widehat{v}^\downarrow_3(s,\eta) &= \ee^{-s} \big( - \eta \widehat{v}(s,\eta) + \eta \pd_s \widehat{v}(s,\eta) \big) \,.
\end{align}
To formulate the method of descent in hyperboloidal similarity coordinates we introduce the \emph{descent operators}. Note in the following that the coefficients $c_{1}$, $c_2$ are odd, even, respectively, and smooth.
\begin{definition}
\label{DescOpDef}
Let $R>0$ and $d,k\in\NN$.
\begin{enumerate}[itemsep=1em, topsep=1em]
\item If $d>3$, the \emph{descent operator} is given by the operator $\mathbf{D}^\downarrow_d : \dom(\mathbf{D}^\downarrow_d) \subset \HH^{k+1}_{\rad}(\BB^{d}_{R}) \rightarrow \HH^{k}_\rad(\BB_{R}^{d-2})$, densely defined on $\dom(\mathbf{D}^\downarrow_d) = C^\infty_\rad (\overline{\BB^d_{R}})^{2}$ by $\mathbf{D}_d^\downarrow \mathbf{f} \coloneqq \widehat{\mathbf{D}}_d^\downarrow \widehat{\mathbf{f}}(|\,.\,|)$, where
\begin{equation*}
\widehat{\mathbf{D}}_d^\downarrow \widehat{\mathbf{f}} \coloneqq
(d-2) \widehat{\mathbf{f}} + c_{1} \widehat{\mathbf{f}}' + c_2 \widehat{\mathbf{L}}_d \widehat{\mathbf{f}} \,.
\end{equation*}
\item If $d=3$, the \emph{descent operator} is given by the operator $\mathbf{D}^\downarrow_3 : \dom(\mathbf{D}^\downarrow_3) \subset \HH^{k}_{\rad}(\BB_{R}^3) \rightarrow \HH^{k}_\odd(\BB_{R})$, densely defined on $\dom(\mathbf{D}^\downarrow_3) = C^\infty_\rad (\overline{\BB^3_{R}})^{2}$ by $\mathbf{D}_3^\downarrow \mathbf{f} \coloneqq \widehat{\mathbf{D}}_3^\downarrow \widehat{\mathbf{f}}$, where
\begin{equation*}
\widehat{\mathbf{D}}_3^\downarrow \widehat{\mathbf{f}}(y) \coloneqq
y \widehat{\mathbf{f}}(y) - y \widehat{f}_{1}(y) \widehat{\mathbf{e}}_2
\end{equation*}
for all $y \in \overline{\BB_{R}}$.
\end{enumerate}
\end{definition}
That is, Eqs. \eqref{v}, \eqref{pdsvtilde} and \eqref{v3}, \eqref{pdsv3} read
\begin{equation*}
\widehat{\mathbf{v}}_d^\downarrow(s,\,.\,) = \widehat{\mathbf{D}}_d^\downarrow \widehat{\mathbf{v}}(s,\,.\,) \,, \qquad \widehat{\mathbf{v}}_3^\downarrow(s,\,.\,) = \ee^{-s} \widehat{\mathbf{D}}_3^\downarrow \widehat{\mathbf{v}}(s,\,.\,) \,,
\end{equation*}
respectively.
\subsection{The intertwining identity}
We have established transformations in \Cref{2descentlemma} that link the radial wave equation in $d$ dimensions with the one in $(d-2)$ dimensions. Now, the pressing question is to what extent this imposes relations between the operators $\mathbf{L}_d$, $\mathbf{L}_{d-2}$, $\mathbf{D}^\downarrow_d$. This can be answered by proving identities for the coefficients involved.
\begin{lemma}
\label{intertwiningCoeff}
Let $d\in \NN$, $d\geq 3$. Set
\begin{equation*}
c_3(\eta) \coloneqq 2\frac{\eta h'(\eta) - h(\eta)}{1-h'(\eta)^{2}} \frac{h(\eta)}{\eta} \,, \qquad
c_4(\eta) \coloneqq 2\frac{\eta h'(\eta) - h(\eta)}{1-h'(\eta)^{2}} \frac{h'(\eta)}{\eta} \,.
\end{equation*}
\begin{enumerate}[itemsep=1em, topsep=1em]
\item We have the identities
\begin{align}
\label{c11prime}
c_{1} {c_{11}^d\!\!}' &= (d-2) c_3 + c_{1}'' c_{12} + c_{1}' c_{11}^{d-2} + \left( c_{21} c_2' + c_2 c_4 \right) c_{11}^d \,, \\
\label{c20prime}
c_{1} {c_{20}^d\!\!}' &= (d-2) c_4 + c_2'' c_{12} + c_2' c_{11}^{d-2} + \left( c_{21} c_2' + c_2 c_4 \right) c_{20}^d \,, \\
\label{c12prime}
c_{1} c_{12}' &= 2 c_{1}' c_{12} + c_{1} c_3 + \left( c_{21} c_2' + c_2 c_4 \right) c_{12}  \,, \\
\label{c21prime}
c_{1} c_{21}' &= c_{1} c_4 + c_2 c_3 +  2 c_2' c_{12} + c_{1}' c_{21} + \left( c_{21} c_2' + c_2 c_4 \right) c_{21}  \,.
\end{align}
\item If $d=1$, then
\begin{align}
\label{c111}
0 &= c_{11}^1(\eta) - \eta - c_{12}(\eta) - \eta c_{20}^1(\eta) + c_{21}(\eta) \,, \\
\label{etac3}
\eta c_3(\eta) &= \eta c_{21}(\eta) - 2 c_{12}(\eta) \,, \\
\label{etac4}
\eta c_4(\eta) &= -c_{21}(\eta) - 2\eta(\eta) \,.
\end{align}
\end{enumerate}
\end{lemma}
\begin{proof}
\begin{enumerate}[wide, itemsep=1em, topsep=1em]
\item First note that
\begin{equation*}
c_{11}^{d-2} = c_{11}^d + c_3 \,, \qquad c_{20}^{d-2} = c_{20}^d + c_4 \,,
\end{equation*}
for all $d\in\NN$, $d\geq 3$. Let $u,v\in C^\infty(\RR^{1,d})$ such that $v = u\circ \eta_{T}$. The commutation relation \eqref{intertwiningodd} then reads
\begin{align*}
&\Big( (d-2) + c_{1}(\eta) \pd_\eta + c_2(\eta) \pd_s \Big) \Big( \g^{00}(s,\eta)
\\&\quad\times
\big( \pd_s^{2} - c_{11}^d(\eta) \pd_\eta - c_{12}(\eta) \pd_\eta^{2} - c_{20}^d(\eta) \pd_s - c_{21}(\eta) \pd_\eta\pd_s \big) \widehat{v}(s,\eta) \Big) \\&=
\g^{00}(s,\eta) \\&\quad\times
\Big( \pd_s^{2} - c_{11}^{d-2}(\eta) \pd_\eta - c_{12}(\eta) \pd_\eta^{2} - c_{20}^{d-2}(\eta) \pd_s - c_{21}(\eta) \pd_\eta\pd_s \Big) \Big( (d-2) + c_{1}(\eta) \pd_\eta + c_2(\eta) \pd_s \Big) \widehat{v}(s,\eta)
\end{align*}
and after computing the identity
\begin{equation}
\frac{ c_{1}(\eta) \pd_\eta \g^{00}(s,\eta) + c_2(\eta) \pd_s \g^{00}(s,\eta) }{ \g^{00}(s,\eta) } = - c_2'(\eta) c_{21}(\eta) - c_2(\eta) c_4(\eta)
\end{equation}
in order to exchange derivatives of the metric with derivatives of coefficients we see
\begin{align*}
&\Big( (d-2) + c_{1}(\eta) \pd_\eta + c_2(\eta) \pd_s
\\&\quad-
c_2'(\eta) c_{21}(\eta)
-c_2(\eta) c_4(\eta) \Big) \Big( \pd_s^{2} - c_{11}^d(\eta) \pd_\eta - c_{12}(\eta) \pd_\eta^{2} - c_{20}^d(\eta) \pd_s - c_{21}(\eta) \pd_\eta\pd_s \Big) \widehat{v}(s,\eta) \\&= \Big( \pd_s^{2} - c_{11}^{d-2}(\eta) \pd_\eta - c_{12}(\eta) \pd_\eta^{2} - c_{20}^{d-2}(\eta) \pd_s - c_{21}(\eta) \pd_\eta\pd_s \Big) \Big( (d-2) + c_{1}(\eta) \pd_\eta + c_2(\eta) \pd_s \Big) \widehat{v}(s,\eta) \,.
\end{align*}
Inserting $\eta$, $s$ for $\widehat{v}(s,\eta)$ yields \eqref{c11prime}, \eqref{c20prime}, respectively. Inserting  $\frac{1}{2} \eta^{2}$, $s\eta$ for $\widehat{v}(s,\eta)$ and exploiting the previous identities yields \eqref{c12prime}, \eqref{c21prime}, respectively.
\item The commutation relation \eqref{intertwiningthree} reads in hyperboloidal similarity coordinates directly
\begin{align*}
& \eta \ee^{-s} \g^{00}(s,\eta) \big( \pd_s^{2} - c_{11}^3(\eta) \pd_\eta - c_{12}(\eta) \pd_\eta^{2} - c_{20}^3(\eta) \pd_s - c_{21}(\eta) \pd_\eta\pd_s \big) \widehat{v}(s,\eta) \\&=
\g^{00}(s,\eta) \big( \pd_s^{2} - c_{11}^1(\eta) \pd_\eta - c_{12}(\eta) \pd_\eta^{2} - c_{20}^1(\eta) \pd_s - c_{21}(\eta) \pd_\eta\pd_s \big) \big( \eta \ee^{-s} \widehat{v}(s,\eta) \big)
\end{align*}
which yields after cancelling $\ee^{-s}\g^{00}(s,\eta)$
\begin{align*}
&\eta\big(c_{11}^3(\eta) \pd_\eta + c_{12}(\eta) \pd_\eta^{2} + c_{20}^3(\eta) \pd_s + c_{21}(\eta) \pd_\eta\pd_s \big) \widehat{v}(s,\eta) \\&=
\big( (c_{11}^1(\eta) - c_{21}(\eta)) \pd_\eta + c_{12}(\eta) \pd_\eta^{2} + (c_{20}^1(\eta) + 2) \pd_s + c_{21}(\eta) \pd_\eta\pd_s - c_{20}^1(\eta) - 1 \big) \big( \eta \widehat{v}(s,\eta) \big) \,.
\end{align*}
Inserting $\widehat{v}(s,\eta) = 1$ yields \Cref{c111}. Using this identity when $\eta$, $s$ is assigned for $\widehat{v}(s,\eta)$ yields \eqref{etac3}, \eqref{etac4}, respectively.
\qedhere
\end{enumerate}
\end{proof}
The commutation relations presented in \Cref{2descentlemma} now appear as curious \emph{intertwining identities}.
\begin{proposition}
\label{intertwiningLemma}
Let $R>0$ and $d\in\NN$, $d\geq 3$.
\begin{enumerate}[itemsep=1em, topsep=1em]
\item The commutation relation \eqref{intertwiningodd} manifests itself as
\begin{equation}
\label{intertwining-d}
\mathbf{D}_d^\downarrow \mathbf{L}_d \mathbf{f} = \mathbf{L}_{d-2} \mathbf{D}_d^\downarrow \mathbf{f}
\end{equation}
for all $\mathbf{f} \in C^\infty_\rad(\overline{\BB^{d}_{R}})^{2}$.
\item The commutation relation \eqref{intertwiningthree} manifests itself as
\begin{equation}
\label{intertwining-3}
\mathbf{D}_3^\downarrow \mathbf{L}_3 \mathbf{f} = \mathbf{L}_{1} \mathbf{D}_3^\downarrow \mathbf{f} + \mathbf{D}_3^\downarrow \mathbf{f}
\end{equation}
for all $\mathbf{f} \in C^\infty_{\mathrm{rad}} (\overline{\BB^3_{R}})^{2}$.
\end{enumerate}
\end{proposition}
\begin{proof}
We use in the computations below
\begin{align*}
\widehat{\mathbf{L}}_d ( f \widehat{\mathbf{f}} ) &= f \widehat{\mathbf{L}}_d \widehat{\mathbf{f}} + \big( (f' c_{11}^d + f'' c_{12}) \widehat{f}_{1} + 2f'c_{12} \widehat{f}_{1}' + f'c_{21} \widehat{f}_2  \big) \widehat{\mathbf{e}}_2 \,, \\
(\widehat{\mathbf{L}}_d \widehat{\mathbf{f}})' &= \widehat{\mathbf{L}}_d \widehat{\mathbf{f}}' + ({c_{11}^d\!\!}' \widehat{f}_{1}' + c_{12}' \widehat{f}_{1}'' + {c_{20}^d\!\!}' \widehat{f}_2 + c_{21}' \widehat{f}_2') \widehat{\mathbf{e}}_2 \,, \\
\widehat{\mathbf{L}}_{d-2} \widehat{\mathbf{f}} &= \widehat{\mathbf{L}}_d \widehat{\mathbf{f}} + (c_3 \widehat{f}_{1}' + c_4 \widehat{f}_2) \widehat{\mathbf{e}}_2 \,,
\end{align*}
for all even $\widehat{\mathbf{f}} \in C^\infty(\overline{\BB_{R}})^{2}$ and $f\in C^\infty(\overline{\BB_{R}})$. The first and second identity are an application of the product rule in the definition of the wave evolution operator. The third identity follows from $c_{11}^{d-2} = c_3 + c_{11}^d$ and $c_{20}^{d-2} = c_4 + c_{20}^d$.
\begin{enumerate}[wide, itemsep=1em, topsep=1em]
\item With the definition and an application of the Leibniz rule from above we compute
\begin{align*}
\widehat{\mathbf{D}}^\downarrow_d \widehat{\mathbf{L}}_d \widehat{\mathbf{f}} &= (d-2) \widehat{\mathbf{L}}_d \widehat{\mathbf{f}} + c_{1} (\widehat{\mathbf{L}}_d \widehat{\mathbf{f}})' + c_2 \widehat{\mathbf{L}}_d \widehat{\mathbf{L}}_d \widehat{\mathbf{f}} \\&=
(d-2) \widehat{\mathbf{L}}_d \widehat{\mathbf{f}} + c_{1} \widehat{\mathbf{L}}_d \widehat{\mathbf{f}}' + c_2 \widehat{\mathbf{L}}_d \widehat{\mathbf{L}}_d  \widehat{\mathbf{f}} +
c_{1} \big(
{c^d_{11}\!\!}' \widehat{f}_{1}' + c_{12}' \widehat{f}_{1}'' + {c_{20}^d\!\!}' \widehat{f}_2 + c_{21}' \widehat{f}_2'
\big) \widehat{\mathbf{e}}_2
\end{align*}
and
\begin{align*}
\widehat{\mathbf{L}}_{d-2} \widehat{\mathbf{D}}^\downarrow_d \widehat{\mathbf{f}} &=
(d-2) \widehat{\mathbf{L}}_{d-2} \widehat{\mathbf{f}} + c_{1} \widehat{\mathbf{L}}_{d-2} \widehat{\mathbf{f}}' + c_2 \widehat{\mathbf{L}}_{d-2} \widehat{\mathbf{L}}_d \widehat{\mathbf{f}} \\&\quad+
\big( (c_{1}' c_{11}^{d-2} + c_{1}'' c_{12}) \widehat{f}_{1}' + 2c_{1}'c_{12}\widehat{f}_{1}'' + c_{1}'c_{21}\widehat{f}_2'  \big) \widehat{\mathbf{e}}_2 \\&\quad + \big( (c_2' c_{11}^{d-2} + c_2'' c_{12}) \widehat{f}_2 + 2c_2'c_{12}\widehat{f}_2' + c_2'c_{21} (c_{11}^d \widehat{f}_{1}' + c_{12} \widehat{f}_{1}'' + c_{20}^d \widehat{f}_2 + c_{21} \widehat{f}_2')  \big) \widehat{\mathbf{e}}_2 \\&=
(d-2) \widehat{\mathbf{L}}_d \widehat{\mathbf{f}} + c_{1} \widehat{\mathbf{L}}_d \widehat{\mathbf{f}}' + c_2 \widehat{\mathbf{L}}_d \widehat{\mathbf{L}}_d  \widehat{\mathbf{f}} \\&\quad+
\big(
((d-2) c_3 + c_{1}'' c_{12} + c_{1}' c_{11}^{d-2} + \left( c_{21} c_2' + c_2 c_4 \right) c_{11}^d) \widehat{f}_{1}' \\&\quad+
(2 c_{1}' c_{12} + c_{1} c_3 + \left( c_{21} c_2' + c_2 c_4 \right) c_{12}) \widehat{f}_{1}'' \\&\quad+
((d-2) c_4 + c_2'' c_{12} + c_2' c_{11}^{d-2} + \left( c_{21} c_2' + c_2 c_4 \right) c_{20}^d ) \widehat{f}_2 \\&\quad+
(c_{1} c_4 + c_2 c_3 +  2 c_2' c_{12} + c_{1}' c_{21} + \left( c_{21} c_2' + c_2 c_4 \right) c_{21}) \widehat{f}_2' \big) \widehat{\mathbf{e}}_2 \,.
\end{align*}
Now the relation holds by the first part of \Cref{intertwiningCoeff}.
\item We compute
\begin{equation*}
\widehat{\mathbf{L}}_3 \widehat{\mathbf{D}}_d^\downarrow \widehat{\mathbf{f}} = \eta \widehat{\mathbf{L}}_{1} \widehat{\mathbf{f}} - \big( \eta c_3 \widehat{f}_{1}' + \eta(c_4 + 1) \widehat{f}_2 \big) \widehat{\mathbf{e}}_2
\end{equation*}
and
\begin{align*}
\widehat{\mathbf{L}}_{1} \widehat{\mathbf{D}}_3^\downarrow \widehat{\mathbf{f}} + \widehat{\mathbf{D}}_3^\downarrow \widehat{\mathbf{f}} = \eta \widehat{\mathbf{L}}_{1} \widehat{\mathbf{f}} +
\big( (c_{11}^1 - \eta - \eta c_{20}^1 - c_{21} ) \widehat{f}_{1} + (2c_{12} - \eta c_{21}) \widehat{f}_{1}' + (c_{21} + \eta) \widehat{f}_2 \big) \widehat{\mathbf{e}}_2 \,.
\end{align*}
Comparing coefficients on both sides reveals the identity with aid of the second part of \Cref{intertwiningCoeff}.
\qedhere
\end{enumerate}
\end{proof}
\subsection{Analysis of the descent operators}
We continue with an analysis of the previously defined descent operators, where the crucial part is to understand their mapping properties.
\begin{proposition}
\label{majorresult}
Let $R>0$ and $d,k\in\NN$ such that $d\geq 5$ is odd.
\begin{enumerate}[itemsep=1em, topsep=1em]
\item The descent operator $\mathbf{D}_d^\downarrow$ is bounded and extends to a boundedly invertible operator in $\mathfrak{L}\big( \HH^{k+1}_\rad(\BB^{d}_{R}) , \HH^{k}_\rad(\BB_{R}^{d-2}) \big)$, in particular
\begin{equation}
\| \mathbf{D}_d^\downarrow\mathbf{f} \|_{\HH^{k}(\BB_{R}^{d-2})} \simeq \| \mathbf{f} \|_{\HH^{k+1}(\BB^{d}_{R})}
\end{equation}
for all $\mathbf{f} \in \HH^{k+1}_\rad(\BB^{d}_{R})$.
\item The descent operator $\mathbf{D}_3^\downarrow$ is bounded and extends to a boundedly invertible operator in $\mathfrak{L}\big( \HH^{k}_\rad(\BB_{R}^3) , \HH^{k}_\odd(\BB_{R}) \big)$, in particular
\begin{equation}
\| \mathbf{D}_3^\downarrow\mathbf{f} \|_{\HH^{k}(\BB_{R})} \simeq \| \mathbf{f} \|_{\HH^{k}(\BB_{R}^3)}
\end{equation}
for all $\mathbf{f} \in \HH^{k}_\rad(\BB_{R}^3)$.
\end{enumerate}
\end{proposition}
\begin{proof}
\begin{enumerate}[wide, itemsep=1em, topsep=1em]
\item
\begin{itemize}[wide]
\item[``$\lesssim$'':]
Note that the coefficients in the descent operator
\begin{equation*}
\widehat{\mathbf{D}}^\downarrow_d \widehat{\mathbf{f}} =
\begin{bmatrix}
(d-2)\widehat{f}_{1} + c_{1} \widehat{f}_{1}' + c_2\widehat{f}_2 \\
c_2c_{11}^d\widehat{f}_{1}'
+c_2c_{12}\widehat{f}_{1}''
+\big( c_2c_{20}^d + (d-2) \big)\widehat{f}_2
+\big( c_{1} + c_2c_{21} \big) \widehat{f}_2'
\end{bmatrix}
\end{equation*}
behave like
\begin{align*}
c_{1} &= \eta \varphi_{1}  \,,  & & &
c_2 &= \eta^{2} \varphi_2 \,, \\
c_2c_{11}^d &= \eta\varphi_{11}^d \,, &
c_2c_{12} &= \eta^{2}\varphi_{12} \,, &
c_2c_{20}^d + (d-2) &= \varphi_{20}^d \,, &
c_{1} + c_2c_{21} = \eta\varphi_{21} \,,
\end{align*}
on $\overline{\BB_{R}}$, where $\varphi_{1}, \varphi_2, \eta\varphi_{11}^d, \varphi_{12}, \varphi_{20}^d, \varphi_{21} \in C^\infty(\overline{\BB_{R}})$ are even and nonzero at the origin. As derivatives are controlled by corresponding Sobolev norms, we get
\begin{align*}
\big\| |\,.\,|^{\frac{d-3}{2}} [\widehat{\mathbf{D}}_d^\downarrow\widehat{\mathbf{f}}]_{1} \big\|_{H^{k}(\BB_{R})} &\lesssim
\big\| |\,.\,|^{\frac{d-3}{2}} \widehat{f}_{1} \big\|_{H^{k}(\BB_{R})} +
\big\| |\,.\,|^{\frac{d-3}{2}} |\,.\,| \widehat{f}_{1}' \big\|_{H^{k}(\BB_{R})} +
\big\| |\,.\,|^{\frac{d-3}{2}} |\,.\,|^{2} \widehat{f}_2 \big\|_{H^{k}(\BB_{R})}
\\&=
\big\| |\,.\,|^{\frac{d-1}{2} - 1} \widehat{f}_{1} \big\|_{H^{k}(\BB_{R})} +
\big\| |\,.\,|^{\frac{d-1}{2}} \widehat{f}_{1}' \big\|_{H^{k}(\BB_{R})} +
\big\| |\,.\,| |\,.\,|^{\frac{d-1}{2}} \widehat{f}_2 \big\|_{H^{k}(\BB_{R})}
\\&\lesssim
\big\| |\,.\,|^{\frac{d-1}{2}} \widehat{f}_{1} \big\|_{H^{k+1}(\BB_{R})} +
\big\| |\,.\,|^{\frac{d-1}{2}} \widehat{f}_2 \big\|_{H^{k}(\BB_{R})}
\end{align*}
and
\begin{align*}
\big\| |\,.\,|^{\frac{d-3}{2}} [\widehat{\mathbf{D}}_d^\downarrow\widehat{\mathbf{f}}]_2 \big\|_{H^{k-1}(\BB_{R})} &\lesssim
\big\| |\,.\,|^{\frac{d-3}{2}} |\,.\,| \widehat{f}_{1}' \big\|_{H^{k-1}(\BB_{R})} +
\big\| |\,.\,|^{\frac{d-3}{2}} |\,.\,|^{2} \widehat{f}_{1}'' \big\|_{H^{k-1}(\BB_{R})} \\&\quad+
\big\| |\,.\,|^{\frac{d-3}{2}} \widehat{f}_2 \big\|_{H^{k-1}(\BB_{R})} +
\big\| |\,.\,|^{\frac{d-3}{2}} |\,.\,| \widehat{f}_2' \big\|_{H^{k-1}(\BB_{R})} \\&=
\big\| |\,.\,|^{\frac{d-1}{2}} \widehat{f}_{1}' \big\|_{H^{k-1}(\BB_{R})} +
\big\| |\,.\,| |\,.\,|^{\frac{d-1}{2}} \widehat{f}_{1}'' \big\|_{H^{k-1}(\BB_{R})} \\&\quad+
\big\| |\,.\,|^{\frac{d-1}{2} - 1} \widehat{f}_2 \big\|_{H^{k-1}(\BB_{R})} +
\big\| |\,.\,|^{\frac{d-1}{2}} \widehat{f}_2' \big\|_{H^{k-1}(\BB_{R})} \\&\lesssim
\big\| |\,.\,|^{\frac{d-1}{2}} \widehat{f}_{1} \big\|_{H^{k+1}(\BB_{R})} +
\big\| |\,.\,|^{\frac{d-1}{2}} \widehat{f}_2 \big\|_{H^{k}(\BB_{R})} \,.
\end{align*}
In combination with \Cref{oddSobolevBalls} this implies the bound.
\item[``$\gtrsim$'':] We prove that $\mathbf{D}^\downarrow_d$ has dense range $C^\infty_\rad(\overline{\BB_{R}^{d-2}})^{2}$, is injective and conclude with this the reverse bound. Let $\widehat{\mathbf{g}} \in C^\infty(\overline{\BB_{R}})^{2}$ be even and consider solving $\widehat{\mathbf{D}}_d^\downarrow \widehat{\mathbf{f}} = \widehat{\mathbf{g}}$ for an even $\widehat{\mathbf{f}}\in C^\infty(\overline{\BB_{R}})^{2}$, i.e.
\begin{equation*}
\renewcommand{\arraystretch}{1.2}
\left\{
\begin{array}{rcl}
(d-2) \widehat{f}_{1} + c_{1} \widehat{f}_{1}' + c_2 \widehat{f}_2 &=& \widehat{g}_{1} \,, \\
c_2c_{11}^d\widehat{f}_{1}'
+c_2c_{12}\widehat{f}_{1}''
+\big( c_2c_{20}^d + (d-2) \big)\widehat{f}_2
+\big( c_{1} + c_2c_{21} \big) \widehat{f}_2' &=& \widehat{g}_2 \,.
\end{array}
\right.
\end{equation*}
Putting the first equation in the second equation, we see that the first order system is equivalent to the second order ordinary differential equation
\begin{equation*}
\renewcommand{\arraystretch}{1.2}
\left\{
\begin{array}{rcl}
\widehat{f}_{1}'' + p\widehat{f}_{1}' + q\widehat{f}_{1} &=& r_{10}^d \widehat{g}_{1} + r_{11} \widehat{g}_{1}' + r_{20} \widehat{g}_2 \eqqcolon \widehat{G} \,, \\
\widehat{f}_2 &=& \frac{1}{c_2} \widehat{g}_{1} - \frac{d-2}{c_2} \widehat{f}_{1} - \frac{c_{1}}{c_2} \widehat{f}_{1}' \,,
\end{array}
\right.
\end{equation*}
with
\begin{equation*}
p(\eta) = \frac{1}{\eta} \left( 2(d-2) - \frac{\eta}{h'(\eta)} h''(\eta) \right) \,, \qquad
q(\eta) = \frac{d-2}{\eta^{2}} \left( (d-3) - \frac{\eta}{h'(\eta)} h''(\eta) \right) \,,
\end{equation*}
and on the right-hand side
\begin{align*}
r_{10}^d(\eta) &= \frac{1}{\eta^{2}} \left( (d-3) - \frac{\eta}{h'(\eta)} h''(\eta) \right) \,, &
r_{11}(\eta) &= \frac{1}{\eta}\frac{2\eta h'(\eta) - h(\eta) - h(\eta)h'(\eta)^{2}}{ \eta h'(\eta) - h(\eta)} \,, \\
r_{20}(\eta) &= \frac{1- h'(\eta)^{2}}{\eta h'(\eta) - h(\eta)} \frac{h'(\eta)}{\eta} \,. &&
\end{align*}
This differential equation has two fundamental solutions that are given explicitly by
\begin{equation*}
\widehat{\phi}_{11}(\eta) = \frac{h(\eta)}{\eta^{d-2}} \,, \qquad
\widehat{\phi}_{12}(\eta) = \frac{1}{\eta^{d-2}} \,.
\end{equation*}
By defining the second components of the fundamental solutions,
\begin{align*}
\widehat{\phi}_{21}(\eta) &\coloneqq \frac{d-2}{c_2(\eta)} \widehat{\phi}_{11}(\eta) - \frac{c_{1}(\eta)}{c_2(\eta)} \widehat{\phi}_{11}'(\eta) & \widehat{\phi}_{22}(\eta) &\coloneqq \frac{d-2}{c_2(\eta)} \widehat{\phi}_{12}(\eta) - \frac{c_{1}(\eta)}{c_2(\eta)} \widehat{\phi}_{12}'(\eta) \\  &=
(d-3)\frac{h(\eta)}{\eta^{d-2}} \,, & &= \frac{d-2}{\eta^{d-2}} \,,
\end{align*}
the full solution to the inhomogeneous problem is given by Duhamel's formula and reads
\begin{align*}
\widehat{f}_{1}(\eta) &= \alpha_{1}  \widehat{\phi}_{11}(\eta) + \alpha_2  \widehat{\phi}_{12}(\eta) -  \widehat{\phi}_{11}(\eta) \int_{0}^\eta \frac{\widehat{\phi}_{12}(\eta')}{\widehat{W}(\eta')} \widehat{G}(\eta') \dd \eta' +  \widehat{\phi}_{12}(\eta) \int_{0}^\eta \frac{\widehat{\phi}_{11}(\eta')}{\widehat{W}(\eta')} \widehat{G}(\eta') \dd \eta' \,, \\
\widehat{f}_2(\eta) &= \alpha_{1} \widehat{\phi}_{21}(\eta) + \alpha_2 \widehat{\phi}_{22} -  \widehat{\phi}_{21}(\eta) \int_{0}^\eta \frac{\widehat{\phi}_{12}(\eta')}{\widehat{W}(\eta')} \widehat{G}(\eta') \dd \eta' + \widehat{\phi}_{22}(\eta) \int_{0}^\eta \frac{\widehat{\phi}_{11}(\eta')}{\widehat{W}(\eta')} \widehat{G}(\eta') \dd \eta' - \frac{\widehat{g}_{1}(\eta)}{c_2(\eta)} \,,
\end{align*}
where
\begin{equation*}
\widehat{W}(\eta) = \widehat{\phi}_{11}(\eta)\widehat{\phi}_{12}'(\eta) - \widehat{\phi}_{11}'(\eta)\widehat{\phi}_{12}(\eta) = - \frac{h'(\eta)}{\eta^{2(d-2)}}
\end{equation*}
is the Wronskian of $\widehat{\phi}_{11}$, $\widehat{\phi}_{12}$. The next important task is to group the solution advantageously in terms of integral operators and exploit \Cref{integralSobolev}. We have
\begin{equation*}
\widehat{f}_{1} = \alpha_{1} \widehat{\phi}_{11} + \alpha_2 \widehat{\phi}_{12} + \widehat{T}_{11} \widehat{g}_{1} + \widehat{T}_{12} \widehat{g}_2 \,,
\end{equation*}
where we set
\begin{align*}
\widehat{T}_{11} \widehat{g}_{1}(\eta) &= -\widehat{\phi}_{11}(\eta)
\int_{0}^\eta \frac{\widehat{\phi}_{12}(\eta')r_{10}^d(\eta')}{\widehat{W}(\eta')} \widehat{g}_{1}(\eta') \dd\eta' + \widehat{\phi}_{12}(\eta)
\int_{0}^\eta \frac{\widehat{\phi}_{11}(\eta')r_{10}^d(\eta')}{\widehat{W}(\eta')} \widehat{g}_{1}(\eta') \dd\eta' \\&\quad
-\widehat{\phi}_{11}(\eta)
\int_{0}^\eta \frac{\widehat{\phi}_{12}(\eta')r_{11}(\eta')}{\widehat{W}(\eta')} \widehat{g}_{1}'(\eta') \dd\eta' + \widehat{\phi}_{12}(\eta)
\int_{0}^\eta \frac{\widehat{\phi}_{11}(\eta')r_{11}(\eta')}{\widehat{W}(\eta')} \widehat{g}_{1}'(\eta') \dd\eta' \\&=
-\widehat{\phi}_{11}(\eta)
\int_{0}^\eta \left( \frac{\widehat{\phi}_{12}(\eta')r_{10}^d(\eta')}{\widehat{W}(\eta')} - \pd_{\eta'} \frac{\widehat{\phi}_{12}(\eta')r_{11}(\eta')}{\widehat{W}(\eta')} \right) \widehat{g}_{1}(\eta') \dd\eta' \\&\quad+
\widehat{\phi}_{12}(\eta)
\int_{0}^\eta \left( \frac{\widehat{\phi}_{11}(\eta')r_{10}^d(\eta')}{\widehat{W}(\eta')} - \pd_{\eta'} \frac{\widehat{\phi}_{11}(\eta')r_{11}(\eta')}{\widehat{W}(\eta')} \right)\widehat{g}_{1}(\eta') \dd\eta' \\&=
-\widehat{\phi}_{11}(\eta)
\int_{0}^\eta t_{11}^d(\eta') {\eta'}^{d-3} \widehat{g}_{1}(\eta') \dd\eta' +
\widehat{\phi}_{12}(\eta)
\int_{0}^\eta t_{12}^d(\eta') {\eta'}^{d-3} \widehat{g}_{1}(\eta') \dd\eta' \,,
\end{align*}
as it follows from integration by parts and the fact that the integrand vanishes at the origin as long as $d\geq 5$, and
\begin{align*}
\widehat{T}_{12} \widehat{g}_2(\eta) &= -\widehat{\phi}_{11}(\eta)
\int_{0}^\eta \frac{\widehat{\phi}_{12}(\eta')r_{20}(\eta')}{\widehat{W}(\eta')} \widehat{g}_2(\eta') \dd\eta' + \widehat{\phi}_{12}(\eta)
\int_{0}^\eta \frac{\widehat{\phi}_{11}(\eta')r_{20}(\eta')}{\widehat{W}(\eta')} \widehat{g}_2(\eta') \dd\eta' \\&=-
\widehat{\phi}_{11}(\eta) \int_{0}^\eta t_{21}(\eta') {\eta'}^{d-3} \widehat{g}_2(\eta') \dd\eta' + \widehat{\phi}_{12}(\eta)
\int_{0}^\eta t_{22}(\eta') {\eta'}^{d-3} \widehat{g}_2(\eta') \dd\eta' \,,
\end{align*}
with even functions $t_{11}^d, t_{12}^d, t_{21}, t_{22} \in C^\infty(\overline{\BB_{R}})$, each nonzero at the origin, and given by
\begin{align*}
t_{11}^d(\eta) &= \frac{1-h'(\eta)^{2}}{\eta h'(\eta)-h(\eta)} + (d-3) \frac{\eta-h(\eta)h'(\eta)}{ \eta h'(\eta)-h(\eta)} \frac{1}{\eta} - \frac{\eta^{2} - h(\eta)^{2}}{ \big( \eta h'(\eta)-h(\eta) \big)^{2}} h''(\eta) \,, \\
t_{12}^d(\eta) &= 2\frac{\eta - h(\eta)h'(\eta)}{\eta h'(\eta)-h(\eta)}h'(\eta) + (d-3) \frac{\eta-h(\eta)h'(\eta)}{ \eta h'(\eta)-h(\eta)} \frac{h(\eta)}{\eta} - \frac{\eta^{2} - h(\eta)^{2}}{ \big( \eta h'(\eta)-h(\eta) \big)^{2}} h(\eta) h''(\eta) \,, \\
t_{21}(\eta) &= -\frac{1-h'(\eta)^{2}}{\eta h'(\eta) - h(\eta)} \,, \\
t_{22}(\eta) &= -\frac{1-h'(\eta)^{2}}{\eta h'(\eta) - h(\eta)} h(\eta) \,.
\end{align*}
Moreover,
\begin{equation*}
\widehat{f}_2 = \alpha_{1}  \widehat{\phi}_{21} + \alpha_2  \widehat{\phi}_{22} + \widehat{T}_{21} \widehat{g}_{1} + \widehat{T}_{22} \widehat{g}_2 \,,
\end{equation*}
with
\begin{align*}
\widehat{T}_{21} \widehat{g}_{1} (\eta) &=-
\widehat{\phi}_{21}(\eta) \int_{0}^\eta t_{11}^d(\eta') {\eta'}^{d-3} \widehat{g}_{1}(\eta') \dd\eta' +
\widehat{\phi}_{22}(\eta) \int_{0}^\eta t_{12}^d(\eta') {\eta'}^{d-3} \widehat{g}_{1}(\eta') \dd\eta' \\&\quad+
\frac{h(\eta)^{2} - \eta^{2}}{\eta h'(\eta) - h(\eta)} \frac{h'(\eta)}{\eta} \widehat{g}_{1}(\eta)
\end{align*}
and
\begin{equation*}
\widehat{T}_{22} \widehat{g}_2 (\eta) =-
\widehat{\phi}_{21}(\eta) \int_{0}^\eta t_{21}(\eta') {\eta'}^{d-3} \widehat{g}_2(\eta') \dd\eta' +
\widehat{\phi}_{22}(\eta) \int_{0}^\eta t_{22}(\eta') {\eta'}^{d-3} \widehat{g}_2(\eta') \dd\eta' \,.
\end{equation*}
At this point we get from the first part of \Cref{integralSobolev} that $\widehat{T}_{11}\widehat{g}_{1}, \widehat{T}_{12}\widehat{g}_2, \widehat{T}_{21}\widehat{g}_{1}, \widehat{T}_{22}\widehat{g}_2$ define even and smooth functions on $\overline{\BB_{R}}$. Since no nontrivial linear combination of $[\, \widehat{\phi}_{11}, \widehat{\phi}_{21} \,]$, $[\, \widehat{\phi}_{12}, \widehat{\phi}_{22} \,]$ belongs to $C^\infty(\overline{\BB_{R}})^{2}$ this implies $\alpha_{1} = \alpha_2 = 0$. This already shows that $\mathbf{D}^\downarrow_d$ is injective and
\begin{equation}
\label{generalsol}
\widehat{\mathbf{f}} = \widehat{\mathbf{T}}\widehat{\mathbf{g}} \coloneqq
\begin{bmatrix}
\widehat{T}_{11}\widehat{g}_{1} + \widehat{T}_{12}\widehat{g}_2 \\
\widehat{T}_{21}\widehat{g}_{1} + \widehat{T}_{22}\widehat{g}_2
\end{bmatrix}
\,.
\end{equation}
The last step for this direction in the proof consists of proving bounds for the operators $\widehat{T}_{11}$, $\widehat{T}_{12}$, $\widehat{T}_{21}$, $\widehat{T}_{22}$ in the respective Sobolev norms, where we exploit that Sobolev norms for radial functions can be described in terms of weighted Sobolev norms for their radial representative, see \Cref{oddSobolevBalls}. According to what we found above we consider
\begin{equation*}
\big\| |\,.\,|^{\frac{d-1}{2}} \widehat{f}_{1} \big\|_{H^{k+1}(\BB_{R})} \leq  \big\| |\,.\,|^{\frac{d-1}{2}} \widehat{T}_{11} \widehat{g}_{1} \big\|_{H^{k+1}(\BB_{R})} + \big\| |\,.\,|^{\frac{d-1}{2}} \widehat{T}_{12} \widehat{g}_2 \big\|_{H^{k+1}(\BB_{R})}
\end{equation*}
and
\begin{align*}
\big\| |\,.\,|^{\frac{d-1}{2}} \widehat{f}_2 \big\|_{H^{k}(\BB_{R})} &\leq \big\| |\,.\,|^{\frac{d-1}{2}}  \widehat{T}_{21} \widehat{g}_{1} \big\|_{H^{k}(\BB_{R})} + \big\| |\,.\,|^{\frac{d-1}{2}}  \widehat{T}_{22} \widehat{g}_2 \big\|_{H^{k}(\BB_{R})} \\&\lesssim
\big\| |\,.\,|^{\frac{d-1}{2}}  \widehat{T}_{21} \widehat{g}_{1} \big\|_{H^{k}(\BB_{R})} \\&+
\big\| |\,.\,|^{\frac{d-1}{2}}  \widehat{T}_{22} \widehat{g}_2 \big\|_{L^{2}(\BB_{R})} + \big\| |\,.\,|^{\frac{d-3}{2}}  \widehat{T}_{22} \widehat{g}_2 \big\|_{H^{k-1}(\BB_{R})} + \big\| |\,.\,|^{\frac{d-1}{2}}  (\widehat{T}_{22} \widehat{g}_2)' \big\|_{H^{k-1}(\BB_{R})} \,.
\end{align*}
In order to make sense of these inequalities we bound the norms on the right-hand sides.
\begin{enumerate}[wide, itemsep=1em, topsep=1em]
\item[\textit{Bound for $\widehat{T}_{11}$.}] Consider
\begin{equation*}
\big\| |\,.\,|^{\frac{d-1}{2}} \widehat{T}_{11} \widehat{g}_{1} \big\|_{H^{k+1}(\BB_{R})} \lesssim \big\| |\,.\,|^{\frac{d-1}{2}} \widehat{T}_{11} \widehat{g}_{1} \big\|_{L^{2}(\BB_{R})} + \big\| |\,.\,|^{\frac{d-3}{2}} \widehat{T}_{11}\widehat{g}_{1} \big\|_{H^{k}(\BB_{R})} + \big\| |\,.\,|^{\frac{d-3}{2}} (\widehat{T}_{11}\widehat{g}_{1})' \big\|_{H^{k}(\BB_{R})} \,.
\end{equation*}
So,
\begin{align*}
\eta^{\frac{d-1}{2}}\widehat{T}_{11} \widehat{g}_{1}(\eta) &= \frac{h(\eta)}{\eta^{\frac{d-3}{2}}} \int_{0}^\eta {\eta'}^{\frac{d-3}{2}} t_{11}(\eta') {\eta'}^{\frac{d-3}{2}} \widehat{g}_{1} (\eta') \dd\eta' +  \frac{1}{\eta^{\frac{d-3}{2}}} \int_{0}^\eta {\eta'}^{\frac{d-3}{2}} t_{12}(\eta') {\eta'}^{\frac{d-3}{2}} \widehat{g}_{1} (\eta') \dd\eta' \,, \\
\eta^{\frac{d-3}{2}}\widehat{T}_{11} \widehat{g}_{1}(\eta) &= \frac{h(\eta)}{\eta^{\frac{d-1}{2}}} \int_{0}^\eta {\eta'}^{\frac{d-3}{2}} t_{11}(\eta') {\eta'}^{\frac{d-3}{2}} \widehat{g}_{1} (\eta') \dd\eta' + \frac{1}{\eta^{\frac{d-1}{2}}} \int_{0}^\eta {\eta'}^{\frac{d-3}{2}} t_{12}(\eta') {\eta'}^{\frac{d-3}{2}} \widehat{g}_{1} (\eta') \dd\eta' \,, \\
\eta^{\frac{d-1}{2}} (\widehat{T}_{11} \widehat{g}_{1})'(\eta) &=
-\frac{\eta h'(\eta) - (d-2)h(\eta)}{\eta^{\frac{d-1}{2}}}
\int_{0}^\eta {\eta'}^{\frac{d-3}{2}} t_{11}(\eta') {\eta'}^{\frac{d-3}{2}} \widehat{g}_{1} (\eta') \dd\eta' \\&\quad-
\frac{d-2}{\eta^{\frac{d-1}{2}}} \int_{0}^\eta {\eta'}^{\frac{d-3}{2}} t_{12}(\eta') {\eta'}^{\frac{d-3}{2}} \widehat{g}_{1} (\eta') \dd\eta' +
\eta r_{11}(\eta)\eta^{\frac{d-3}{2}} \widehat{g}_{1}(\eta) \,,
\end{align*}
and by an application of the second part of \Cref{integralSobolev} we conclude immediately
\begin{equation}
\label{T1estimate}
\big\| |\,.\,|^{\frac{d-1}{2}} \widehat{T}_{11} \widehat{g}_{1} \big\|_{H^{k+1}(\BB_{R})} \lesssim \big\| |\,.\,|^{\frac{d-3}{2}} \widehat{g}_{1} \big\|_{H^{k}(\BB_{R})} \,.
\end{equation}
\item[\textit{Bound for $\widehat{T}_{12}$.}] Now we turn to the norm
\begin{align*}
&\big\| |\,.\,|^{\frac{d-1}{2}} \widehat{T}_{12} \widehat{g}_2 \big\|_{H^{k+1}(\BB_{R})} \lesssim \big\| |\,.\,|^{\frac{d-1}{2}} \widehat{T}_{12} \widehat{g}_2 \big\|_{L^{2}(\BB_{R})} + \big\| |\,.\,|^{\frac{d-3}{2}} \widehat{T}_{12} \widehat{g}_2 \big\|_{H^{k}(\BB_{R})} + \big\| |\,.\,|^{\frac{d-1}{2}} (\widehat{T}_{12} \widehat{g}_2)' \big\|_{H^{k}(\BB_{R})} \\ &\lesssim
\big\| |\,.\,|^{\frac{d-1}{2}} \widehat{T}_{12} \widehat{g}_2 \big\|_{L^{2}(\BB_{R})} \\&+
\big\| |\,.\,|^{\frac{d-5}{2}} \widehat{T}_{12} \widehat{g}_2 \big\|_{H^{k-1}(\BB_{R})} + \big\| |\,.\,|^{\frac{d-3}{2}} (\widehat{T}_{12} \widehat{g}_2)' \big\|_{H^{k-1}(\BB_{R})} + \big\| |\,.\,|^{\frac{d-1}{2}} (\widehat{T}_{12} \widehat{g}_2)'' \big\|_{H^{k-1}(\BB_{R})} \,.
\end{align*}
An application of \Cref{integralSobolev} to
\begin{equation*}
\eta^{\frac{d-1}{2}} \widehat{T}_{12} \widehat{g}_2(\eta) = -\frac{h(\eta)}{\eta^{\frac{d-3}{2}}}
\int_{0}^\eta {\eta'}^{\frac{d-3}{2}} t_{21}(\eta') {\eta'}^{\frac{d-3}{2}} \widehat{g}_2 (\eta') \dd\eta' +
\frac{1}{\eta^{\frac{d-3}{2}}} \int_{0}^\eta {\eta'}^{\frac{d-3}{2}} t_{22}(\eta') {\eta'}^{\frac{d-3}{2}} \widehat{g}_2 (\eta') \dd\eta'
\end{equation*}
yields a weighted bound in $L^{2}(\BB_{R})$. To obtain bounds in $H^{k-1}(\BB_{R})$, we note $h(\eta) - h(0) = \eta^{2} \varphi(\eta)$ for some even $\varphi\in C^\infty(\overline{\BB_{R}})$ with $\varphi(0) \neq 0$ and since $t_{22} = h t_{21}$ we get
\begin{equation*}
\widehat{T}_{12} \widehat{g}_2(\eta) = -\frac{\varphi(\eta)}{\eta^{d-4}}
\int_{0}^\eta t_{21}(\eta') {\eta'}^{d-3} \widehat{g}_2 (\eta') \dd\eta' +
\frac{1}{\eta^{d-2}} \int_{0}^\eta \varphi(\eta') t_{21}(\eta') {\eta'}^{d-1} \widehat{g}_2 (\eta') \dd\eta' \,.
\end{equation*}
Now we have better control on the singularities and see
\begin{align*}
\eta^{\frac{d-5}{2}} \widehat{T}_{12} \widehat{g}_2(\eta) &= -\frac{\varphi(\eta)}{\eta^{\frac{d-3}{2}}}
\int_{0}^\eta {\eta'}^{\frac{d-3}{2}} t_{21}(\eta') {\eta'}^{\frac{d-3}{2}} \widehat{g}_2 (\eta') \dd\eta' \\&\quad+
\frac{1}{\eta^{\frac{d+1}{2}}} \int_{0}^\eta {\eta'}^{\frac{d+1}{2}} \varphi(\eta') t_{21}(\eta') {\eta'}^{\frac{d-3}{2}} \widehat{g}_2 (\eta') \dd\eta' \,, \\
\eta^{\frac{d-3}{2}} (\widehat{T}_{12} \widehat{g}_2)'(\eta) &= -\frac{\eta\varphi'(\eta) - (d-4)\varphi(\eta)}{\eta^{\frac{d-3}{2}}} \int_{0}^\eta {\eta'}^{\frac{d-3}{2}} t_{21}(\eta') {\eta'}^{\frac{d-3}{2}} \widehat{g}_2 (\eta') \dd\eta' \\&\quad-
\frac{d-2}{\eta^{\frac{d+1}{2}}} \int_{0}^\eta {\eta'}^{\frac{d+1}{2}} t_{22}(\eta') {\eta'}^{\frac{d-3}{2}} \widehat{g}_2 (\eta') \dd\eta' \,, \\
\eta^{\frac{d-1}{2}} (\widehat{T}_{12} \widehat{g}_2)''(\eta) &= -\frac{\eta^{2} \varphi(\eta) - (d-4)(2\eta\varphi'(\eta) + (d-3)\varphi(\eta))}{\eta^{\frac{d-3}{2}}} \int_{0}^\eta {\eta'}^{\frac{d-3}{2}} t_{21}(\eta') {\eta'}^{\frac{d-3}{2}} \widehat{g}_2 (\eta') \dd\eta' \\&\quad+  \frac{(d-2)(d-1)}{\eta^{\frac{d+1}{2}}} \int_{0}^\eta {\eta'}^{\frac{d+1}{2}} t_{22}(\eta') {\eta'}^{\frac{d-3}{2}} \widehat{g}_2 (\eta') \dd\eta' \\&\quad-
\eta(\eta\varphi'(\eta)-2h(\eta))t_{21}(\eta) \eta^{\frac{d-3}{2}} \widehat{g}_2(\eta) \,.
\end{align*}
This implies again with \Cref{integralSobolev}
\begin{equation}
\label{T2estimate}
\big\| |\,.\,|^{\frac{d-1}{2}} \widehat{T}_{12} \widehat{g}_2 \big\|_{H^{k+1}(\BB_{R})} \lesssim \big\| |\,.\,|^{\frac{d-3}{2}} \widehat{g}_2 \big\|_{H^{k-1}(\BB_{R})} \,.
\end{equation}
\item[\textit{Bound for $\widehat{T}_{21}$.}] Let us turn to the estimates for the second component. We start with
\begin{align*}
\eta^{\frac{d-1}{2}} \widehat{T}_{21}\widehat{g}_{1}(\eta) &=
-(d-3)\frac{h(\eta)}{\eta^{\frac{d-3}{2}}} \int_{0}^\eta {\eta'}^{\frac{d-3}{2}} t_{11}(\eta') {\eta'}^{\frac{d-3}{2}} \widehat{g}_{1}(\eta') \dd\eta' \\&\quad+
\frac{d-2}{\eta^{\frac{d-3}{2}}} \int_{0}^\eta {\eta'}^{\frac{d-3}{2}} t_{12}(\eta') {\eta'}^{\frac{d-3}{2}} \widehat{g}_{1}(\eta') \dd\eta' +
\eta \frac{h(\eta)^{2} - \eta^{2}}{\eta h'(\eta) - h(\eta)} \frac{h'(\eta)}{\eta} \eta^{\frac{d-3}{2}} \widehat{g}_{1}(\eta)
\end{align*}
and infer immediately from \Cref{integralSobolev}
\begin{equation}
\label{S1estimate}
\big\| |\,.\,|^{\frac{d-1}{2}} \widehat{T}_{21} \widehat{g}_{1} \big\|_{H^{k}(\BB_{R})} \lesssim \big\| |\,.\,|^{\frac{d-3}{2}} \widehat{g}_{1} \big\|_{H^{k}(\BB_{R})} \,.
\end{equation}
\item[\textit{Bound for $\widehat{T}_{22}$.}] Finally,
\begin{align*}
\eta^{\frac{d-1}{2}} \widehat{T}_{22} \widehat{g}_2(\eta) &= -(d-3)\frac{h(\eta)}{\eta^{\frac{d-3}{2}}} \int_{0}^\eta {\eta'}^{\frac{d-3}{2}} t_{21}(\eta') {\eta'}^{\frac{d-3}{2}} \widehat{g}_2(\eta') \dd \eta' \\&\quad+ \frac{d-2}{\eta^{\frac{d-3}{2}}} \int_{0}^\eta {\eta'}^{\frac{d-3}{2}} t_{22}(\eta') {\eta'}^{\frac{d-3}{2}} \widehat{g}_2(\eta') \dd \eta' \,, \\
\eta^{\frac{d-3}{2}} \widehat{T}_{22} \widehat{g}_2(\eta) &= -(d-3)\frac{h(\eta)}{\eta^{\frac{d-1}{2}}} \int_{0}^\eta {\eta'}^{\frac{d-3}{2}} t_{21}(\eta') {\eta'}^{\frac{d-3}{2}} \widehat{g}_2(\eta') \dd \eta' \\&\quad+ \frac{d-2}{\eta^{\frac{d-1}{2}}} \int_{0}^\eta {\eta'}^{\frac{d-3}{2}} t_{22}(\eta') {\eta'}^{\frac{d-3}{2}} \widehat{g}_2(\eta') \dd \eta' \,, \\
\eta^{\frac{d-1}{2}} (\widehat{T}_{22} \widehat{g}_2)'(\eta) &= -(d-3)\frac{\eta h'(\eta) - (d-2)h(\eta)}{\eta^{\frac{d-1}{2}}} \int_{0}^\eta {\eta'}^{\frac{d-3}{2}} t_{21}(\eta') {\eta'}^{\frac{d-3}{2}} \widehat{g}_2(\eta') \dd \eta' \\&\quad-
\frac{(d-2)^{2}}{\eta^{\frac{d-1}{2}}} \int_{0}^\eta {\eta'}^{\frac{d-3}{2}} t_{22}(\eta') {\eta'}^{\frac{d-3}{2}} \widehat{g}_2(\eta') \dd \eta' \\&\quad+
t_{22}(\eta) \eta^{\frac{d-3}{2}} \widehat{g}_2(\eta) \,.
\end{align*}
Altogether, this implies
\begin{equation}
\label{S2estimate}
\big\| |\,.\,|^{\frac{d-1}{2}} \widehat{T}_{22}\widehat{g}_2 \big\|_{H^{k}(\BB_{R})} \lesssim \big\| |\,.\,|^{\frac{d-3}{2}} \widehat{g}_2 \big\|_{H^{k-1}(\BB_{R})} \,.
\end{equation}
\end{enumerate}
The estimates \eqref{T1estimate} - \eqref{S2estimate} give together with \Cref{oddSobolevBalls} the desired result.
\end{itemize}
This finishes the proof of the first part.
\item
\begin{itemize}[wide, itemindent=\parindent]
\item[``$\lesssim$'':]
We get directly for the first component
\begin{equation*}
\big\| [\widehat{\mathbf{D}}_3^\downarrow\widehat{\mathbf{f}}]_{1} \big\|_{H^{k}(\BB_{R})} = \big\| |\,.\,|\widehat{f}_{1} \big\|_{H^{k}(\BB_{R})}
\end{equation*}
and for the second component
\begin{equation*}
\big\| [\widehat{\mathbf{D}}_3^\downarrow\widehat{\mathbf{f}}]_2 \big\|_{H^{k-1}(\BB_{R})} \leq \big\| |\,.\,|\widehat{f}_{1} \big\|_{H^{k-1}(\BB_{R})} + \big\| |\,.\,|\widehat{f}_2 \big\|_{H^{k-1}(\BB_{R})} \leq \big\| |\,.\,|\widehat{f}_{1} \big\|_{H^{k}(\BB_{R})} + \big\| |\,.\,|\widehat{f}_2 \big\|_{H^{k-1}(\BB_{R})}
\end{equation*}
With \Cref{oddSobolevBalls} this gives precisely
\begin{equation*}
\| \widehat{\mathbf{D}}_3^\downarrow\widehat{\mathbf{f}} \|_{\HH^{k-1}(\BB_{R}) } \lesssim \| \widehat{\mathbf{f}} \|_{\HH^{k-1}(\BB_{R}^3) } \,.
\end{equation*}
\item[``$\gtrsim$'':]
As before we invert $\mathbf{D}_3^\downarrow$. Let $\widehat{\mathbf{g}}\in C_{\mathrm{odd}}^\infty(\BB_{R})^{2}$ and consider $\widehat{\mathbf{D}}_d^\downarrow \widehat{\mathbf{f}} = \widehat{\mathbf{g}}$, that is,
\begin{equation*}
\renewcommand{\arraystretch}{1.2}
\left\{
\begin{array}{rcl}
y \widehat{f}_{1} &=& \widehat{g}_{1} \,, \\
-y \widehat{f}_{1} + y \widehat{f}_2 &=& \widehat{g}_2 \,.
\end{array}
\right.
\end{equation*}
Since $\widehat{g}_{1}$ is odd we have $\widehat{g}_{1}(0)=0$ and with the fundamental theorem of calculus we infer even functions $\widehat{f}_{1}, \widehat{f}_2 \in C^\infty(\overline{\BB_{R}})$ given by
\begin{equation}
\label{generalsol3}
\widehat{f}_{1}(\eta) = \int_{0}^1 \widehat{g}_{1}'(t\eta) \dd t \,, \qquad \widehat{f}_2(\eta) = \int_{0}^1 \big( \widehat{g}_{1}'(t\eta) + \widehat{g}_2'(t\eta) \big) \dd t \,.
\end{equation}
Now the bound follows similarly from
\begin{equation*}
\big\| |\,.\,| \widehat{f}_{1} \big\|_{H^{k}(\BB_{R})} = \| \widehat{g}_{1} \|_{H^{k}(\BB_{R})}
\end{equation*}
and
\begin{equation*}
\big\| |\,.\,| \widehat{f}_2 \big\|_{H^{k-1}(\BB_{R})} \leq \| \widehat{g}_{1} \|_{H^{k-1}(\BB_{R})} + \| \widehat{g}_2 \|_{H^{k-1}(\BB_{R})} \leq \| \widehat{g}_{1} \|_{H^{k}(\BB_{R})} + \| \widehat{g}_2 \|_{H^{k-1}(\BB_{R})} \,.
\end{equation*}
This translates with \Cref{oddSobolevBalls} to the inequality
\begin{equation*}
\| \widehat{\mathbf{f}} \|_{\HH^{k-1}(\BB_{R}^3) } \lesssim \| \widehat{\mathbf{D}}_3^\downarrow\widehat{\mathbf{f}} \|_{\HH^{k-1}(\BB_{R}) } \,.
\end{equation*}
\end{itemize}
From density and boundedness we infer the bounded and bijective extensions.
\qedhere
\end{enumerate}
\end{proof}
\subsection{Control of the free wave evolution}
With the previous result we can now treat well-posedness for the wave equation in hyperboloidal similarity coordinates in odd space dimensions. Growth estimates are in some sense derived directly from the wave equation itself, namely via the one-dimensional energy bounds at which we arrive with the descent method in steps of two space dimensions.
\begin{proposition}
\label{MScProp}
Let $R>0$, $d,k\in \NN$ such that $d\geq 3$ is odd. There is a bounded and bijective linear operator $\mathbf{D}_d: \HH^{k+\frac{d-3}{2}}_\rad(\BB^{d}_{R}) \rightarrow \HH^{k}_\odd(\BB_{R})$ with
\begin{equation}
\label{MScDescentequivalent}
\| \mathbf{D}_d\mathbf{f} \|_{\HH^{k}(\BB_{R})} \simeq \| \mathbf{f} \|_{\HH^{k+\frac{d-3}{2}}(\BB^{d}_{R})}
\end{equation}
such that
\begin{equation}
\label{MScDescentintertwining}
\mathbf{D}_d \mathbf{L}_d \mathbf{f} - \mathbf{D}_d \mathbf{f} = \mathbf{L}_{1} \mathbf{D}_d \mathbf{f}
\end{equation}
for all $\mathbf{f} \in C^\infty_\rad(\overline{\BB^{d}_{R}})^{2}$.
\end{proposition}
\begin{proof}
By successive application of the descent operators we obtain an operator
\begin{equation*}
\mathbf{D}_d: \HH^{k+\frac{d-3}{2}}_\rad(\BB^{d}_{R}) \rightarrow \HH^{k}_\odd(\BB_{R}) \,, \qquad
\mathbf{D}_d \mathbf{f} \coloneqq  (\mathbf{D}_{3}^\downarrow \circ \ldots \circ \mathbf{D}_{d}^\downarrow) \mathbf{f} \,.
\end{equation*}
From the bounds proved in \Cref{majorresult} we infer \eqref{MScDescentequivalent}. Repeated application of the intertwining identities from \Cref{intertwiningLemma} yields \eqref{MScDescentintertwining}.
\end{proof}
This result enables us to reduce problems in odd dimensions to the one-dimensional wave equation and transfer the solutions back. For instance, the equivalence of Sobolev norms \eqref{MScDescentequivalent} will yield energy bounds and behind the intertwining identity \eqref{MScDescentintertwining} is the reduction of the radial wave equation from $d$ dimensions to one dimension from which we will obtain existence of solutions. With this we arrive at the main result of this section which is a systematic generalization of \cite[Theorem 3.12]{MR4338226} to odd space dimensions.
\begin{theorem}
\label{MScresult}
Let $R\geq \frac{1}{2}$, $d,k\in\NN_{0}$ such that $d\geq 3$ is odd and $k\geq\frac{d-1}{2}$. The operator $\mathbf{L}_d$ is closable and its closure
\begin{equation}
\label{MScgenerator}
\overline{\mathbf{L}_d} = \mathbf{D}_d^{-1}( \overline{\mathbf{L}_{1}} + \mathbf{I})\mathbf{D}_d \,, \quad \dom(\overline{\mathbf{L}_d}) = \big\{ \mathbf{f} \in \HH^{k}_\rad(\BB^{d}_{R}) \mid \mathbf{D}_d\mathbf{f} \in \dom(\overline{\mathbf{L}_{1}}) \} \,,
\end{equation}
is the generator of the rescaled similar semigroup
\begin{equation}
\label{SdSGRP}
\mathbf{S}_d(s) = \ee^s \mathbf{D}_d^{-1} \mathbf{S}_{1}(s) \mathbf{D}_d
\end{equation}
of bounded linear operators $\mathfrak{L} \big( \HH^{k}_\rad(\BB^{d}_{R}) \big)$ such that
\begin{equation}
\label{MScgrowthbound}
\| \mathbf{S}_d(s)\mathbf{f} \|_{\HH^{k}(\BB^{d}_{R})} \lesssim \ee^{s/2} \| \mathbf{f} \|_{\HH^{k}(\BB^{d}_{R})}
\end{equation}
for all $\mathbf{f} \in \HH^{k}_{\mathrm{rad}}(\BB^{d}_{R})$ and all $s\geq 0$.
\end{theorem}
\begin{proof}
Let $\mathbf{S}_{1}$ be the strongly continuous semigroup on $\mathfrak{L}\big( \HH^{k-\frac{d-3}{2}}_{\odd}(\BB_{R}) \big)$ from \Cref{one-dim-semigrp-prop}. Then $\mathbf{S}_d$ forms a strongly continuous semigroup on $\mathfrak{L}\big( \HH^{k}_\mathrm{rad}(\BB^{d}_{R}) \big)$ with the growth bound \eqref{MScgrowthbound} and generator \eqref{MScgenerator}.
\end{proof}
\begin{remark}
If $k\geq \frac{d+5}{2}$ set $\mathbf{v}(s,\,.\,) \coloneqq \mathbf{S}_d(s) \mathbf{f}$ for $\mathbf{f}\in\dom(\overline{\mathbf{L}_d})$ and let $v(s,\,.\,) \coloneqq [\mathbf{v}(s,\,.\,)]_{1}$ for $s\geq 0$. Then $v\in C^{2}([0,\infty)\times\overline{\BB^d_{R}})$, $v(s,\,.\,)$ is radial and a classical solution to the wave equation in hyperboloidal similarity coordinates. Indeed, this is immediate from the Sobolev embedding $H^{k}(\BB^d_{R}) \subset C^{2}(\overline{\BB^d_{R}})$ where $k - \frac{d+1}{2} \geq 2$. Also, this complies with what one would expect from constructing classical solutions in higher dimensions from classical solutions in one dimension via \Cref{SdSGRP} and the involved mapping properties of the descent operators. When a classical solution in one dimension is mapped back with the inverse descent operator, note with Eqs. \eqref{generalsol3}, \eqref{generalsol} that this costs one derivative at the origin. So, if we want to exploit a classical solution from \Cref{S1Classic} we shall provide $k - \frac{d-3}{2} \geq 4$ in order to end up with a classical solution in $d$ dimensions. This is the above condition for Sobolev embedding.
\end{remark}
\section{Hyperboloidal formulation of the Yang-Mills equation}
In this short section we give a formulation of the Yang-Mills equation
\begin{equation}
\label{YM2}
\big( - \pd_{t}^{2} + \Delta_{x} \big) u(t,x) = (d-4)(|x|^{2}u(t,x)^3 + 3u(t,x)^{2}) \,, \quad (t,x) \in \RR^{1,d} \,,
\end{equation}
as an autonomous first-order evolution equation in hyperboloidal similarity coordinates, which will be adapted to a stability analysis of the extended blowup
\begin{equation}
u_{T}^{*}(t,x) = - \frac{a_d}{b_d(T-t)^{2} + |x|^{2}} \,, \quad T\in\RR \,.
\end{equation}
With our methods we are able to treat the stability problem for all supercritical odd space dimensions, so from now on let $d \geq 7$ be an arbitrary but fixed odd integer. To begin with, we insert $u = u_{T}^{*} + \widetilde{u}$ in \Cref{YM2}, where $\widetilde{u}$ is viewed as a perturbation around the blowup. This produces the equation
\begin{equation}
\big( - \pd_{t}^{2} + \Delta_{x} + V_{T}(t,x) \big) \widetilde{u}(t,x) = F_{T}(t,x,\widetilde{u}(t,x))
\end{equation}
satisfied by the perturbation, where
\begin{equation*}
V_{T}(t,x) = - (d-4)\big( 3|x|^{2} u_{T}^{*}(t,x)^{2} + 6u_{T}^{*}(t,x) \big)
\end{equation*}
is a potential and the nonlinearity is given by
\begin{equation*}
F_{T}(t,x,\widetilde{u}) = (d-4)\big( 3\big(
|x|^{2} u_{T}^{*}(t,x) + 1 \big) \widetilde{u}^{2}
+ |x|^{2} \widetilde{u}^3\big) \,.
\end{equation*}
If we transform this to hyperboloidal similarity coordinates via $\widetilde{v} = \widetilde{u}\circ\eta_{T}$ we get
\begin{equation}
\label{HSCYM}
\Box_{\g} \widetilde{v}(s,y) + V_{T}(\eta_{T}(s,y)) \widetilde{v}(s,y) = F_{T}(\eta_{T}(s,y),\widetilde{v}(s,y)) \,.
\end{equation}
To hark back on the wave operator in hyperboloidal similarity coordinates \eqref{secondorderwave}, this may be formulated as
\begin{equation*}
\pd_{0}^{2} \widetilde{v} = \left( c^{ij} \pd_{i}\pd_j + c^{i0} \pd_{i}\pd_{0} + c^0 \pd_{0} + c^i \pd_{i} \right) \widetilde{v} -
\frac{V_{T}(\eta_{T})}{\g^{00}} \widetilde{v} + \frac{F_{T}(\eta_{T},\widetilde{v})}{\g^{00}} \,.
\end{equation*}
Note that
\begin{equation*}
u_{T}^{*}(\eta_{T}(s,y)) = - \ee^{2s} \frac{a_d}{b_d h(y)^{2} + |y|^{2}} \,.
\end{equation*}
Then we observe that
\begin{equation*}
-\frac{V_{T}(\eta_{T}(s,y))}{\g^{00}(s,y)} = V(y)
\end{equation*}
is independent of $s$, where
\begin{equation}
\label{Vpotential}
V(y) = - 3a_d(d-4) \frac{\left( |y|h'(|y|) - h(|y|) \right)^{2}}{1 - h'(|y|)^{2}} \frac{ (a_d-2)|y|^{2} - 2b_d h(|y|)^{2}}{(b_d h(|y|)^{2} + |y|^{2})^{2}}
\end{equation}
is smooth and radial, and
\begin{equation*}
\frac{F_{T}(\eta_{T}(s,y),\widetilde{v}(s,y))}{\g^{00}(s,y)} =
\ee^{2s} N\big(y, \ee^{-2s}\widetilde{v}(s,y)\big)
\end{equation*}
with
\begin{equation}
\label{nonlinYM}
N(y,\alpha) \coloneqq - (d-4) \frac{\left( |y|h'(|y|) - h(|y|) \right)^{2}}{1 - h'(|y|)^{2}} \left( 3\frac{(1-a_d)|y|^{2} + b_d h(|y|)^{2}}{b_d h(|y|)^{2} + |y|^{2}} \alpha^{2} + |y|^{2} \alpha^3 \right) \,.
\end{equation}
With this, we can formulate our Yang-Mills equation equivalently as a first order system
\begin{equation}
\label{YMSG}
\pd_s\widetilde{\mathbf{v}}(s,\,.\,) = (\mathbf{L}_d + \mathbf{L}_V)\widetilde{\mathbf{v}}(s,\,.\,) + \ee^{2s} \mathbf{N}(\ee^{-2s}\widetilde{\mathbf{v}}(s,\,.\,)) \,,
\end{equation}
where $\widetilde{\mathbf{v}}(s,\,.\,) = [\, \widetilde{v}(s,\,.\,), \pd_s \widetilde{v}(s,\,.\,) \,]$ and
\begin{equation}
\label{linnonlin}
\mathbf{L}_V \widetilde{\mathbf{v}}(s,\,.\,) =
V \widetilde{v}(s,\,.\,) \mathbf{e}_2
\,,
\qquad
\mathbf{N}(\widetilde{\mathbf{v}}(s,\,.\,)) =
N(\,.\,,\widetilde{v}(s,\,.\,)) \mathbf{e}_2
\,.
\end{equation}
Hence, in terms of the variable $\mathbf{\Phi}(s) = \ee^{-2s} \widetilde{\mathbf{v}}(s,\,.\,)$ we have a desired autonomous system
\begin{equation}
\label{YMnewvariable}
\pd_s \mathbf{\Phi}(s) = ( \mathbf{L}_d - 2 \mathbf{I} + \mathbf{L}_V ) \mathbf{\Phi}(s) + \mathbf{N}(\mathbf{\Phi}(s)) \,.
\end{equation}
The Yang-Mills flow constitutes a perturbed linear first order system where the blowup time $T\in\RR$ will only enter later through the initial data.
\section{Spectral analysis of the linearized wave evolution}
The first step in the perturbative approach to \Cref{YMnewvariable} is to study the \emph{linearized wave evolution}
\begin{equation*}
\mathbf{L} = \overline{\mathbf{L}_d} - 2 \mathbf{I} + \mathbf{L}_V \,, \qquad \dom(\mathbf{L}) = \dom(\overline{\mathbf{L}_d}) \,,
\end{equation*}
where $\overline{\mathbf{L}_d}$ is the closure of the free wave evolution as presented in \Cref{MScresult} and $\mathbf{L}_V$ is the multiplicative potential from \Cref{linnonlin} that clearly extends from
\begin{equation*}
\mathbf{L}_V\mathbf{f} \coloneqq V f_{1} \mathbf{e}_2
\end{equation*}
for all $\mathbf{f} \in C^\infty_\rad(\overline{\BB^d_{R}})^{2}$ to a bounded linear operator on $\HH^{k}_\rad(\BB^d_{R})$. In the above setting for the Yang-Mills flow, stability of the blowup will emerge from decay of the solution propagator. The latter property hinges on a precise description of the spectrum of the linearized wave evolution $\mathbf{L}$, which shall concern us in this section.
\begin{lemma}
\label{LinEv}
Let $R\geq\frac{1}{2}$ and $k\in\NN$ with $k\geq\frac{d-1}{2}$. The operator $\mathbf{L}$ is the generator of a strongly continuous semigroup
\begin{equation*}
\mathbf{S}: [0,\infty) \rightarrow \mathfrak{L}\big(\HH^{k}_\rad(\BB^{d}_{R})\big) \,.
\end{equation*}
\end{lemma}
\begin{proof}
As $\overline{\mathbf{L}_d}$ is the generator of a strongly continuous semigroup by \Cref{MScresult} and $-2\mathbf{I} + \mathbf{L}_V\in\mathfrak{L}(\HH^{k}_\rad\big(\BB^{d}_{R})\big)$, the bounded perturbation theorem \cite[p. 158]{MR1721989} applies and we obtain that the operator $\mathbf{L}$ is the generator of a strongly continuous semigroup $\mathbf{S}$ of bounded linear operators in $\mathfrak{L}\big(\HH^{k}_\rad(\BB^{d}_{R})\big)$.
\end{proof}
To get suitable growth estimates for this semigroup, a detailed spectral analysis of the generator is in order. Due to its structure and a compact embedding, the potential term turns out to be compact operator.
\begin{lemma}
Let $R>0$ and $k\in\NN$. The operator $\mathbf{L}_V \in \mathfrak{L}\big(\HH^{k}_\rad(\BB^{d}_{R})\big)$ is compact.
\end{lemma}
\begin{proof}
Let $(\mathbf{f}_{i})_{i\in\NN}$ be a bounded sequence in $\HH^{k}_\rad(\BB^{d}_{R})$. Note that
\begin{equation*}
\| \mathbf{L}_V\mathbf{f}_{i} \|_{\HH^{k}(\BB^{d}_{R})} = \| V f_{i,1} \|_{H^{k-1}(\BB^{d}_{R})} \lesssim \| f_{i,1} \|_{H^{k-1}(\BB^{d}_{R})} \,.
\end{equation*}
Since $(f_{i,1})_{i\in\NN}$ is a bounded sequence in $H^{k}_\mathrm{rad}(\BB^{d}_{R})$ we can extract by virtue of the compact embedding $H^{k}_\mathrm{rad}(\BB^{d}_{R}) \hookrightarrow H^{k-1}_\mathrm{rad}(\BB^{d}_{R})$ a convergent subsequence in $H^{k-1}_\mathrm{rad}(\BB^{d}_{R})$. After passing to this subsequence, we may as well assume that $(f_{i,1})_{i\in\NN}$ converges in $H^{k-1}_\mathrm{rad}(\BB^{d}_{R})$. Now the above estimate shows that $\mathbf{L}_V$ is compact.
\end{proof}
The exponential growth bound of the free wave evolution and compactness of the potential yield in combination with abstract operator theory a primary picture of the resolvent and spectrum of the linearized wave evolution.
\begin{lemma}
\label{LemmaLinEvEV}
Let $R\geq \frac{1}{2}$ and $k\in\NN$ such that $k\geq \frac{d-1}{2}$.
\begin{enumerate}[itemsep=1em, topsep=1em]
\item There is an $r_{0} > 0$ such that
\begin{equation*}
\big\{ z\in\CC \mid \Re(z) > -\tfrac{3}{2},\, |z|> r_{0} \big\} \subseteq \varrho(\mathbf{L})
\end{equation*}
and for any arbitrary but fixed $0 < \omega_{0} < 3/2$ the resolvent estimate
\begin{equation*}
\|\mathbf{R}_\mathbf{L}(z) \mathbf{f}\|_{\HH^{k}(\BB^{d}_{R})} \lesssim \| \mathbf{f} \|_{\HH^{k}(\BB^{d}_{R})}
\end{equation*}
holds for all $\mathbf{f}\in \HH^{k}_\rad(\BB^{d}_{R})$ and all $z\in\CC$ with $\Re(z) \geq -\omega_{0}$, $|z|\geq r_{0}$.
\item The part of the spectrum
\begin{equation*}
\sigma(\mathbf{L}) \cap \big\{ \lambda\in\CC \mid \Re(\lambda) > -\tfrac{3}{2} \big\}
\end{equation*}
is finite and consists of isolated eigenvalues with finite algebraic multiplicity.
\end{enumerate}
\end{lemma}
\begin{proof}
\begin{enumerate}[wide, itemsep=1em, topsep=1em]
\item Observe that the growth bound \eqref{MScgrowthbound} in \Cref{MScresult} together with \cite[p. 55, 1.10 Theorem]{MR1721989} implies
\begin{equation}
\label{resshift}
\big\{ z\in\mathbb{C} \mid \Re(z) > -\tfrac{3}{2} \big\} \subseteq \varrho(\overline{\mathbf{L}_d} - 2\mathbf{I}) \,.
\end{equation}
Suppose that $\Re(z) > -\frac{3}{2}$. In this case the important identity
\begin{equation}
\label{LBirmannSchwinger}
z\mathbf{I} - \mathbf{L} = \big( \mathbf{I} - \mathbf{L}_V\mathbf{R}_{\overline{\mathbf{L}_d} - 2\mathbf{I}}(z) \big)\big( z\mathbf{I} - \overline{\mathbf{L}_d} + 2\mathbf{I} \big)
\end{equation}
holds and shows that $z\in\varrho(\mathbf{L})$ if and only if $\mathbf{I} - \mathbf{L}_V\mathbf{R}_{\overline{\mathbf{L}_d} - 2\mathbf{I}}(z)$ is boundedly invertible. The latter is the case if the Neumann series of $\mathbf{L}_V\mathbf{R}_{\overline{\mathbf{L}_d} - 2\mathbf{I}}(z)$ converges. To achieve this, recall that
\begin{equation*}
\mathbf{L}_V\mathbf{R}_{\overline{\mathbf{L}_d} - 2\mathbf{I}}(z))\mathbf{f} =
V[\mathbf{R}_{\overline{\mathbf{L}_d} - 2\mathbf{I}}(z)\mathbf{f}]_{1} \mathbf{e}_2
\end{equation*}
and from the first component of the identity $(z\mathbf{I}-\overline{\mathbf{L}_d} + 2\mathbf{I})\mathbf{R}_{\overline{\mathbf{L}_d} - 2\mathbf{I}}(z)\mathbf{f} = \mathbf{f}$ we get
\begin{equation*}
(z+2)[\mathbf{R}_{\overline{\mathbf{L}_d} - 2\mathbf{I}}(z)\mathbf{f}]_{1} - [\mathbf{R}_{\overline{\mathbf{L}_d} - 2\mathbf{I}}(z)\mathbf{f}]_2 = f_{1} \,.
\end{equation*}
Also recall from \cite[p. 55, 1.10 Theorem]{MR1721989} the resolvent estimate
\begin{equation*}
\| \mathbf{R}_{\overline{\mathbf{L}_d} - 2\mathbf{I}}(z) \|_{\mathfrak{L}\big(\HH^{k}_\rad(\BB^{d}_{R})\big)} \lesssim \frac{1}{\Re(z) + \frac{3}{2}}
\end{equation*}
for all $z\in\CC$ with $\Re(z)>-\frac{3}{2}$. Now, since $V\in C^\infty(\overline{\BB^{d}_{R}})$ we obtain with this
\begin{align*}
\| \mathbf{L}_V\mathbf{R}_{\overline{\mathbf{L}_d} - 2\mathbf{I}}(z)\mathbf{f} \|_{\HH^{k}(\BB^{d}_{R})} &\lesssim
\| [\mathbf{R}_{\overline{\mathbf{L}_d} - 2\mathbf{I}}(z)\mathbf{f}]_{1} \|_{H^{k}(\BB^{d}_{R})} \\&\leq
\frac{1}{|z+2|} \big(
\| f_{1} \|_{H^{k}(\BB^{d}_{R})} + \| [\mathbf{R}_{\overline{\mathbf{L}_d} - 2\mathbf{I}}(z)\mathbf{f}]_2 \|_{H^{k}(\BB^{d}_{R})} \big) \\&\lesssim
\frac{1}{|z|} \|\mathbf{f}\|_{\HH^{k}(\BB^{d}_{R})}
\end{align*}
for all $\mathbf{f}\in\HH^{k}_\rad(\BB^{d}_{R})$ and all $\Re(z)\geq -\omega_{0}$. Hence, if $|z|$ is sufficiently large, then
\begin{equation*}
\|\mathbf{L}_V\mathbf{R}_{\overline{\mathbf{L}_d} - 2\mathbf{I}}(z)\|_{\mathfrak{L}\big(\HH^{k}_\rad(\BB^{d}_{R})\big)} < 1
\end{equation*}
and the Neumann series of $\mathbf{L}_V\mathbf{R}_{\overline{\mathbf{L}_d} - 2\mathbf{I}}(z)$ converges. This yields the first part of the lemma and the resolvent estimate follows from identity \eqref{LBirmannSchwinger}.
\item Suppose $\lambda\in\sigma(\mathbf{L})$ and $\Re(\lambda) > -\frac{3}{2}$. Then identity \eqref{LBirmannSchwinger} shows $1\in \sigma(\mathbf{L}_V\mathbf{R}_{\overline{\mathbf{L}_d} - 2\mathbf{I}}(\lambda))$. Also,
\begin{equation*}
\Omega \coloneqq \big\{ z\in\CC \mid \Re(z) > -\tfrac{3}{2} \big\} \rightarrow \mathfrak{L}\big( \HH^{k}_\rad(\BB^{d}_{R}) \big) \,, \quad
z\mapsto \mathbf{L}_V\mathbf{R}_{\overline{\mathbf{L}_d} - 2\mathbf{I}}(z) \,,
\end{equation*}
is an analytic map whose image is contained in the ideal of compact linear operators, since $\mathbf{L}_V$ is compact. Now though, the analytic Fredholm theorem \cite[p. 194, Theorem 3.14.3]{MR3364494} asserts that there is a discrete subset $S\subset\Omega$ such that $1\in \sigma(\mathbf{L}_V\mathbf{R}_{\overline{\mathbf{L}_d} - 2\mathbf{I}}(\lambda))$, if and only if $\lambda\in S$, and the residues of 
\begin{equation*}
\Omega \setminus S \rightarrow \mathfrak{L}\big( \HH^{k}_\rad(\BB^{d}_{R}) \big) \,, \quad z \mapsto \mathbf{R}_{\mathbf{L}_V\mathbf{R}_{\overline{\mathbf{L}_d} - 2\mathbf{I}}(z)}(1) = \big( \mathbf{I} - \mathbf{L}_V\mathbf{R}_{\overline{\mathbf{L}_d} - 2\mathbf{I}}(z) \big)^{-1} \,,
\end{equation*}
at $\lambda\in S$, which coincide with the spectral projection associated to $1\in \sigma(\mathbf{L}_V\mathbf{R}_{\overline{\mathbf{L}_d} - 2\mathbf{I}}(\lambda))$, have finite rank. Because $S$ is contained in a compact set by the first part of the lemma, it is finite. Moreover, by compactness we have in fact $1\in\sigma_\mathrm{p}(\mathbf{L}_V\mathbf{R}_{\overline{\mathbf{L}_d} - 2\mathbf{I}}(\lambda))$, to which there is an eigenvector $\mathbf{g}\in\HH^{k}_\rad(\BB^{d}_{R})\setminus\{ \mathbf{0} \}$. This implies $\mathbf{f} \coloneqq \mathbf{R}_{\overline{\mathbf{L}_d} - 2\mathbf{I}}(\lambda)\mathbf{g} \in \dom(\mathbf{L}) \setminus\{ \mathbf{0} \}$ and
\begin{equation*}
(\lambda\mathbf{I} - \mathbf{L})\mathbf{f} = \big( \mathbf{I} - \mathbf{L}_V\mathbf{R}_{\overline{\mathbf{L}_d} - 2\mathbf{I}}(\lambda) \big)\big( \lambda\mathbf{I} - \overline{\mathbf{L}_d} + 2\mathbf{I} \big) \mathbf{R}_{\overline{\mathbf{L}_d} - 2\mathbf{I}}(\lambda)\mathbf{g} = \big( \mathbf{I} - \mathbf{L}_V\mathbf{R}_{\overline{\mathbf{L}_d} - 2\mathbf{I}}(\lambda) \big)\mathbf{g} = \mathbf{0} \,,
\end{equation*}
that is, $\mathbf{f}$ is an eigenvector of $\mathbf{L}$ with eigenvalue $\lambda$.
\qedhere
\end{enumerate}
\end{proof}
\subsection{The mode stability problem in hyperboloidal similarity coordinates}
We know from \Cref{LemmaLinEvEV} that the unstable part of the spectrum of $\mathbf{L}$ is confined to a compact region and, even better, consists of only finitely many eigenvalues. However, characterizing those eigenvalues in detail leads to a rather challenging problem known as the \emph{mode stability problem}, on which we elaborate below. The following lemma provides the link between the spectral problem for $\mathbf{L}$ and the equation governing the mode stability problem.
\begin{lemma}
\label{ModeStabilityYM}
Let $R\geq\frac{1}{2}$ and $k\in\NN$ such that $k\geq \frac{d-1}{2}$. Let $\Re(\lambda) \geq 0$. Then $\lambda\in\sigma(\mathbf{L})$ if and only if there exists a nonzero $f_{\lambda} \in C^\infty_\rad(\overline{\BB^d_{R}})$ such that
\begin{equation}
\label{ModeLambda}
\widetilde{v}_{\lambda}(s,y) = \ee^{(\lambda+2)s} f_{\lambda}(y)
\end{equation}
is a \emph{mode solution} to the equation
\begin{equation}
\label{ModeEq}
\left( \Box_{\g} + V_{T}\circ\eta_{T} \right) \widetilde{v}_{\lambda} = 0 \quad \text{on } \RR\times\BB^d_{R} \setminus\{0\} \,.
\end{equation}
If such a \emph{mode} $f_{\lambda}$ exists, it is unique up to a multiplicative constant and
\begin{equation}
\label{MultFundSol}
\ker(\lambda\mathbf{I} - \mathbf{L}) = \langle \mathbf{f}_{\lambda} \rangle \quad \text{where} \quad \mathbf{f}_{\lambda} = f_{\lambda}
\begin{bmatrix}
1 \\
\lambda + 2
\end{bmatrix}
\in C^\infty_\rad(\overline{\BB^{d}_{R}})^{2}
\,.
\end{equation}
\end{lemma}
\begin{proof}
Suppose $\lambda\in\sigma(\mathbf{L})$ and $\Re(\lambda) \geq 0$. We get from \Cref{LemmaLinEvEV} that $\lambda\in\sigma_\mathrm{p}(\mathbf{L})$, so there is an eigenvector $\mathbf{f}_{\lambda} \in \ker(\lambda\mathbf{I} - \mathbf{L})$. The components $f_{\lambda,1}$, $f_{\lambda,2}$ of $\mathbf{f}_{\lambda} \in \dom(\mathbf{L}) \subset \HH^{k}_\rad(\BB^{d}_{R})$ are radial, satisfy $f_{\lambda,2} = (\lambda+2)f_{\lambda,1}$ and because $\mathbf{f}_{\lambda} \neq \mathbf{0}$ we have $f_{\lambda} \coloneqq f_{\lambda,1} \neq 0$. By \Cref{oddSobolevBalls} we infer for the radial representative that $|\,.\,|^{\frac{d-1}{2}}\widehat{f}_{\lambda} \in H^{k}(\BB_{R})$. From one-dimensional Sobolev embedding we conclude $\widehat{f}_{\lambda} \in C^{k-1}((0,R]) \subset C^{2}((0,R])$ and as the condition
\begin{equation*}
(\lambda\mathbf{I} - \mathbf{L})\mathbf{f}_{\lambda} = \mathbf{0}
\end{equation*}
is equivalent to
\begin{equation*}
\pd_s \left( \ee^{(\lambda+2)s} \mathbf{f}_{\lambda} \right) = \overline{\mathbf{L}_d} \left( \ee^{(\lambda+2)s} \mathbf{f}_{\lambda} \right) + \mathbf{L}_{V} \left( \ee^{(\lambda+2)s} \mathbf{f}_{\lambda} \right) \,,
\end{equation*}
we have that the first component $\widetilde{v}_{\lambda}(s,y) \coloneqq \ee^{(\lambda+2)s} f_{\lambda}(y)$ is a classical solution to the radial linear wave equation
\begin{equation*}
\left( \Box_{\g} + V_{T}\circ\eta_{T} \right) \widetilde{v}_{\lambda} = 0 \quad \text{on } \RR\times\BB^d_{R} \setminus\{0\} \,.
\end{equation*}
The radial mode ansatz yields the differential equation
\begin{equation}
\label{ModeEqHSC}
\widehat{f}_{\lambda}'' + p \widehat{f}_{\lambda}' + q \widehat{f}_{\lambda} = 0 \quad \text{on } (0,R)
\end{equation}
with coefficients
\begin{align*}
p(\eta) &= \frac{d-1}{\eta} + 2\left( \lambda - \frac{d-7}{2} \right) \frac{h(\eta)h'(\eta) - \eta}{h(\eta)^{2} - \eta^{2}} - \frac{\eta h''(\eta)}{\eta h'(\eta) - h(\eta)}  \,, \\
q(\eta) &=
-\frac{1-h'(\eta)^{2}}{h(\eta)^{2} - \eta^{2}} \left(
\widehat{V}(\eta) +
(\lambda+3)(d-2) +
\frac{(\lambda+3)(d-1)}{1 - h'(\eta)^{2}} \frac{\eta - h(\eta)h'(\eta)}{\eta}
\right) -
\frac{(\lambda+3)\eta h''(\eta)}{\eta h'(\eta) - h(\eta)} \,.
\end{align*}
The relevant poles are at $\eta = \frac{1}{2}$ for both coefficients $p$, $q$ and at $\eta = 0$ for $p$. The Frobenius indices are
\begin{equation*}
\left\{ 0, \tfrac{d-5}{2} - \lambda \right\} \text{ at } \eta = \tfrac{1}{2} \quad \text{and} \quad \{ 0, -d+2 \} \text{ at } \eta = 0
\end{equation*}
and in fact independent of the specific height function $h$. To conclude smoothness of $\widehat{f}_{\lambda}$ we investigate the behaviour at the singular points in analogy to \cite{MR4469070}.
\par\medskip
At the singular point $\eta=\frac{1}{2}$ we have to distinguish two cases.
\begin{enumerate}[wide, itemsep=1em, topsep=1em]
\item[\textit{Case $\frac{d-5}{2} - \lambda \in \NN_{0}$.}] Then the fundamental solutions to \Cref{ModeEqHSC} exhibit the behaviour
\begin{equation}
\label{Frob12in}
\widehat{\phi}_{1}(\eta) = \left( \tfrac{1}{2}-\eta \right)^{\frac{d-5}{2} - \lambda} \varphi_{1}(\eta) \,, \quad
\widehat{\phi}_2(\eta) = \varphi_2(\eta) + \alpha\left( \tfrac{1}{2}-\eta \right)^{\frac{d-5}{2} - \lambda} \log\left( \tfrac{1}{2}-\eta \right) \varphi_{1}(\eta) \,,
\end{equation}
where $\varphi_{1}$, $\varphi_2$ are analytic in a neighbourhood around $\eta = \frac{1}{2}$ and satisfy $\varphi_{1}(\frac{1}{2}) = \varphi_2(\frac{1}{2}) = 1$ and $\alpha\in\CC$. However, $|\,.\,|^{\frac{d-1}{2}}\widehat{\phi}_2$ belongs to $H^{k}(\BB_{R})$ only if $\alpha = 0$, in which case both fundamental solutions are analytic at $\eta = \frac{1}{2}$. If otherwise $\alpha\neq 0$ then $\widehat{f}_{\lambda}$ has to be a multiple of $\widehat{\phi}_{1}$ and will be again analytic at $\eta = \frac{1}{2}$.
\item[\textit{Case $\frac{d-5}{2} - \lambda \notin \NN_{0}$.}] We get from the Frobenius method $\varphi_{1}$, $\varphi_2$ that are analytic in a neighbourhood around $\eta = \frac{1}{2}$ and satisfy $\varphi_{1}(\frac{1}{2}) = \varphi_2(\frac{1}{2}) = 1$ such that
\begin{equation}
\label{Frob12notin}
\widehat{\phi}_{1}(\eta) = \varphi_{1}(\eta) \,, \quad
\widehat{\phi}_2(\eta) = \left( \tfrac{1}{2}-\eta \right)^{\frac{d-5}{2} - \lambda} \varphi_2(\eta) \,.
\end{equation}
As $k\geq \frac{d-1}{2}$, it follows that $|\,.\,|^{\frac{d-1}{2}} \widehat{\phi}_2 \notin H^{k}(\BB_{R})$, no matter the location of $\lambda\in \CC$, in each dimension $d\geq 7$. So $\widehat{f}_{\lambda}$ is a multiple of $\widehat{\phi}_{1}$ and smooth at $\eta = \frac{1}{2}$.
\end{enumerate}
\par\medskip
At the remaining singular point $\eta = 0$ the fundamental solutions are of the form
\begin{equation}
\label{Frob0}
\widehat{\phi}_{1}(\eta) = \varphi_{1}(\eta)\,, \qquad
\widehat{\phi}_2(\eta) = \eta^{-d+2}\varphi_2(\eta) +  \alpha \log(\eta)\varphi_{1}(\eta) \,,
\end{equation}
where $\varphi_{1}$, $\varphi_2$ are analytic in a neighbourhood around $\eta = 0$ with $\varphi_{1}(0) = \varphi_2(0) = 1$ and $\alpha \in \CC$. However, observe that one fundamental solution is highly singular at the origin with $|\,.\,|^{\frac{d-1}{2}} \widehat{\phi}_2 \notin H^{k}(\BB_{R})$ in each dimension, no matter the value of $\alpha$. As a consequence, $\widehat{f}_{\lambda}$ is a multiple of the fundamental solution $\widehat{\phi}_{1}$ and thus analytic also at the origin $\eta = 0$. This part also shows that in any case, $\widehat{f}_{\lambda}$ is a constant nonzero multiple of one fundamental solution.
\par\medskip
We conclude that $f_{\lambda}$ belongs to $C^\infty_\rad(\overline{\BB^d_{R}})$, is a solution to \Cref{ModeEqHSC} and that \eqref{MultFundSol} holds.
\end{proof}
In accordance to the literature \cite{MR3538419}, \cite{MR3623242}, \cite{MR3278903}, \cite{MR3475668}, \cite{MR4469070}, a nonzero smooth solution of the form \eqref{ModeLambda} to \Cref{ModeEq} is called a \emph{mode solution}. It is called an \emph{unstable mode solution} if $\Re(\lambda)\geq 0$ and the corresponding $\lambda$ is called an \emph{unstable eigenvalue}. We will see in the proof of \Cref{LemmaKerLin} without effort that time translation invariance of the Yang-Mills equation induces the unstable eigenvalue $\lambda=1$. The \emph{mode stability problem} claims that this is in fact the only unstable eigenvalue and, in this respect, we say that the self-similar solution $u_{T}^{*}$ is \emph{mode stable} if unstable mode solutions exist only for $\lambda=1$. Posed this way, the mode stability problem lays the focus on the study of solutions to the spectral ODE \eqref{ModeEqHSC}, which does not require any knowledge of the concrete operator $\mathbf{L}$, its spectral properties nor the precise functional setting. This is convenient for disclosing the core of the analysis behind the mode stability problem. Later, when we identify $\lambda=1$ as the only source for instabilities for the linearized wave evolution, we need the additional information that the geometric and algebraic multiplicities of $\lambda=1$ as an eigenvalue of $\mathbf{L}$ coincide. This is addressed in \Cref{SpecProjProp}.
\par\medskip
We continue the spectral analysis and prove mode stability of $u_{T}^{*}$. Interestingly, the structure of the linear wave equation \eqref{ModeEq} enables us to reduce the core of this difficult problem to standard similarity coordinates where it has been solved recently by I. Glogi\'{c} \cite{MR4469070} in all dimensions.
\begin{proposition}
\label{LemmaKerLin}
Let $R\geq \frac{1}{2}$, $k\in\NN$ such that $k\geq \frac{d-1}{2}$. We have
\begin{equation*}
\sigma(\mathbf{L})\cap \{\lambda\in\CC\mid \Re(\lambda)\geq 0 \} = \{1\}
\end{equation*}
and
\begin{equation*}
\ker(\mathbf{I}-\mathbf{L}) = \langle\mathbf{f}_{1}^{*}\rangle
\end{equation*}
with
\begin{equation}
\label{symmetrymode}
\mathbf{f}_{1}^{*}(y) = \frac{h(y)}{( b_d h(y)^{2} + |y|^{2} )^{2}}
\begin{bmatrix}
1\\3
\end{bmatrix}
\,.
\end{equation}
\end{proposition}
\begin{proof}
Suppose $\lambda\in\sigma(\mathbf{L})$ and $\Re(\lambda)\geq 0$. By \Cref{ModeStabilityYM} there is an eigenvector $\mathbf{f}_{\lambda} \in \ker( \lambda\mathbf{I} - \mathbf{L} )$ whose first component $f_{\lambda,1} \in C^\infty_\rad(\overline{\BB^d_{R}})$ yields a mode solution $\widetilde{v}_{\lambda}(s,y) \coloneqq \ee^{(\lambda+2)s} f_{\lambda,1}(y)$ to \Cref{ModeEq}. We show that this already implies that $\lambda$ is constrained to \cite[Eq. (3.10)]{MR4469070} by changing via the transition diffeomorphism
\begin{equation}
\label{SSCtoHSC}
\RR\times\overline{\BB_{1}^d} \rightarrow \RR\times\overline{\BB_{\frac{1}{2}}^d} \,, \quad (\tau,\xi) \mapsto \eta_{T}^{-1}(T-\ee^{-\tau}, \ee^{-\tau} \xi) = \big( \tau - \log h_{1}(\xi), h_{1}(\xi)^{-1}\xi \big) \,,
\end{equation}
where
\begin{equation*}
h_{1}(\xi) = 1 + \frac{1}{2} \sqrt{2(1+|\xi|^{2})} \,,
\end{equation*}
see \Cref{geomHSC}, to standard similarity coordinates and then draw the conclusion. Now set
\begin{equation*}
\widetilde{w}_{\lambda}(\tau,\xi) = (\widetilde{v}_{\lambda}\circ\eta_{T}^{-1})(T - \ee^{-\tau},\ee^{-\tau}\xi) = \ee^{(\lambda+2)\tau} g_{\lambda}(\xi)
\end{equation*}
where $g_{\lambda} \in C^\infty(\overline{\BB_{1}^d})$ is given by
\begin{equation*}
g_{\lambda}(\xi) \coloneqq \widehat{g}_{\lambda}(|\xi|) \coloneqq \frac{\widehat{f}_{\lambda,1}(h_{1}(|\xi|)^{-1}|\xi|)}{h_{1}(|\xi|)^{\lambda+2}} \,.
\end{equation*}
This transforms the homogeneous linear wave equation \eqref{ModeEq} via
\begin{align*}
&(\Box_{\g}\widetilde{v}_{\lambda}) \circ\eta_{T}^{-1}(T - \ee^{-\tau},\ee^{-\tau}\xi) \\&=
\ee^{2\tau} \left( - \pd_\tau^{2} - \pd_\tau - 2\xi^i\pd_{\xi^i}\pd_\tau - 2\xi^i\pd_{\xi^i} + (\delta^{ij}-\xi^i\xi^{j})\pd_{\xi^i}\pd_{\xi^{j}} \right)\widetilde{w}_{\lambda}(\tau,\xi) \\&=
\ee^{(\lambda+4)\tau} \left( (\delta^{ij} - \xi^i\xi^{j})\pd_{\xi^i}\pd_{\xi^{j}} - 2(\lambda+3)\xi^i\pd_{\xi^i} - (\lambda+2)(\lambda+3) \right) g_{\lambda}(\xi)
\end{align*}
and $V_{T}(T-\ee^{-\tau}, \ee^{-\tau}\xi) = \ee^{2\tau} V_d(\xi) = \ee^{2\tau} \widehat{V}_d(|\xi|)$ with
\begin{equation*}
\widehat{V}_d(\rho) = - 3(d-4)a_d \frac{(a_d - 2)\rho^{2} - 2b_d}{(b_d + \rho^{2})^{2}}
\end{equation*}
into the equation
\begin{equation*}
\big(1-\rho^{2}\big) \widehat{g}_{\lambda}''(\rho)
+\frac{d-1}{\rho} \widehat{g}_{\lambda}'(\rho)
-2(\lambda+3)\rho\widehat{g}_{\lambda}'(\rho)
-(\lambda+2)(\lambda+3)\widehat{g}_{\lambda}(\rho)
+\widehat{V}_d(\rho)\widehat{g}_{\lambda}(\rho) = 0 \,.
\end{equation*}
This is precisely \cite[Eq. (3.10)]{MR4469070} and, at this point, according to the established claim \cite[p. 22]{MR4469070}, there is no $\lambda\neq 1$ with $\Re(\lambda)\geq 0$ such that this equation admits a nonzero solution $\widehat{g}_{\lambda} \in C^\infty([0,1])$.
\par\medskip
For $\lambda = 1$ we observe that the blowup solution $u_{T}^{*}$ given in \Cref{selfsimilarblowupextended} satisfies
\begin{equation*}
-\pd_{t}^{2} u_{T}^{*}(t,x) + \Delta_{x} u_{T}^{*}(t,x) = (d-4)(|x|^{2}u_{T}^{*}(t,x)^3 + 3u_{T}^{*}(t,x)^{2}) \,.
\end{equation*}
Since $u_{T}^{*}$ depends smoothly on the parameter $T\in\RR$ we can differentiate this equation with respect to $T$ and obtain
\begin{align*}
(-\pd_{t}^{2} + \Delta_{x}) (\pd_{T} u_{T}^{*})(t,x) &= (d-4) \big( 3|x|^{2} u_{T}^{*}(t,x)^{2} + 6u_{T}^{*}(t,x) \big)(\pd_{T} u_{T}^{*})(t,x) \\&=
-V_{T}(t,x) (\pd_{T} u_{T}^{*})(t,x) \,.
\end{align*}
Now
\begin{equation*}
(\pd_{T} u_{T}^{*})\circ\eta_{T}(s,y) = - 2 a_d b_d \ee^{3s} \frac{h(y)}{( |y|^{2} + b_d h(y)^{2} )^{2}}
\end{equation*}
is the smooth solution to \Cref{ModeEq} for $\lambda=1$ and \Cref{ModeStabilityYM} yields $\langle \mathbf{f}_{1}^{*} \rangle = \ker(\mathbf{I}-\mathbf{L})$ since the other fundamental solution does not belong to $H^{k}(\BB_{R})$ according to the Frobenius method \eqref{Frob12in}.
\end{proof}
\subsection{Properties of the spectral projection}
So far, we have determined the eigenvalue $1\in\sigma(\mathbf{L})$ as an instability in the linearized wave evolution. Hence we are interested in decomposing the time evolution into a stable and unstable part. For this, the method of choice is the \emph{Riesz projection}, which is a projection that allows us to remove the unstable eigenvalue. \Cref{LemmaLinEvEV} and \Cref{LemmaKerLin} show that $1\in\sigma(\mathbf{L})$ is an isolated eigenvalue encircled by $\gamma: [0,2\pi]\rightarrow\CC$, $t\mapsto 1 + \ee^{\ii t}$, so there corresponds a spectral projection
\begin{equation}
\label{RieszProjection}
\mathbf{P} \coloneqq \frac{1}{2\pi\ii}\int_\gamma \mathbf{R}_\mathbf{L}(z) \dd z
\end{equation}
with the following important properties.
\begin{proposition}
\label{SpecProjProp}
Let $R\geq \frac{1}{2}$, $k\in\NN$ such that $k\geq \frac{d-1}{2}$.
\begin{enumerate}[itemsep=1em, topsep=1em]
\item The Riesz projection $\mathbf{P} \in \mathfrak{L}\big( \HH^{k}_\rad(\BB^{d}_{R}) \big)$ has the properties
\begin{equation*}
\mathbf{P}\mathbf{L} \subset \mathbf{L}\mathbf{P} \,, \qquad
\sigma(\mathbf{L}\restriction_{\ran(\mathbf{P})}) = \{1 \} \,, \qquad
\sigma(\mathbf{L}\restriction_{\ran(\mathbf{I}-\mathbf{P})}) = \sigma(\mathbf{L})\setminus\{1 \} \,. 
\end{equation*}
\item The range of $\mathbf{P}$ is one-dimensional and spanned by the symmetry mode from \Cref{symmetrymode},
\begin{equation*}
\ran(\mathbf{P}) = \langle \mathbf{f}_{1}^{*} \rangle \,.
\end{equation*}
\end{enumerate}
\end{proposition}
\begin{proof}
\begin{enumerate}[wide, itemsep=1em, topsep=1em]
\item This construction is standard in spectral theory and the asserted properties are basic consequences, see \cite[p. 178, Theorem 6.17]{MR1335452}.
\item By construction $\sigma(\mathbf{L}\restriction_{\ran(\mathbf{P})}) = \{1\}$ holds with eigenvector $\mathbf{f}_{1}^{*}$ so we have $\mathbf{f}_{1}^{*} \in \ran(\mathbf{P})$. To see the other inclusion, first recall from \Cref{LemmaLinEvEV} that $1\in\sigma_\mathrm{p}(\mathbf{L})$ has finite algebraic multiplicity, i.e. the subspace $\ran(\mathbf{P}) \subset \HH^{k}_\rad(\BB^d_{R})$ is finite dimensional. So, the part $\mathbf{L}\restriction_{\ran(\mathbf{P})}$ of $\mathbf{L}$ in $\ran(\mathbf{P})$ is an operator defined on $\ran(\mathbf{P})$ with spectrum $\sigma(\mathbf{L}\restriction_{\ran(\mathbf{P})}) = \{1\}$. As an operator on a finite dimensional space with only $0$ as eigenvalue we get that $\mathbf{I} - \mathbf{L}\restriction_{\ran(\mathbf{P})}$ is nilpotent. We begin by assuming that the corresponding nilpotence index is strictly larger than one, in which case
\begin{equation}
\label{nilpot}
\mathbf{I} - \mathbf{L}\restriction_{\ran(\mathbf{P})} \neq \mathbf{0} \,.
\end{equation}
Then $\ran(\mathbf{I} - \mathbf{L}\restriction_{\ran(\mathbf{P})}) \neq \{\mathbf{0}\}$ and since $\mathbf{I} - \mathbf{L}\restriction_{\ran(\mathbf{P})}$ is nilpotent,
\begin{equation*}
\ran(\mathbf{I} - \mathbf{L}\restriction_{\ran(\mathbf{P})}) \cap \ker(\mathbf{I} - \mathbf{L}\restriction_{\ran(\mathbf{P})}) \neq \{ \mathbf{0} \} \,.
\end{equation*}
Together with $\ker(\mathbf{I} - \mathbf{L}) = \langle\mathbf{f}_{1}^{*}\rangle$ this implies that there exists $\mathbf{0} \neq \mathbf{f} \in \ran(\mathbf{P})$ such that
\begin{equation*}
(\mathbf{I} - \mathbf{L})\mathbf{f} = \mathbf{f}_{1}^{*} \,.
\end{equation*}
The components $f_{1}$, $f_2$ of $\mathbf{f}\in\HH^{k}_\rad(\BB_{R})$ are radial and we have by our characterization of Sobolev norms for radial functions in \Cref{oddSobolevBalls} that $|\,.\,|^{\frac{d-1}{2}} \widehat{f}_{1} \in H^{k}(\BB_{R})$ and thus, by one-dimensional Sobolev embedding $\widehat{f}_{1} \in C^{k-1}((0,R]) \subset C^{2}((0,R])$. That is, $\widehat{f}_{1}$ satisfies
\begin{equation*}
\renewcommand{\arraystretch}{1.2}
\left\{
\begin{array}{rcl}
\widehat{f}_{1,1}^{*} &=& 3 \widehat{f}_{1} - \widehat{f}_2 \,, \\
\widehat{f}_{1,2}^{*} &=& 3 \widehat{f}_2 - \big( c_{11}^d \widehat{f}_{1}' + c_{12} \widehat{f}_{1}'' + c_{20}^d \widehat{f}_2 + c_{21} \widehat{f}_2' \big) - \widehat{V}\widehat{f}_{1} \,,
\end{array}
\right.
\end{equation*}
classically. Inserting the first equation in the second yields
\begin{equation*}
-9  \widehat{f}_{1} + c_{11}^d \widehat{f}_{1}' + c_{12} \widehat{f}_{1}'' + 3 c_{20}^d \widehat{f}_{1} + 3 c_{21} \widehat{f}_{1}' + \widehat{V}\widehat{f}_{1} = (c_{20}^d - 6)\widehat{f}_{1,1}^{*} + c_{21} \widehat{f}_{1,1}^{*'} \,.
\end{equation*}
Multiplying through with $\ee^{3s}$, we get
\begin{align*}
&
\big( - \pd_s^{2} + c_{11}^d(\eta) \pd_\eta + c_{12}(\eta) \pd_\eta^{2} + c_{20}^d(\eta) \pd_s + c_{21}(\eta) \pd_\eta \pd_s + \widehat{V}(\eta) \big) \big( \ee^{3s} \widehat{f}_{1}(\eta) \big) \\&=  \ee^{3s} \big( (c_{20}^d(\eta) - 6) \widehat{f}_{1,1}^{*}(\eta) + c_{21}(\eta) \widehat{f}_{1,1}^{*'}(\eta) \big) \,.
\end{align*}
Multiplying by $\g^{00}(s,\eta)$, and setting $\widetilde{v}(s,y) = \ee^{3s} \widehat{f}_{1}(|y|)$ we see that this is an inhomogeneous linear wave equation,
\begin{equation}
\label{LinInhomWaveEq}
\big( \Box_{\g} + V_{T}(\eta_{T}(s,y)) \big)\widetilde{v}(s,y) =
\g^{00}(s,y) \ee^{3s} \big( (c_{20}^d(|y|) - 6) \widehat{f}_{1,1}^{*}(|y|) + c_{21}(|y|) \widehat{f}_{1,1}^{*'}(|y|) \big) \,.
\end{equation}
We restrict ourselves to $y\in\BB^d_{\frac{1}{2}}$ and consider the change of variables \eqref{SSCtoHSC}. We first compute on the right-hand side
\begin{align*}
\g^{00}(\tau - \log h_{1}(\xi), h_{1}(\xi)^{-1} \xi) &= \ee^{2\tau} \left( (1-|\xi|^{2}) H(|\xi|)^{2} + 2 |\xi| H(|\xi|) - 1 \right) \,, \\
c_{20}^d(| h_{1}(\xi)^{-1} \xi |) &= \frac{1 - 2|\xi|H(|\xi|) + (d-1)|\xi|^{-1}H(|\xi|) + (1-|\xi|^{2}) H'(|\xi|) }{ (1-|\xi|^{2}) H(|\xi|)^{2} + 2 |\xi| H(|\xi|) - 1 } \,, \\
c_{21}(| h_{1}(\xi)^{-1} \xi |) &= \frac{2(1-|\xi|^{2})H(|\xi|) + 2|\xi|}{ (1-|\xi|^{2}) H(|\xi|)^{2} + 2 |\xi| H(|\xi|) - 1 } \left( - \frac{|\xi|h_{1}'(|\xi|) - h_{1}(|\xi|)}{h_{1}(|\xi|)^{2}} \right) \,,
\end{align*}
and
\begin{align*}
\widehat{f}_{1,1}^{*}(| h_{1}(\xi)^{-1} \xi |) &= -h_{1}(|\xi|)^3 \widehat{g}_{1,1}^{*}(|\xi|) \,, \\
\widehat{f}_{1,1}^{*'}(| h_{1}(\xi)^{-1} \xi |) &= - h_{1}(|\xi|)^3 \left( 3H(|\xi|)\widehat{g}_{1,1}^{*}(|\xi|) + \widehat{g}_{1,1}^{*'}(|\xi|) \right) \left( -\frac{h_{1}(|\xi|)^{2}}{|\xi|h_{1}'(|\xi|) - h_{1}(|\xi|)} \right) \,,
\end{align*}
where
\begin{equation*}
\widehat{g}_{1,1}^{*}(\rho) = \frac{1}{(b_d + \rho^{2})^{2}} \quad \text{and} \quad H(\rho) = \pd_\rho \log h_{1}(\rho) \,.
\end{equation*}
Then we set
\begin{equation*}
\widetilde{w}(\tau,\xi) = (\widetilde{v}\circ\eta_{T}^{-1})(T - \ee^{-\tau},\ee^{-\tau}\xi) = \ee^{3\tau} g_{1}(\xi)
\end{equation*}
where
\begin{equation}
\label{SSCrelHSC}
g_{1}(\xi) = \widehat{g}_{1}(|\xi|) = \frac{\widehat{f}_{1}(h_{1}(|\xi|)^{-1}|\xi|)}{h_{1}(|\xi|)^3} \quad \text{resp.} \quad f_{1}(y) = \widehat{f}_{1}(|y|) = - \frac{\widehat{g}_{1}( - h(|y|)^{-1} |y| )}{h(|y|)^3} \,.
\end{equation}
This transforms \Cref{LinInhomWaveEq} via
\begin{align*}
&\left( \Box_{\g}\widetilde{v} + (V_{T}\circ\eta_{T})\widetilde{v} \right) \circ\eta_{T}^{-1}(T - \ee^{-\tau},\ee^{-\tau}\xi) \\&=
\ee^{2\tau} \left( - \pd_\tau^{2} - \pd_\tau - 2\xi^i\pd_{\xi^i}\pd_\tau - 2\xi^i\pd_{\xi^i} + (\delta^{ij}-\xi^i\xi^{j})\pd_{\xi^i}\pd_{\xi^{j}} + V_d(\xi) \right)\widetilde{w}(\tau,\xi) \\&=
\ee^{5\tau} \left( (\delta^{ij}-\xi^i\xi^{j})\pd_{\xi^i}\pd_{\xi^{j}} - 8\xi^i\pd_{\xi^i} - 12 + V_d(\xi) \right) g_{1}(\xi)
\end{align*}
into the equation
\begin{equation}
\label{InhomLinWave}
(1-\rho^{2})\widehat{g}_{1}''(\rho) + \big( (d-1)\rho^{-1} - 8\rho \big) \widehat{g}_{1}'(\rho) - \big( 12 - \widehat{V}_d(\rho) \big) \widehat{g}_{1}(\rho) = G(\rho)
\end{equation}
with the inhomogeneity $G = G^{*} + G^h$ given by
\begin{align}
\nonumber
G^{*}(\rho) &= - \left( 7 \widehat{g}_{1,1}^{*}(\rho) + 2\rho \widehat{g}_{1,1}^{*'}(\rho) \right) \\&=-
\frac{1}{\rho^6 \psi(\rho)} G_7^{*'}(\rho) \quad \text{where} \quad G_7^{*}(\rho) = \rho^7 \psi_{1}(\rho)^{2} \,, \\
\nonumber
G^h(\rho) &= - \frac{1}{\widehat{g}_{1,1}^{*}(\rho)} \left( \big( (d-1) \rho^{-1} - 8\rho \big) \big( H(\rho)\widehat{g}_{1,1}^{*}(\rho)^{2} \big) + (1-\rho^{2}) \pd_\rho \big( H(\rho) \widehat{g}_{1,1}^{*}(\rho)^{2} \big) \right) \\&=-
\frac{1 - \rho^{2}}{\widehat{g}_{1,1}^{*}(\rho)} \frac{\left( 1 - \rho^{2} \right)^{\frac{d-9}{2}}}{\rho^{d-1}} G_d'(\rho) \quad \text{where} \quad G_d(\rho) = \frac{\rho^{d-1}H(\rho)\widehat{g}_{1,1}^{*}(\rho)^{2}}{\left( 1 - \rho^{2} \right)^{\frac{d-9}{2}}}  \,,
\end{align}
as it follows from a curious simplification. The term $G^h$ is an artefact of hyperboloidal similarity coordinates and depends on the choice of the height function $h$. By comparing this with \cite[Eq. (3.37)]{MR4469070} we see that $G^h$ does not appear in the corresponding problem in standard similarity coordinates and this subtlety will demand extra care in our analysis below. Note that $\widehat{g}_{1,1}^{*}$ is just the symmetry mode in standard similarity coordinates, hence a solution to the homogeneous version of \Cref{InhomLinWave}. So via reduction of order \cite{MR2961944} we obtain two linearly independent solutions
\begin{align}
\psi_{1}(\rho) &= \widehat{g}_{1,1}^{*}(\rho) = \frac{1}{(b_d + \rho^{2})^{2}} \,, \\
\psi_2(\rho) &= \psi_{1}(\rho)I_d(\rho) \quad \text{with} \quad I_d'(\rho) = \frac{(1-\rho^{2})^{\frac{d-9}{2}}}{\rho^{d-1}} \frac{1}{\psi_{1}(\rho)^{2}} \,.
\end{align}
of the homogeneous problem associated to \Cref{InhomLinWave}. The Wronskian of $\psi_{1}$, $\psi_2$ is given by
\begin{equation}
\label{WronskianSSC}
W(\rho) = \psi_{1}(\rho)\psi_2'(\rho) - \psi_{1}'(\rho)\psi_2(\rho) = \frac{(1-\rho^{2})^{\frac{d-9}{2}}}{\rho^{d-1}}
\end{equation}
and the variation of constants formula yields $\alpha_{1},\alpha_2 \in \CC$ such that the general solution is given by
\begin{equation*}
\widehat{g}_{1}(\rho) = \alpha_{1} \psi_{1}(\rho) + \alpha_2\psi_2(\rho) - \psi_{1}(\rho) \int_{0}^\rho \frac{\psi_2(\rho')}{W(\rho')}\frac{G(\rho')}{1-{\rho'}^{2}}\dd\rho' + \psi_2(\rho) \int_{0}^\rho \frac{\psi_{1}(\rho')}{W(\rho')}\frac{G(\rho')}{1-{\rho'}^{2}}\dd\rho' \,.
\end{equation*}
Clearly, $\psi_{1} \in C^\infty([0,1])$ and $\psi_2$ has a pole of order $(d-2)$ at $\rho = 0$ and $f_{1}\in H^{k}(\BB^d_{R})$ forces $\alpha_2 = 0$ through the relation \eqref{SSCrelHSC}. First we investigate the contribution of the coordinate artefact. That is, we compute
\begin{align*}
\int_{0}^\rho \frac{\psi_{1}(\rho')}{W(\rho')} \frac{G^h(\rho')}{1-{\rho'}^{2}} \dd\rho' &= - G_d(\rho) \,, \\
\int_{0}^\rho \frac{\psi_2(\rho')}{W(\rho')} \frac{G^h(\rho')}{1-{\rho'}^{2}} \dd\rho' &= - G_d(\rho) I_d(\rho) + \int_{0}^\rho G_d(\rho') I_d'(\rho') \dd\rho' \,,
\end{align*}
where we integrated by parts. The latter integral is easily computed so that the simplified solution reads
\begin{equation*}
\widehat{g}_{1}(\rho) = \psi_{1}(\rho) \left( \alpha - \log h_{1}(\rho) \right)  - \psi_{1}(\rho) \int_{0}^\rho \frac{\psi_2(\rho')}{W(\rho')}\frac{G^{*}(\rho')}{1-{\rho'}^{2}}\dd\rho' + \psi_2(\rho) \int_{0}^\rho \frac{\psi_{1}(\rho')}{W(\rho')}\frac{G^{*}(\rho')}{1-{\rho'}^{2}}\dd\rho' \,.
\end{equation*}
The behaviour of the solution $\widehat{f}_{1}$ depends on the space dimension $d$ and is analysed as follows.
\begin{enumerate}[wide, itemsep=1em, topsep=1em]
\item[\textit{Case $d=7$.}] As above, we compute
\begin{align*}
\int_{0}^\rho \frac{\psi_{1}(\rho')}{W(\rho')} \frac{G^{*}(\rho')}{1-{\rho'}^{2}} \dd\rho' &= - G^{*}_7(\rho) \,, \\
\int_{0}^\rho \frac{\psi_2(\rho')}{W(\rho')} \frac{G^{*}(\rho')}{1-{\rho'}^{2}} \dd\rho' &= - G^{*}_7(\rho) I_7(\rho) + \int_{0}^\rho G^{*}_7(\rho') I_7'(\rho') \dd\rho' \,.
\end{align*}
Hence
\begin{align*}
- \psi_{1}(\rho) \int_{0}^\rho \frac{\psi_2(\rho')}{W(\rho')}\frac{G^{*}(\rho')}{1-{\rho'}^{2}}\dd\rho' + \psi_2(\rho) \int_{0}^\rho \frac{\psi_{1}(\rho')}{W(\rho')}\frac{G^{*}(\rho')}{1-{\rho'}^{2}}\dd\rho' = \psi_{1}(\rho) \int_{0}^\rho G_7^{*}(\rho') I_7'(\rho') \dd\rho'
\end{align*}
which is easily integrated so that the general solution to \Cref{InhomLinWave} for $d=7$ reads
\begin{equation*}
\widehat{g}_{1}(\rho) = \left( \frac{5}{3} + \rho^{2} \right)^{-2} \left( \alpha + \frac{1}{2} \log \frac{1 - \rho^{2}}{h_{1}(\rho)^{2}} \right) \,.
\end{equation*}
But the relation \eqref{SSCrelHSC} would imply for our solution $f_{1} \notin H^{k}(\BB^d_{R})$ due to the logarithmic singularity, no matter the value of $\alpha\in\CC$, which is a contradiction.
\item[\textit{Case $d\geq 9$ odd.}] We set
\begin{equation}
\label{FundSoldgeq9}
\psi_2(\rho) = - \psi_{1}(\rho) \int_\rho^1 \frac{(1-{\rho'}^{2})^{\frac{d-9}{2}}}{{\rho'}^{d-1}} \frac{1}{\psi_{1}(\rho')^{2}} \dd\rho' \,.
\end{equation}
Then $\psi_2\in C^\infty((0,1])$ has a zero of order $\frac{d-7}{2}$ at $\rho = 1$ and a pole of order $(d-2)$ at $\rho = 0$. From \Cref{WronskianSSC} we see that the Wronskian $W$ has a zero of order $\frac{d-9}{2}$ at $\rho = 1$ and a pole of order $(d-1)$ at $\rho = 0$. Hence it follows that
\begin{equation*}
\psi_{1}(\rho) \int_{0}^\rho \frac{\psi_2(\rho')}{W(\rho')} \frac{G^{*}(\rho')}{1-{\rho'}^{2}} \dd\rho'
\end{equation*}
belongs to $C^\infty([0,1])$, as the integrand defines a smooth function on $[0,1]$. However, as the remaining integral is precisely the Duhamel term for the corresponding problem in standard similarity coordinates, we infer from \cite[Eq. (3.41)]{MR4469070}
\begin{equation*}
\psi_2(\rho) \int_{0}^\rho \frac{\psi_{1}(\rho')}{W(\rho')} \frac{G^{*}(\rho')}{1-{\rho'}^{2}} \dd\rho' = \varphi_{1}(\rho) + \left( 1 - \rho \right)^{\frac{d-7}{2}} \log \left( 1 - \rho \right) \varphi_2(\rho)
\end{equation*}
where $\varphi_{1}$, $\varphi_2$ are analytic near $\rho = 1$ with $\varphi_{1}(1), \varphi_2(1) \neq 0$. Consequently $f_{1}\notin H^{k}(\BB^{d}_{R})$ by the relation \eqref{SSCrelHSC} which is a contradiction.
\end{enumerate}
\par\medskip
Thus, the assumption \eqref{nilpot} cannot be true and we conclude $\mathbf{I} - \mathbf{L}\restriction_{\ran(\mathbf{P})} = \mathbf{0}$. So $(\mathbf{I} - \mathbf{L})\mathbf{P}\mathbf{f} = \mathbf{0}$ for all $\mathbf{f}\in \HH^{k}_\rad(\BB^{d}_{R})$ and hence the other inclusion $\ran(\mathbf{P}) \subseteq \ker(\mathbf{I}-\mathbf{L}) = \langle\mathbf{f}_{1}^{*}\rangle$.
\qedhere
\end{enumerate}
\end{proof}
\subsection{Control of the linearized wave evolution}
We obtain with the projection $\mathbf{P}$ a decomposition of the Hilbert space $\HH^{k}_\rad(\BB^{d}_{R})$ into complementary closed subspaces,
\begin{equation*}
\HH^{k}_\rad(\BB^{d}_{R}) = \mathfrak{M}\oplus\mathfrak{N} \,, \qquad \mathfrak{M}=\ran(\mathbf{P}) \,, \quad \mathfrak{N}=\ran(\mathbf{I} - \mathbf{P}) \,.
\end{equation*}
In particular, from the fact $\mathbf{L}\mathbf{S}(s) \subset \mathbf{S}(s)\mathbf{L}$ it follows that $\mathbf{P}$ commutes with the semigroup $\mathbf{S}(s)$, so the subspace semigroups
\begin{equation*}
[0,\infty) \rightarrow \mathfrak{L}(\mathfrak{M}) \,, \quad s\mapsto \mathbf{S}(s)\restriction_\mathfrak{M} \,, \quad \text{and} \quad
[0,\infty) \rightarrow \mathfrak{L}(\mathfrak{N}) \,, \quad s\mapsto \mathbf{S}(s)\restriction_\mathfrak{N} \,,
\end{equation*}
with generators $\mathbf{L}\restriction_\mathfrak{M}$, $\mathbf{L}\restriction_\mathfrak{N}$ on $\mathfrak{M}$, $\mathfrak{N}$ are well-defined, respectively, see \cite[p. 60]{MR1721989}. The following main result about the linearized flow gives us exponential growth on the one-dimensional subspace $\mathfrak{M}$ on the one hand, and control via exponential decay on the infinite dimensional complementary subspace $\mathfrak{N}$.
\begin{theorem}
\label{stableunstable}
Let $R\geq \frac{1}{2}$, $k\in\NN$ such that $k\geq \frac{d-1}{2}$. Let $\mathbf{S}$ be the semigroup generated by the linearized wave evolution $\mathbf{L}$, see \Cref{LinEv}. Let $\mathbf{P}$ be the spectral projection associated to the isolated eigenvalue $1\in\sigma(\mathbf{L})$, see \Cref{SpecProjProp}.
\begin{enumerate}[itemsep=1em, topsep=1em]
\item We have
\begin{equation*}
\mathbf{S}(s)\mathbf{P}\mathbf{f} = \ee^s \mathbf{P}\mathbf{f}
\end{equation*}
for all $\mathbf{f}\in\HH^{k}_\rad(\BB^{d}_{R})$ and all $s\geq 0$.
\item \label{parttwo} There is a positive constant $\omega_{0}>0$ such that
\begin{equation*}
\| \mathbf{S}(s)(\mathbf{I} - \mathbf{P})\mathbf{f} \|_{\HH^{k}(\BB^{d}_{R})} \lesssim \ee^{-\omega_{0} s} \| (\mathbf{I} - \mathbf{P})\mathbf{f} \|_{\HH^{k}(\BB^{d}_{R})} \lesssim \ee^{-\omega_{0} s} \| \mathbf{f} \|_{\HH^{k}(\BB^{d}_{R})}
\end{equation*}
for all $\mathbf{f}\in\HH^{k}_\rad(\BB^{d}_{R})$ and all $s\geq 0$.
\end{enumerate}
\end{theorem}
\begin{proof}
\begin{enumerate}[wide, itemsep=1em, topsep=1em]
\item We see with \Cref{SpecProjProp} that the subspace semigroup $\mathbf{S}(s)\restriction_{\mathfrak{M}}$ acts on a one-dimensional space and is thus given by the exponential function.
\item The idea is to project away the unstable eigenspace and exploit the spectral gap. Since spectra are closed and $\sigma(\mathbf{L}\restriction_{\mathfrak{N}})$ is contained in the negative half-plane by \Cref{LemmaKerLin} and \Cref{SpecProjProp} we can pick $0 < \omega_{0} \leq \frac{3}{2}$ such that $-\frac{3}{2} \leq \sup\big\{ \Re(\lambda)\in\CC \mid \lambda\in \sigma(\mathbf{L}\restriction_{\mathfrak{N}}) \big\} < - \omega_{0} < 0$, see \Cref{Fig_Spec}. Moreover, from $\mathbf{R}_\mathbf{L}(z)\restriction_{\mathfrak{N}} = \mathbf{R}_{\mathbf{L}\restriction_{\mathfrak{N}}}(z)$ and the first part of \Cref{LemmaLinEvEV} we infer
\begin{equation*}
\| \mathbf{R}_{\mathbf{L}\restriction_{\mathfrak{N}}}(z) \|_{\mathfrak{L} (\mathfrak{N})} \lesssim 1
\end{equation*}
for all $z\in\CC$ with $\Re(z) \geq \omega_{0}$ and $|z| \geq r_{0}$. Since $\{ z\in\CC \mid \Re(z) \geq -\omega_{0} \} \subset \varrho(\mathbf{L}\restriction_{\mathfrak{N}})$ and the resolvent is analytic we get the same estimate $\| \mathbf{R}_{\mathbf{L}\restriction_{\mathfrak{N}}}(z) \|_{\mathfrak{L} (\mathfrak{N})} \lesssim 1$ on the compact set $\{ z\in\CC \mid \Re(z) \geq -\omega_{0},\, |z|\leq r_{0} \}$. Hence we have a uniform bound on the resolvent
\begin{equation*}
\| \mathbf{R}_{\mathbf{L}\restriction_{\mathfrak{N}}}(z) \|_{\mathfrak{L} (\mathfrak{N})} \lesssim 1
\end{equation*}
for all $z\in\CC$ with $\Re(z)\geq-\omega_{0}$. Thus, from the Gearhart-Pr\"uss Theorem \cite[p. 302, 1.11 Theorem]{MR1721989} and \cite[p. 61, Corollary]{MR1721989} we see that the subspace semigroup $\mathbf{S}(s)\restriction_{\mathfrak{N}}$ is exponentially stable and this establishes the second part. The very last bound follows from the fact that $\mathbf{P}$ is bounded.\qedhere
\end{enumerate}
\end{proof}
\begin{figure}
\centering
\includegraphics[width=0.45\textwidth]{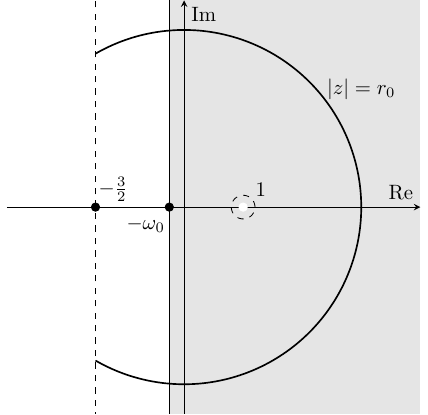}
\caption{Portrait of the spectrum for the linearized evolution $\mathbf{L}$. The shaded region is part of the resolvent set of the linearized evolution $\mathbf{L}$ and the isolated unstable eigenvalue $1$ is the only point in the spectrum contained in the positive half-plane.}
\label{Fig_Spec}
\end{figure}
\section{Stability analysis of the nonlinear wave evolution}
The full nonlinear problem outlined in \Cref{YMnewvariable} is treated in a weak formulation via Duhamel's formula,
\begin{equation}
\label{Duhamel}
\mathbf{\Phi}(s) = \mathbf{S}(s-s_{0})\mathbf{\Phi}(s_{0}) + \int_{s_{0}}^s \mathbf{S}(s-s')\mathbf{N}(\mathbf{\Phi}(s')) \dd s' \,.
\end{equation}
As this is naturally the equation for a ``small'' perturbation we can aim for a fixed point argument. For this we need to control the nonlinearity that appears in the Duhamel integral in the spaces where the linearized evolution is defined. Let us define
\begin{equation}
\mathbf{N}: C^\infty_\rad(\overline{\BB^d_{R}})^{2} \rightarrow C^\infty_\rad(\overline{\BB^d_{R}})^{2} \quad \text{by} \quad
\mathbf{N}(\mathbf{f})(y) = N(y, f_{1}(y)) \mathbf{e}_2 \,,
\end{equation}
where $N\in C^\infty(\RR^d\times\RR)$ is given in \Cref{nonlinYM} and show that this extends to $\HH^{k}_\rad(\BB^d_{R})$. The following result gives us a local Lipschitz bound on the nonlinearity.
\begin{lemma}
\label{nonlinBound}
Let $R>0$ and $k\in\NN$ such that $k\geq \frac{d+1}{2}$. Then the nonlinearity $\mathbf{N}$ extends to a map $\mathbf{N}: \HH^{k}_\rad(\BB^d_{R}) \rightarrow \HH^{k}_\rad(\BB^d_{R})$ which satisfies the local Lipschitz bound
\begin{equation*}
\| \mathbf{N}(\mathbf{f}) -  \mathbf{N}(\mathbf{g}) \|_{\HH^{k}(\BB^{d}_{R})} \lesssim \big( \|\mathbf{f}\|_{\HH^{k}(\BB^{d}_{R})} + \|\mathbf{g}\|_{\HH^{k}(\BB^{d}_{R})} \big) \| \mathbf{f} - \mathbf{g} \|_{\HH^{k}(\BB^{d}_{R})}
\end{equation*}
for all $\mathbf{f},\mathbf{g}\in\HH^{k}(\BB^{d}_{R})$ that are bounded in size by $\|\mathbf{f}\|_{\HH^{k}(\BB^{d}_{R})},\|\mathbf{g}\|_{\HH^{k}(\BB^{d}_{R})} \lesssim 1$.
\end{lemma}
\begin{proof}
Recall $N \in C^\infty(\RR^d\times\RR)$ from \Cref{nonlinYM} and note
\begin{align*}
N(y,\alpha) - N(y,\beta) =
- (d-4) \frac{\left( |y|h'(|y|) - h(|y|) \right)^{2}}{1 - h'(|y|)^{2}} &\Bigg( 3\frac{(1-a_d)|y|^{2} + b_d h(y)^{2}}{|y|^{2} + b_dh(y)^{2}} (\alpha + \beta) \\& - 3|y|^{2} (\alpha^{2} + \alpha\beta + \beta^{2}) \Bigg) (\alpha - \beta)
\end{align*}
for all $y\in\RR^d$ and all $\alpha,\beta\in\RR$. Since $k \geq \frac{d+1}{2}> \frac{d}{2}$ we can use the algebra property of $H^{k}(\BB^{d}_{R}) \hookrightarrow C(\overline{\BB^{d}_{R}})$ to estimate
\begin{align*}
\| \mathbf{N}(\mathbf{f}) -  \mathbf{N}(\mathbf{g}) \|_{\HH^{k}(\BB^{d}_{R})} &= \| N(\,.\,, f_{1}) - N(\,.\,,g_{1}) \|_{H^{k-1}(\BB^d_{R})} \\&\lesssim
\| N(\,.\,, f_{1}) - N(\,.\,,g_{1}) \|_{H^{k}(\BB^d_{R})} \\&\lesssim
\big( \|f_{1}\|_{H^{k}(\BB^d_{R})} +  \|g_{1}\|_{H^{k}(\BB^d_{R})} \big) \|f_{1} - g_{1} \|_{H^{k}(\BB^d_{R})} \\&\lesssim
\big( \|\mathbf{f}\|_{\HH^{k}(\BB^{d}_{R})} + \|\mathbf{g}\|_{\HH^{k}(\BB^{d}_{R})} \big) \| \mathbf{f} - \mathbf{g} \|_{\HH^{k}(\BB^{d}_{R})}
\end{align*}
if $\|\mathbf{f}\|_{\HH^{k}(\BB^{d}_{R})},\|\mathbf{g}\|_{\HH^{k}(\BB^{d}_{R})} \lesssim 1$. The remaining extension argument is standard.
\end{proof}
\subsection{nonlinear wave evolution on the co-dimension one stable subspace}
We run the fixed point argument in a Banach space that is adapted to the exponential decay of the linearized evolution on the stable subspace. Henceforth, fix $0 < \omega_{0} < \frac{3}{2}$ from \Cref{stableunstable} and define for $R,s_{0}\in\RR$, $k\in\NN$ such that $R\geq \frac{1}{2}$, $k\geq \frac{d+1}{2}$, a Banach space
\begin{equation*}
\mathfrak{X}^{k}_{R}(s_{0}) \coloneqq \left\{ \mathbf{\Phi}\in C \big( [s_{0},\infty),\HH^{k}_\rad(\BB^{d}_{R}) \big) \mid \| \mathbf{\Phi} \|_{\mathfrak{X}^{k}_{R}(s_{0})} \coloneqq \sup_{s\geq s_{0}} \left( \ee^{\omega_{0} s} \| \mathbf{\Phi}(s) \|_{\HH^{k}(\BB^{d}_{R})} \right) < \infty \right\} \,.
\end{equation*}
There comes an obstruction to general global in time existence for \Cref{Duhamel} that is caused by the one-dimensional instability due to the symmetry mode $\mathbf{f}_{1}^{*}$. To overcome this, one considers the projection of the equation on the stable subspace and notes that this corresponds to subtracting a term from the initial data. A slight modification leads to a right-hand side
\begin{equation*}
\mathbf{K}_\mathbf{f}(\mathbf{\Phi})(s) \coloneqq \mathbf{S}(s-s_{0}) \big( \mathbf{f} - \mathbf{C}_{s_{0}}(\mathbf{\Phi},\mathbf{f}) \big) + \int_{s_{0}}^s  \mathbf{S}(s-s') \mathbf{N}(\mathbf{\Phi(s')}) \dd s'
\end{equation*}
with a correction $\mathbf{C}_{s_{0}}: \mathfrak{X}^{k}_{R}(s_{0})\times\HH^{k}_\rad(\BB^{d}_{R}) \rightarrow \HH^{k}_\rad(\BB^{d}_{R})$ given by
\begin{equation*}
\mathbf{C}_{s_{0}}(\mathbf{\Phi},\mathbf{f}) = \mathbf{P}\left( \mathbf{f} + \int_{s_{0}}^\infty \ee^{s_{0}-s'}\mathbf{N}(\mathbf{\Phi}(s')) \dd s' \right) \,.
\end{equation*}
This is known as the \emph{Lyapunov-Perron method} and has been implemented in \cite{MR4338226}, \cite{MR3680948}, \cite{MR3742520}. Now, global existence is anticipated for the modified equation.
\begin{proposition}
\label{stableEvo}
Let $R,s_{0}\in\RR$, $k\in\NN$ such that $R\geq \frac{1}{2}$ and $k\geq \frac{d+1}{2}$.
\begin{enumerate}[itemsep=1em, topsep=1em]
\item There are positive constants $C,\delta>0$ such that for all $\mathbf{f}\in\HH^{k}_\rad(\BB^{d}_{R})$ with $\| \mathbf{f} \|_{\HH^{k}(\BB^{d}_{R})} \leq \delta/C$ there exists a unique solution $\mathbf{\Phi}_\mathbf{f} \in \mathfrak{X}^{k}_{R}(s_{0})$ with $\| \mathbf{\Phi}_\mathbf{f} \|_{\mathfrak{X}^{k}_{R}(s_{0})} \leq \delta$ to the fixed point problem
\begin{equation*}
\mathbf{K}_\mathbf{f}(\mathbf{\Phi}) = \mathbf{\Phi} \,.
\end{equation*}
\item The problem is well-posed in the sense that $\mathbf{f}\mapsto \mathbf{\Phi}_\mathbf{f}$ is Lipschitz as a function from a small ball in $\HH^{k}_\rad(\BB^{d}_{R})$ to $\mathfrak{X}^{k}_{R}(s_{0})$.
\end{enumerate}
\end{proposition}
\begin{proof}
\begin{enumerate}[wide, itemsep=1em, topsep=1em]
\item We define $\B_\delta = \{ \mathbf{\Phi}\in\mathfrak{X}^{k}_{R}(s_{0}) \mid \| \mathbf{\Phi} \|_{\mathfrak{X}^{k}_{R}(s_{0})} \leq \delta \}$ as a closed ball in our Banach space. We decompose $\mathbf{K}_\mathbf{f}(\mathbf{\Phi})$ according to
\begin{align*}
\mathbf{P}\mathbf{K}_\mathbf{f}(\mathbf{\Phi})(s) &= -\int_s^\infty \ee^{s-s'} \mathbf{P}\mathbf{N}(\mathbf{\Phi}(s')) \dd s' \,,  \\
(\mathbf{I} - \mathbf{P})\mathbf{K}_\mathbf{f}(\mathbf{\Phi})(s) &= \mathbf{S}(s - s_{0}) (\mathbf{I} - \mathbf{P})\mathbf{f} + \int_{s_{0}}^s \mathbf{S}(s - s') (\mathbf{I} - \mathbf{P})\mathbf{N}(\mathbf{\Phi}(s')) \dd s' \,.
\end{align*}
Suppose $\mathbf{\Phi} \in \B_\delta$. Using boundedness of $\mathbf{P}$ and \Cref{nonlinBound} with $\mathbf{N}(\mathbf{0}) = \mathbf{0}$ we estimate
\begin{equation*}
\| \mathbf{P}\mathbf{K}_\mathbf{f}(\mathbf{\Phi})(s) \|_{\HH^{k}(\BB^{d}_{R})} \lesssim
\int_s^\infty \ee^{s-s'} \| \mathbf{\Phi}(s') \|_{\HH^{k}(\BB^{d}_{R})}^{2} \dd s' \lesssim
\int_s^\infty \ee^{s-s'} \ee^{-2\omega_{0} s'} \| \mathbf{\Phi} \|_{\mathfrak{X}^{k}_{R}(s_{0})}^{2} \dd s' \lesssim
\delta^{2} \ee^{-2\omega_{0} s} \,,
\end{equation*}
and with the bound from \Cref{stableunstable}
\begin{align*}
\| (\mathbf{I} - \mathbf{P})\mathbf{K}_\mathbf{f}(\mathbf{\Phi})(s) \|_{\HH^{k}(\BB^{d}_{R})} &\lesssim
\ee^{-\omega_{0} (s-s_{0})} \| \mathbf{f}\|_{\HH^{k}(\BB^{d}_{R})} + \int_{s_{0}}^s \ee^{-\omega_{0}(s-s')}  \| \mathbf{\Phi}(s') \|_{\HH^{k}(\BB^{d}_{R})}^{2} \dd s' \\&\lesssim
\frac{\delta}{C} \ee^{-\omega_{0}s} + \int_{s_{0}}^s \ee^{-\omega_{0}(s-s')} \ee^{-2\omega_{0}s'}  \| \mathbf{\Phi} \|_{\mathfrak{X}^{k}_{R}(s_{0})}^{2} \dd s' \\&\lesssim
\frac{\delta}{C} \ee^{-\omega_{0}s} + \delta^{2} \ee^{-\omega_{0} s} \,.
\end{align*}
This shows $\mathbf{K}_\mathbf{f}(\mathbf{\Phi}) \in \B_\delta$ if $C>0$ is large and $\delta>0$ small enough.
\par\medskip
Next, in order to prove that $\mathbf{K}_\mathbf{f}$ is a contraction let $\mathbf{\Phi}, \mathbf{\Psi} \in \B_\delta$ and consider
\begin{align*}
\mathbf{P} \big( \mathbf{K}_\mathbf{f}(\mathbf{\Phi})(s) - \mathbf{K}_\mathbf{f}(\mathbf{\Psi})(s) \big) &=
-\int_s^\infty \ee^{s-s'} \mathbf{P} \big( \mathbf{N}(\mathbf{\Phi}(s')) - \mathbf{N}(\mathbf{\Psi}(s')) \big) \dd s' \,, \\
\big(\mathbf{I} - \mathbf{P}\big) \big( \mathbf{K}_\mathbf{f}(\mathbf{\Phi})(s) - \mathbf{K}_\mathbf{f}(\mathbf{\Psi})(s) \big) &=
\int_{s_{0}}^s \mathbf{S}(s - s') \big(\mathbf{I} - \mathbf{P} \big) \big( \mathbf{N}(\mathbf{\Phi}(s')) - \mathbf{N}(\mathbf{\Psi}(s')) \big) \dd s' \,.
\end{align*}
Indeed, as above
\begin{align*}
&\| \mathbf{P} \big( \mathbf{K}_\mathbf{f}(\mathbf{\Phi})(s) - \mathbf{K}_\mathbf{f}(\mathbf{\Psi})(s) \big) \|_{\HH^{k}(\BB^{d}_{R})} \\&\lesssim
\int_s^\infty \ee^{s-s'} \big( \| \mathbf{\Phi}(s') \|_{\HH^{k}(\BB^{d}_{R})} + \| \mathbf{\Psi}(s') \|_{\HH^{k}(\BB^{d}_{R})} \big) \| \mathbf{\Phi}(s') - \mathbf{\Psi}(s') \|_{\HH^{k}(\BB^{d}_{R})} \dd s' \\&\lesssim
\int_s^\infty \ee^{s-s'} \ee^{-2\omega_{0} s'} \big( \| \mathbf{\Phi} \|_{\mathfrak{X}^{k}_{R}(s_{0})} + \| \mathbf{\Psi} \|_{\mathfrak{X}^{k}_{R}(s_{0})} \big) \| \mathbf{\Phi} - \mathbf{\Psi} \|_{\mathfrak{X}^{k}_{R}(s_{0})} \dd s' \\&\lesssim \delta \ee^{-2\omega_{0}s} \| \mathbf{\Phi} - \mathbf{\Psi} \|_{\mathfrak{X}^{k}_{R}(s_{0})} \,,
\end{align*}
and
\begin{align*}
&\| \big(\mathbf{I} - \mathbf{P}\big) \big(\mathbf{K}_\mathbf{f}(\mathbf{\Phi})(s) - \mathbf{K}_\mathbf{f}(\mathbf{\Psi})(s) \big) \|_{\HH^{k}(\BB^{d}_{R})} \\&\lesssim
\int_{s_{0}}^s \ee^{-\omega_{0}(s-s')} \big( \| \mathbf{\Phi}(s') \|_{\HH^{k}(\BB^{d}_{R})} + \| \mathbf{\Psi}(s') \|_{\HH^{k}(\BB^{d}_{R})} \big) \| \mathbf{\Phi}(s') - \mathbf{\Psi}(s') \|_{\HH^{k}(\BB^{d}_{R})} \dd s' \\&\lesssim
\int_{s_{0}}^s \ee^{-\omega_{0}(s-s')} \ee^{-2\omega_{0} s'} \big( \| \mathbf{\Phi} \|_{\mathfrak{X}^{k}_{R}(s_{0})} + \| \mathbf{\Psi} \|_{\mathfrak{X}^{k}_{R}(s_{0})} \big) \| \mathbf{\Phi} - \mathbf{\Psi} \|_{\mathfrak{X}^{k}_{R}(s_{0})} \dd s' \\&\lesssim
\delta \ee^{-\omega_{0} s} \| \mathbf{\Phi} - \mathbf{\Psi} \|_{\mathfrak{X}^{k}_{R}(s_{0})} \,.
\end{align*}
Thus
\begin{equation*}
\| \mathbf{K}_\mathbf{f}(\mathbf{\Phi}) - \mathbf{K}_\mathbf{f}(\mathbf{\Psi}) \|_{\mathfrak{X}^{k}_{R}(s_{0})} \lesssim \delta \| \mathbf{\Phi} - \mathbf{\Psi} \|_{\mathfrak{X}^{k}_{R}(s_{0})}
\end{equation*}
for all $\mathbf{\Phi}, \mathbf{\Psi} \in \B_\delta$. It follows that $\mathbf{K}_\mathbf{f}: \B_\delta \rightarrow \B_\delta$ is a contraction and Banach's fixed point theorem yields existence and uniqueness of a fixed point $\mathbf{\Phi}_\mathbf{f} \in \B_\delta$.
\item To see Lipschitz continuous dependence on the initial data, note with the triangle inequality
\begin{align*}
\| \mathbf{\Phi}_\mathbf{f} - \mathbf{\Phi}_\mathbf{g} \|_{\mathfrak{X}^{k}_{R}(s_{0})} &= \| \mathbf{K}_\mathbf{f}(\mathbf{\Phi}_\mathbf{f}) - \mathbf{K}_\mathbf{g}(\mathbf{\Phi}_\mathbf{g}) \|_{\mathfrak{X}^{k}_{R}(s_{0})} \\&\leq
\| \mathbf{K}_\mathbf{f}(\mathbf{\Phi}_\mathbf{f}) - \mathbf{K}_\mathbf{f}(\mathbf{\Phi}_\mathbf{g}) \|_{\mathfrak{X}^{k}_{R}(s_{0})} + \| \mathbf{K}_\mathbf{f}(\mathbf{\Phi}_\mathbf{g}) - \mathbf{K}_\mathbf{g}(\mathbf{\Phi}_\mathbf{g}) \|_{\mathfrak{X}^{k}_{R}(s_{0})} \,.
\end{align*}
Since we have $\| \mathbf{K}_\mathbf{f}(\mathbf{\Phi}_\mathbf{f}) - \mathbf{K}_\mathbf{f}(\mathbf{\Phi}_\mathbf{g}) \|_{\mathfrak{X}^{k}_{R}(s_{0})} \lesssim \delta \| \mathbf{\Phi}_\mathbf{f} - \mathbf{\Phi}_\mathbf{g} \|_{\mathfrak{X}^{k}_{R}(s_{0})}$ from the first item and
\begin{equation*}
\mathbf{K}_\mathbf{f}(\mathbf{\Phi}_\mathbf{g})(s) - \mathbf{K}_\mathbf{g}(\mathbf{\Phi}_\mathbf{g})(s) = \mathbf{S}(s-s_{0})(\mathbf{I} - \mathbf{P})(\mathbf{f} - \mathbf{g}) \,,
\end{equation*}
we obtain from the bound in \Cref{stableunstable} altogether
\begin{equation*}
\| \mathbf{\Phi}_\mathbf{f} - \mathbf{\Phi}_\mathbf{g} \|_{\mathfrak{X}^{k}_{R}(s_{0})} \lesssim \delta \| \mathbf{\Phi}_\mathbf{f} - \mathbf{\Phi}_\mathbf{g} \|_{\mathfrak{X}^{k}_{R}(s_{0})} + \| \mathbf{f} - \mathbf{g} \|_{\HH^{k}(\BB^{d}_{R})} \,.
\end{equation*}
If $\delta>0$ is small enough this implies that the dependence on the initial data is Lipschitz continuous.\qedhere
\end{enumerate}
\end{proof}
\subsection{Preparation of hyperboloidal initial data}
Finally, we begin to study the Cauchy problem for the Yang-Mills equation with prescribed initial data at time $t=0$.
\par\medskip
Let $\delta,\varepsilon>0$ and $k\in\NN_{0}$. The set of compactly supported smooth radial perturbations of the blowup is defined by
\begin{equation*}
\B^{k}_{\delta,\varepsilon} \coloneqq \left\{ \mathbf{f}\in C^\infty_\mathrm{rad}(\RR^d) \times C^\infty_\mathrm{rad}(\RR^d) \mid \supp(\mathbf{f})\subseteq\BB^d_\varepsilon\times\BB^d_\varepsilon,\, \| \mathbf{f} \|_{\HH^{k}(\RR^d)} \leq \delta \right\} \,.
\end{equation*}
The union of the complement of the domain of influence of this ball with a truncated light cone gives rise to the domain
\begin{equation*}
\Lambda_\varepsilon \coloneqq \left\{ (t,x) \in \RR^{1,d} \mid |x|\geq r_\varepsilon(t) \right\} \,, \qquad r_\varepsilon(t) =
\begin{cases}
|t| + \varepsilon \,, & |t| > 4\varepsilon \,, \\
0 \,, & |t| \leq 4\varepsilon \,,
\end{cases}
\end{equation*}
see \Cref{Fig_LocExis}. The initial data for the hyperboloidal evolution set out in \Cref{YMnewvariable} are going to be prepared by evolving the data locally around $t=0$ and then finding a suitable hyperboloid where we can evaluate the local solution. For this to work out, it is crucial that the resulting shape of $\Lambda_\varepsilon$ allows to fit a family of initial hyperboloids inside, which is seen as follows.
\begin{lemma}
\label{FitHyp}
Let $\varepsilon>0$ and put $s_{0} = \log\big( - \tfrac{h(0)}{1+2\varepsilon}
\big)$. For each $T\in\overline{\BB_\varepsilon(1)}$,
the hyperboloid $\eta_{T}\big( s_{0}, \RR^d \big)$ is contained in the interior of $\Lambda_\varepsilon$.
\end{lemma}
\begin{proof}
Each hyperboloid $\eta_{T}(s,\RR^d)$ that has its tip in $[-3\varepsilon,-\varepsilon]\times\{0\} \subset \Lambda_\varepsilon$ is ensured to stay in the interior of $\Lambda_\varepsilon$ simply because it has slope less than one and starts below the hypersurface $\{ (t,x)\in\RR^{1,d} \mid |x| = t+\varepsilon \}$. We get for each $T\in\overline{\BB_\varepsilon(1)}$ a hyperboloid whose tip satisfies this condition by putting the initial hyperboloidal time to $s=s_{0}$, where $1 + \ee^{-s_{0}}h(0) = -2\varepsilon$, i.e. $s_{0} = \log\big( - \tfrac{h(0)}{1+2\varepsilon}
\big)$. Thus $\eta_{T}\big( s_{0}, \RR^d \big) \subset \Lambda_\varepsilon$ for all $T\in\overline{\BB_\varepsilon(1)}$.
\end{proof}
We turn to a local existence result for the classical Cauchy evolution in order to construct initial data for the hyperboloidal evolution.
\begin{lemma}
\label{locCauchy}
Let $k\in\NN_{0}$.
\begin{enumerate}[itemsep=1em, topsep=1em]
\item There exists an $\varepsilon>0$ such that for any pair of initial data $(f,g)\in\B^{k}_{1,\varepsilon}$ there exists a unique radial solution $u_{f,g}\in C^\infty(\Lambda_\varepsilon)$ to the Cauchy problem
\begin{equation}
\label{YMCauchy}
\begingroup
\renewcommand{\arraystretch}{1.3}
\left\{
\begin{array}{rcll}
(-\pd_{t}^{2} + \Delta_{x})u(t,x) &=& (d-4)(|x|^{2}u(t,x)^3 + 3u(t,x)^{2}) \,, & (t,x) \in \Lambda_\varepsilon \,, \\
u(0,x) &=& u_{1}^{*}(0,x) + f(x) \,, & x\in\RR^d \,,  \\
\pd_{0}u(0,x) &=& \pd_{0}u_{1}^{*}(0,x) + g(x) \,, & x\in\RR^d \,,
\end{array}
\right.
\endgroup
\end{equation}
satisfying $u_{f,g}(t,x) = u_{1}^{*}(t,x)$ when $|x| \geq |t|+\varepsilon$.
\item For any multi-index $\alpha\in\NN_{0}^{1+d}$ of length $|\alpha| \leq k$ we have the estimate
\begin{equation}
\label{Lipbound}
\sup_{(t,x)\in\Lambda_\varepsilon} | \pd^\alpha u_{f,g}(t,x) - \pd^\alpha u_{1}^{*}(t,x) | \lesssim \| (f,g) \|_{\HH^{k+\frac{d+1}{2}}(\RR^d)} \,.
\end{equation}
\end{enumerate}
\end{lemma}
\begin{proof}
\begin{enumerate}[wide, itemsep=1em, topsep=1em]
\item The classical Cauchy theory applies to \Cref{YMCauchy}, cf.  \cite[Section 2]{MR4338226} as well as the proof of \cite[Lemma 5.3]{MR4338226}. That is, by \cite[Theorem 2.11]{MR4338226} we get an $\varepsilon>0$ with $8\varepsilon\in(0,1)$ and the existence of a solution in the truncated light cone $ \bigcup_{t\in[-4\varepsilon,4\varepsilon]} \{t\}\times\BB_{1-|t|}^d =  \left\{ (t,x)\in\RR^{1,d} \mid |t| \leq 4\varepsilon,\, |x| < 1-|t| \right\}$ for any $(f,g)\in \B^{k}_{1,1}$. Fix this $\varepsilon>0$ and consider the solution $u$ for data $(f,g)\in \B^{k}_{1,\varepsilon}\subset\B^{k}_{1,1}$. Since $\supp(f,g) \subset \BB_\varepsilon^d\times\BB_\varepsilon^d$, finite speed of propagation \cite[Theorem 2.12]{MR4338226} implies that $u_{1}^{*}$ is the unique solution in the domain $\left\{ (t,x)\in\RR^{1,d} \mid |x| \geq |t|+\varepsilon \right\}$. By construction, this domain intersects with the domain of $u$ so that $u$ equals $u_{1}^{*}$ on the overlap, see \Cref{Fig_LocExis}. Also note that the union of the truncated light cone with the complement of the domain of influence of the balls where the initial data are supported is precisely $\Lambda_\varepsilon$. Using \cite[Theorem 2.14]{MR4338226}, we constructed a unique solution $u\in C^\infty(\Lambda_\varepsilon)$.
\item The estimate is a result of the Lipschitz continuous dependence on the initial data, which follows from the classical theory and Sobolev embedding, see \cite[Theorem 2.11, Lemma 5.3]{MR4338226}.
\qedhere
\end{enumerate}
\end{proof}
Now, according to \Cref{FitHyp} and the first part of \Cref{locCauchy}, there is an $\varepsilon>0$ so that the radial solution $u_\mathbf{f}\in C^\infty(\Lambda_\varepsilon)$ to the Cauchy problem \eqref{YMCauchy} can be evaluated pointwise along the hyperboloid $\eta_{T}(s_{0},\RR^{d})$. Similarly, since for fixed $T\in\overline{\BB_{\varepsilon}(1)}$ there is wiggle room for $s$ around $s_{0}$ so that $\eta_{T}(s,\RR^{d})$ stays within $\Lambda_\varepsilon$, the derivative $\pd_{s} \big( u_{\mathbf{f}}\circ\eta_{T}(s,\,.\,) \big)|_{s=s_{0}}$ exists along the hyperboloid. Moreover, we see with the second part of \Cref{locCauchy} that in order to control the evaluation of $u_{\mathbf{f}}$ and its derivatives along $\eta_{T}(s_{0},\RR^{d})$ through $L^{\infty}$-embedding, we need sufficient regularity for the Sobolev norm of the initial data $\mathbf{f}$. This leads to a well-defined \emph{initial data operator}
\begin{equation*}
\mathbf{U}: \B^{k+\frac{d+1}{2}}_{1,\varepsilon}\times\overline{\BB_\varepsilon(1)} \rightarrow \HH^{k}_\rad(\BB^{d}_{R}) \,, \quad
\mathbf{U}(\mathbf{f},T)(y) =
\ee^{-2s_{0}}
\begin{bmatrix}
\hspace*{\fill}\big( (u_{\mathbf{f}} - u_{T}^{*})\circ\eta_{T} \big)(s_{0},y) \\
\pd_{0}\big( (u_{\mathbf{f}} - u_{T}^{*})\circ\eta_{T} \big)(s_{0},y)
\end{bmatrix}
\end{equation*}
for any initial data $\mathbf{f}\in\B^{k+\frac{d+1}{2}}_{1,\varepsilon}$ and blowup time $T\in\overline{\BB_\varepsilon(1)}$. Next, we prove mapping properties of $\mathbf{U}$.
\begin{figure}
\centering
\includegraphics[width=0.8\textwidth]{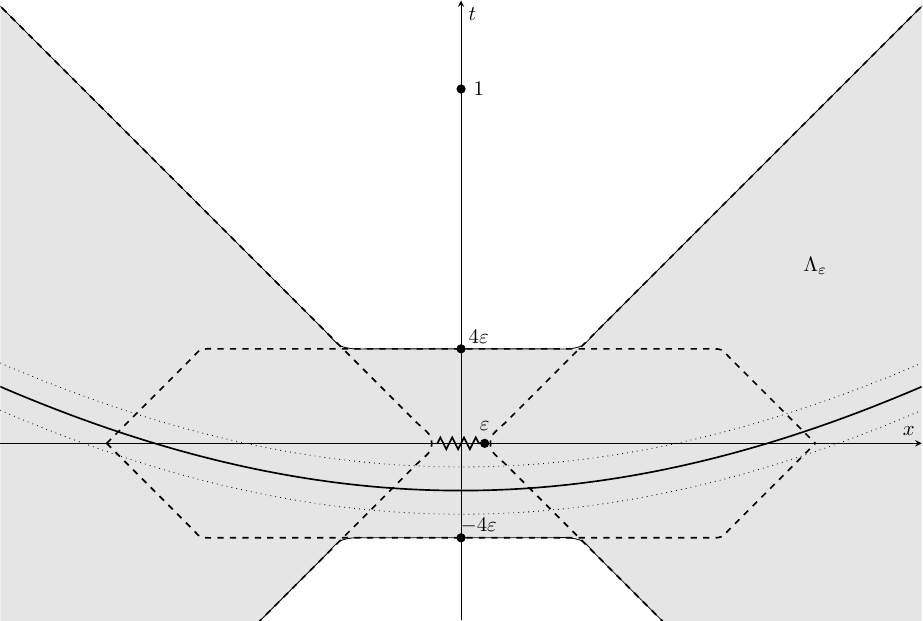}
\caption{Illustration for the preparation of initial data in the plane. The zigzag line marks the support for the perturbations of initial data, the emerging dashed lines indicate their domain of influence. The other dashed lines delimit the domain of local existence for the classical Cauchy evolution. The grey shaded region depicts the union $\Lambda_\varepsilon$ of these two domains. By construction and the choice $s_{0} = \log(-\tfrac{h(0)}{1+2\varepsilon})$, the initial hyperboloid $y\mapsto(T + \ee^{-s_{0}}h(y),\ee^{-s_{0}}y)$ lies within $\Lambda_\varepsilon$ for all $T\in\overline{\BB_\varepsilon(1)}$. The solid hyperboloid is drawn for $T=1$, the dotted hyperboloids for $T=1\pm\varepsilon$, respectively.}
\label{Fig_LocExis}
\end{figure}
\begin{lemma}
\label{InitDatOpProp}
Let $R>0$ and $k\in\NN_{0}$. Consider $\varepsilon>0$ such that the initial data operator $\mathbf{U}: \B^{k+\frac{d+1}{2}}_{1,\varepsilon}\times\overline{\BB_\varepsilon(1)} \rightarrow \HH^{k}_\rad(\BB^{d}_{R})$ is well-defined according to \Cref{locCauchy}.
\begin{enumerate}[itemsep=1em, topsep=1em]
\item The map $\mathbf{U}(\mathbf{f},\,.\,): \overline{\BB_\varepsilon(1)} \rightarrow \HH^{k}_\rad(\BB^{d}_{R})$ is continuous for every fixed $\mathbf{f}\in \B^{k+\frac{d+1}{2}}_{1,\varepsilon}$.
\item There exists a $C_\varepsilon > 0$ such that
\begin{equation}
\label{InitDatOpExp}
\mathbf{U}( \mathbf{f}, T) = C_\varepsilon(T-1)\mathbf{f}_{1}^{*} + \widetilde{\mathbf{U}}( \mathbf{f}, T) \,,
\end{equation}
where the remainder term satisfies
\begin{equation}
\label{InitDatOpEst}
\| \widetilde{\mathbf{U}}(\mathbf{f},T) \|_{\HH^{k}(\BB^{d}_{R})} \lesssim \| \mathbf{f} \|_{\HH^{k+\frac{d+1}{2}}(\RR^d)} + |T-1|^{2}
\end{equation}
for all $(\mathbf{f},T)\in \B^{k+\frac{d+1}{2}}_{1,\varepsilon}\times\overline{\BB_\varepsilon(1)}$.
\end{enumerate}
\end{lemma}
\begin{proof}
\begin{enumerate}[wide, itemsep=1em, topsep=1em]
\item The solutions $u_\mathbf{f}$, $u_{T}^{*}$ are smooth in $\Lambda_\varepsilon$ and the dependence on $T$ is also smooth. As the hyperboloids $y\mapsto(T + \ee^{-s}h(y),\ee^{-s}y)$ lie within $\Lambda_\varepsilon$ for all $T\in\overline{\BB_\varepsilon(1)}$ and $s$ sufficiently close to $s_{0}$ by \Cref{FitHyp}, continuous dependence on the blowup time $T$ follows for $\mathbf{U}$ from smooth dependence of the components of the operator.
\item We decompose $\mathbf{U}(\mathbf{f},T) = \mathbf{U}(\mathbf{0},T) + \mathbf{U}(\mathbf{f},T) - \mathbf{U}(\mathbf{0},T)$ and start noting
\begin{equation*}
\mathbf{U}(\mathbf{0},T)(y) =
\ee^{-2s_{0}}
\begin{bmatrix}
\hspace*{\fill}\big( (u_{1}^{*} - u_{T}^{*})\circ\eta_{T} \big)(s_{0},y) \\
\pd_{0}\big( (u_{1}^{*} - u_{T}^{*})\circ\eta_{T} \big)(s_{0},y)
\end{bmatrix}
\,.
\end{equation*}
By time translation, our self-similar blowup satisfies $u_{T+t_{0}}^{*}(t+t_{0},x) = u_{T}^{*}(t,x)$ for all $t_{0}\in\RR$. For $s$ close to $s_{0}$, a Taylor expansion for
\begin{equation*}
u_{1}^{*}(T+\ee^{-s}h(y),\ee^{-s}y) = - \frac{a_d}{b_d \big(1-T-\ee^{-s}h(y)\big)^{2} + \ee^{-2s}|y|^{2}}
\end{equation*}
around $T=1$ shows
\begin{align*}
&
u_{1}^{*}(T+\ee^{-s}h(y),\ee^{-s}y) \\&\indent=
u_{1}^{*}(1+\ee^{-s}h(y),\ee^{-s}y) + \left.\pd_{T} u_{1}^{*}(T+\ee^{-s}h(y),\ee^{-s}y)\right|_{T=1} (T-1) + r(T,s,y) (T-1)^{2}  \\&\indent=
u_{T}^{*}(T+\ee^{-s}h(y),\ee^{-s}y) + 2 a_d b_d \ee^{3s} f_{1,1}^{*}(y) (T-1) + r(T,s,y) (T-1)^{2} \,,
\end{align*}
where
\begin{equation*}
r(T,s,y) = \int_{0}^1 \left.\pd_{T'}^{2} u_{1}^{*}(T'+\ee^{-s}h(y),\ee^{-s}y) \right|_{T' = 1+\tau(1-T)} \tfrac{(1-\tau)^{2}}{2} \dd \tau
\end{equation*}
defines a jointly smooth function for $T\in\overline{\BB_\varepsilon(1)}$ and $s$ close to $s_{0}$. This also gives
\begin{align*}
&
\pd_s u_{1}^{*}(T+\ee^{-s}h(y),\ee^{-s}y) \\&\indent=
\pd_s u_{T}^{*}(T+\ee^{-s}h(y),\ee^{-s}y) + 2 a_d b_d \ee^{3s} f_{1,2}^{*}(y) (T-1) + \pd_s r(T,s,y) (T-1)^{2}
\end{align*}
so that
\begin{equation*}
\mathbf{U}(\mathbf{0},T)(y) = 2 a_d b_d \ee^{s_{0}} \mathbf{f}_{1}^{*}(y)(T-1) + \ee^{-2s_{0}} \mathbf{r}(T,s_{0},y)(T-1)^{2}
\end{equation*}
with
\begin{equation*}
\| \mathbf{r}(T,s_{0},\,.\,) \|_{\HH^{k}(\BB^{d}_{R})} \lesssim 1
\end{equation*}
for all $T\in\overline{\BB_\varepsilon(1)}$. Finally, we estimate
\begin{equation*}
\mathbf{U}( \mathbf{f}, T) - \mathbf{U}( \mathbf{0}, T) =
\ee^{-2s_{0}}
\begin{bmatrix}
\hspace{\fill}\big( (u_\mathbf{f} - u_{1}^{*})\circ\eta_{T} \big)(s_{0},\,.\,) \\
\pd_{0}\big( (u_\mathbf{f} - u_{1}^{*})\circ\eta_{T} \big)(s_{0},\,.\,)
\end{bmatrix}
\end{equation*}
with part two of \Cref{locCauchy}. Indeed, using the chain rule and product rule and estimating products of the from $\pd^\alpha {\eta_{T}}^\mu(s_{0},\,.\,)$ in $L^{2}(\BB^{d}_{R})$, we get for the first component
\begin{align*}
\| (u_\mathbf{f} - u_{1}^{*})\circ\eta_{T}(s_{0},\,.\,) \|_{H^{k}(\BB^{d}_{R})} &\simeq
\sum_{|\alpha|\leq k} \| \pd^\alpha\big( (u_\mathbf{f} - u_{1}^{*})\circ\eta_{T} \big)(s_{0},\,.\,) \|_{L^{2}(\BB^{d}_{R})} \\&\lesssim
\sum_{|\beta|\leq k} \sup_{y\in\BB^{d}_{R}} \big| ( \pd^\beta(u_\mathbf{f} - u_{1}^{*})\circ\eta_{T})(s_{0},y)\big| \\&\lesssim
\| \mathbf{f} \|_{\HH^{k+\frac{d+1}{2}}(\RR^d)} \,,
\end{align*}
where we used the space-time bounds \eqref{Lipbound} provided in \Cref{locCauchy}. With the same estimates we infer $\| \pd_{0}\big( (u_\mathbf{f} - u_{1}^{*})\circ\eta_{T} \big)(s_{0},\,.\,) \|_{H^{k-1}(\BB^{d}_{R})} \lesssim \| \mathbf{f} \|_{\HH^{k+\frac{d+1}{2}}(\RR^d)}$ and the bound follows.\qedhere
\end{enumerate}
\end{proof}
\subsection{Hyperboloidal evolution}
In the final step we remove the correction term $\mathbf{C}_{s_{0}}(\mathbf{\Phi},\mathbf{f})$ in the right-hand side $\mathbf{K}_\mathbf{f}$ by adjusting the blowup time $T$ and thereby solve the nonlinear problem.
\begin{proposition}
\label{hypEvo}
Let $R\geq \frac{1}{2}$ and $k\in\NN$, $k\geq \frac{d+1}{2}$. There exist positive constants $M,\delta,\varepsilon>0$ such that for any $\mathbf{f}\in \B^{k+\frac{d+1}{2}}_{\delta/{M^{2}},\varepsilon}$ there exists a $T_{\mathbf{f}} \in \overline{\BB_{\delta/M}(1)}$ and a unique map $\mathbf{\Phi}_{\mathbf{f}}\in \X^{k}_{R}(s_{0})$ with $\| \mathbf{\Phi}_{\mathbf{f}} \|_{\X^{k}_{R}(s_{0})} \leq \delta$ that solves
\begin{equation*}
\mathbf{\Phi}(s) = \mathbf{S}(s-s_{0}) \mathbf{U}(\mathbf{f},T_{\mathbf{f}}) + \int_{s_{0}}^s \mathbf{S}(s-s') \mathbf{N}(\mathbf{\Phi}(s')) \dd s'
\end{equation*}
for all $s\geq s_{0}=\log\big(-\frac{h(0)}{1+2\varepsilon}\big)$.
\end{proposition}
\begin{proof}
Pick $\varepsilon>0$ as in \Cref{InitDatOpProp} and define the initial data operator $\mathbf{U}: \B^{k+\frac{d+1}{2}}_{1,\varepsilon} \times \overline{\BB_\varepsilon(1)} \rightarrow \HH^{k}_\rad(\BB^{d}_{R})$. Pick $C,\delta>0$ as obtained in \Cref{stableEvo}. We can then choose $M\geq 1$ with $\delta/M < \varepsilon$ and get from the second part of \Cref{InitDatOpProp}
\begin{equation*}
\| \mathbf{U}(\mathbf{f},T) \|_{\HH^{k}(\BB^{d}_{R})} \lesssim |T-1| + |T-1|^{2} + \| \mathbf{f} \|_{\HH^{k+\frac{d+1}{2}}(\RR^d)} \lesssim \frac{\delta}{M}
\end{equation*}
for all $\mathbf{f}\in\B^{k+\frac{d+1}{2}}_{\delta/{M^{2}},\varepsilon}$ and all $T\in\overline{\BB_{\delta/M}(1)}$, which implies
\begin{equation*}
\| \mathbf{U}(\mathbf{f},T) \|_{\HH^{k}(\BB^{d}_{R})} \leq \frac{\delta}{C}
\end{equation*}
for any such $\mathbf{f}$, $T$ after possibly enlarging $M$. In this situation, by the first part of \Cref{stableEvo} there is a unique solution $\mathbf{\Phi}_{\mathbf{U}(\mathbf{f},T)} \in \X^{k}_{R}(s_{0})$ with $\| \mathbf{\Phi}_{\mathbf{U}(\mathbf{f},T)} \|_{\X^{k}_{R}(s_{0})} \leq \delta$ that solves the fixed point equation $\mathbf{K}_{\mathbf{U}(\mathbf{f},T)}(\mathbf{\Phi}) = \mathbf{\Phi}$, i.e.
\begin{equation*}
\mathbf{\Phi}_{\mathbf{U}(\mathbf{f},T)}(s) = \mathbf{S}(s-s_{0}) \big( \mathbf{U}(\mathbf{f},T) - \mathbf{C}_{s_{0}}(\mathbf{\Phi}_{\mathbf{U}(\mathbf{f},T)},\mathbf{U}(\mathbf{f},T)) \big) + \int_{s_{0}}^s  \mathbf{S}(s-s') \mathbf{N}(\mathbf{\Phi}_{\mathbf{U}(\mathbf{f},T)}(s')) \dd s'
\end{equation*}
for all $s\geq s_{0} \coloneqq \log\big(-\frac{h(0)}{1+2\varepsilon} \big)$. Now the task is to show that there is a $T_\mathbf{f}$ in $\overline{\BB_{\delta/M}(1)}$ for which the correction term
\begin{equation*}
\mathbf{C}_{s_{0}}(\mathbf{\Phi}_{\mathbf{U}(\mathbf{f},T)},\mathbf{U}(\mathbf{f},T)) = \mathbf{P} \left( \mathbf{U}(\mathbf{f},T) + \int_{s_{0}}^\infty \ee^{s_{0}-s'} \mathbf{N}(\mathbf{\Phi}_{\mathbf{U}(\mathbf{f},T)}(s')) \dd s' \right)
\end{equation*}
vanishes. Since $\ran(\mathbf{C}_{s_{0}}) \subseteq \ran(\mathbf{P}) = \langle \mathbf{f}_{1}^{*} \rangle$ this is equivalent to the inner product
\begin{equation*}
\Big( \mathbf{C}_{s_{0}}(\mathbf{\Phi}_{\mathbf{U}(\mathbf{f},T)},\mathbf{U}(\mathbf{f},T)) \,\Big|\, \mathbf{f}_{1}^{*} \Big)_{\HH^{k}(\BB^{d}_{R})} =
\widetilde{C_\varepsilon} (T-1) + \widetilde{\varphi}_\mathbf{f}(T)
\end{equation*}
being zero, where we use \Cref{InitDatOpExp} to see that $\widetilde{\varphi}_\mathbf{f} : \overline{\BB_{\delta/M}(1)} \rightarrow \RR$ is given by
\begin{equation*}
\widetilde{\varphi}_\mathbf{f}(T) = \Big( \mathbf{P} \widetilde{\mathbf{U}}(\mathbf{f},T) + \int_{s_{0}}^\infty \ee^{s_{0}-s'} \mathbf{P} \mathbf{N}(\mathbf{\Phi}_{\mathbf{U}(\mathbf{f},T)}(s')) \dd s' \,\Big|\, \mathbf{f}_{1}^{*} \Big)_{\HH^{k}(\BB^{d}_{R})} \,.
\end{equation*}
The second part of \Cref{stableEvo} yields that $\mathbf{\Phi}_{\mathbf{U}(\mathbf{f},T)}$ depends Lipschitz continuously on the initial data $\mathbf{U}(\mathbf{f},T)$, which in turn depend continuously on $T$ by the first part of \Cref{InitDatOpProp}. So $\widetilde{\varphi}_\mathbf{f}$ is continuous. Finding a zero corresponds to the fixed point problem $T = 1 - \widetilde{C_\varepsilon}^{-1} \widetilde{\varphi}_\mathbf{f}(T)$. To this end, define
\begin{equation*}
\varphi_\mathbf{f} : \overline{\BB_{\delta/M}(1)} \rightarrow \RR \,, \quad \varphi_\mathbf{f}(T) = 1 - \widetilde{C_\varepsilon}^{-1} \widetilde{\varphi}_\mathbf{f}(T) \,.
\end{equation*}
Moreover, Cauchy-Schwarz together with \Cref{InitDatOpProp} and the bounds in the proof of \Cref{stableEvo} yield
\begin{equation*}
|\widetilde{\varphi}_\mathbf{f}(T)| \lesssim |T-1|^{2} + \| \mathbf{f} \|_{\HH^{k}(\BB^{d}_{R})} + \delta^{2} \lesssim \delta/{M^{2}} + \delta^{2}
\end{equation*}
for all $T\in\overline{\BB_{\delta}(1)}$. Upon shrinking $\delta$ and enlarging $M$, this implies $\varphi_\mathbf{f} \in C\big( \overline{\BB_{\delta/M}(1)}, \overline{\BB_{\delta/M}(1)})$. Brouwer's fixed point theorem implies the existence of a fixed point $T_\mathbf{f} \in\overline{\BB_{\delta/M}(1)}$ of $\varphi_\mathbf{f}$.
\end{proof}
\subsection{Proof of the main result}
\begin{figure}
\centering
\includegraphics[width=0.8\textwidth]{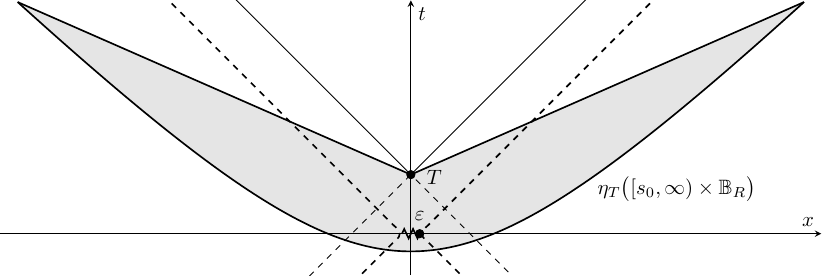}
\caption{This is a depiction of the region where we establish the stable hyperboloidal evolution.}
\end{figure}
Now we are ready to prove stability of the future development for the Cauchy problem of the Yang-Mills equation under small perturbations of the blowup.
\begin{proof}[Proof of \Cref{YMglobThm}]
\begin{enumerate}[wide, itemsep=1em, topsep=1em]
\item We have proved in \Cref{hypEvo} that there are positive constants $M,\delta,\varepsilon>0$ such that for all radial $(f,g)\in C^\infty(\BB^{d}_{R}) \times C^\infty(\BB^{d}_{R})$ with $\supp(f,g)\subseteq \BB_\varepsilon^d\times\BB_\varepsilon^d$ and $\| (f,g) \|_{\HH^{k+\frac{d+1}{2}}(\BB^{d}_{R})} \leq \delta/{M^{2}}$ there is a $T\in[1-\delta/M,1+\delta/M]$ and a unique $\mathbf{\Phi}_{f,g}\in C\big( [s_{0},\infty),\HH^{k}_\rad(\BB^{d}_{R}) \big)$ with $\| \mathbf{\Phi}_{f,g}(s) \|_{\HH^{k}(\BB^{d}_{R})} \leq \delta \ee^{-\omega_{0} s}$ and
\begin{equation*}
\mathbf{\Phi}_{f,g}(s) = \mathbf{S}(s-s_{0}) \mathbf{U}(f,g,T) + \int_{s_{0}}^s \mathbf{S}(s-s') \mathbf{N}(\mathbf{\Phi}_{f,g}(s')) \dd s'
\end{equation*}
for all $s\geq s_{0}$. Standard regularity arguments yield $\mathbf{\Phi}_{f,g} \in C^\infty([s_{0},\infty)\times\BB^{d}_{R})\times C^\infty([s_{0},\infty)\times\BB^{d}_{R})$ and a classical solution to
\begin{equation*}
\begingroup
\renewcommand{\arraystretch}{1.3}
\left\{
\begin{array}{rcl}
\pd_s \mathbf{\Phi}_{f,g}(s,\,.\,) &=& \mathbf{L} \mathbf{\Phi}_{f,g}(s,\,.\,) + \mathbf{N}(\mathbf{\Phi}_{f,g}(s,\,.\,)) \,, \\
\mathbf{\Phi}_{f,g}(s_{0},\,.\,) &=& \mathbf{U}(f,g,T) \,.
\end{array}
\right.
\endgroup
\end{equation*}
It follows that $\widetilde{v}(s,y) = \ee^{2s}\Phi_{1}(s,y)$ defines a smooth function $\widetilde{v} \in C^\infty\big( [s_{0},\infty)\times\BB^{d}_{R} \big)$ that satisfies \Cref{HSCYM}. By construction, the field $u \coloneqq u_{T}^{*} + \widetilde{v}\circ\eta_{T}^{-1}$ solves \Cref{YM2} in the domain $\eta_{T} \big( [s_{0},\infty) \times \BB^{d}_{R} \big)$. The uniqueness part in \Cref{locCauchy} implies the initial conditions $u(0,\,.\,) = u_{T}^{*} + f$ and  $\pd_{0} u(0,\,.\,) = \pd_{0} u_{T}^{*} + g$. Also, finite speed of propagation yields $u\in C^\infty(\Omega_{T,R})$ with $u = u_{1}^{*}$ in $\Omega_{T,R} \setminus \eta_{T} \big( [s_{0},\infty) \times \BB^{d}_{R} \big)$.
\item By the above construction we have
\begin{align*}
(u\circ\eta_{T})(s,\,.\,) &= (u_{T}^{*}\circ\eta_{T})(s,\,.\,) + \ee^{2s}\Phi_{1}(s,\,.\,) \,, \\
\pd_s (u\circ\eta_{T})(s,\,.\,) &= \pd_s (u_{T}^{*}\circ\eta_{T})(s,\,.\,) + \ee^{2s}\Phi_2(s,\,.\,) \,,
\end{align*}
and the estimates for $\left\| (u-u_{T}^{*})\circ\eta_{T}(s,\,.\,) \right\|_{H^{k}(\BB^{d}_{R})}$ and $\left\| \pd_s (u-u_{T}^{*})\circ\eta_{T}(s,\,.\,) \right\|_{H^{k-1}(\BB^{d}_{R})}$ follow from $\| \mathbf{\Phi}_{f,g}(s) \|_{\HH^{k}(\BB^{d}_{R})} \leq \delta \ee^{-\omega_{0} s}$. With this, we conclude from the Sobolev embedding $H^{k}(\BB^{d}_{R}) \hookrightarrow L^{\infty}(\BB^{d}_{R})$ that
\begin{align*}
\ee^{-2s(t)} \big|(u - u_{T}^{*})\circ\eta_{T}\big(s(t),y\big)\big| &\lesssim \| (u - u_{T}^{*})\circ\eta_{T}\big(s(t),\,.\,\big) \|_{H^{k}(\BB^{d}_{R})} \lesssim \ee^{-\omega_{0}s(t)} \,, \\
\big|u_{T}^{*}\circ\eta_{T}\big(s(t),y\big)\big| &= \ee^{2s(t)} \frac{a_d}{b_d h(y)^{2} + |y|^{2}} \simeq \ee^{2s(t)} \,,
\end{align*}
for all $y\in\BB^{d}_{R}$ and all $0\leq t < T$, and hence the bound along the hyperboloids.
\qedhere
\end{enumerate}
\end{proof}
\appendix
\addtocontents{toc}{\protect\setcounter{tocdepth}{1}}
\section{Geometry of hyperboloidal similarity coordinates}
\label{geomHSC}
Let $d\in\NN$ and consider $(1+d)$-dimensional Minkowski space-time $\RR^{1,d}$. Let $T\in\RR$ be a parameter and define the region
\begin{equation*}
\Omega_{T} = \big\{ (t,x)\in\RR^{1,d} \mid |x|>-(T-t) \big\} \,.
\end{equation*}
The coordinates given by the diffeomorphism
\begin{equation*}
\eta_{T}: \RR^{1+d} \rightarrow \Omega_{T}\,, \quad (s,y) \mapsto \big( T+\ee^{-s}h(y),\ee^{-s}y \big) \,,
\end{equation*}
with inverse
\begin{equation*}
\eta_{T}^{-1}: \Omega_{T} \rightarrow \RR^{1+d} \,, \quad (t,x) \mapsto \big( \log g_{T}(t,x) , g_{T}(t,x) x \big) \,,
\end{equation*}
where
\begin{equation*}
h(y) = \sqrt{2+|y|^{2}}-2
\end{equation*}
and
\begin{equation*}
g_{T}(t,x) = \frac{1}{ (T-t) + \frac{1}{2}\sqrt{2\big( (T-t)^{2} + |x|^{2} \big)} } \,,
\end{equation*}
are called \emph{hyperboloidal similarity coordinates}. In the following subsections, we use Latin indices $i,j,k \in \{ 1,\ldots,d \}$ and Greek indices $\kappa,\lambda,\mu,\nu\in\{0,1,\ldots,d\}$. The components of the Minkowski metric and its inverse, respectively, are
\begin{align*}
\minkmet_{00} &= -1 \,, & \minkmet_{0i} &= \minkmet_{i0} = 0 \,, & \minkmet_{ij} = \eucmet_{ij} \,, \\
\minkmet^{00} &= -1 \,, & \minkmet^{0i} &= \minkmet^{i0} = 0 \,, & \minkmet^{ij} = \eucmet^{ij} \,.
\end{align*}
\subsection{Jacobian of hyperboloidal similarity coordinates}
\label{JacobianHSC}
The components of the Jacobian of hyperboloidal similarity coordinates read
\begin{align}
\pd_{0}{\eta_{T}}^0(s,y) &= -\ee^{-s}h(y) \,,&
\pd_{i}{\eta_{T}}^0(s,y) &= \ee^{-s} \pd_{y^i}h(y) \,,\\
\pd_{0}{\eta_{T}}^i(s,y) &= -\ee^{-s} y^i \,,&
\pd_j{\eta_{T}}^i(s,y) &= \ee^{-s} \eucmet^i_j \,,
\end{align}
and
\begin{align}
\pd_{0}{\eta_{T}^{-1}}^0(t,x) &= \frac{\pd_{t} g_{T}(t,x)}{g_{T}(t,x)} \,,&
\pd_{i}{\eta_{T}^{-1}}^0(t,x) &= \frac{\pd_{x^i} g_{T}(t,x)}{g_{T}(t,x)} \,, \\
\pd_{0}{\eta_{T}^{-1}}^i(t,x) &= x^i\pd_{t} g_{T}(t,x) \,,&
\pd_j{\eta_{T}^{-1}}^i(t,x) &=  x^i\pd_{x^{j}} g_{T}(t,x) + g_{T}(t,x) \eucmet^i_j \,.
\end{align}
By block matrix inversion, or direct computation, the components of the inverse matrices
\begin{align}
(\pd_{0}{\eta_{T}^{-1}}^0\circ\eta_{T})(s,y) &= \frac{\ee^s}{y^{k}\pd_{y^{k}}h(y)-h(y)} \,, \\
(\pd_{i}{\eta_{T}^{-1}}^0\circ\eta_{T})(s,y) &= -\ee^s \frac{\pd_{y^i}h(y)}{y^{k}\pd_{y^{k}}h(y)-h(y)} \,, \\
(\pd_{0}{\eta_{T}^{-1}}^i\circ\eta_{T})(s,y) &= \frac{\ee^s y^i}{y^{k}\pd_{y^{k}}h(y)-h(y)} \,, \\
(\pd_j{\eta_{T}^{-1}}^i\circ\eta_{T})(s,y) &=  \ee^s \left( \eucmet_j^i - \frac{y^i\pd_{y^{j}}h(y)}{y^{k}\pd_{y^{k}}h(y)-h(y)} \right) \,,
\end{align}
and
\begin{align}
(\pd_{0}{\eta_{T}}^0\circ\eta_{T}^{-1})(t,x) &= T-t \,,&
(\pd_{i}{\eta_{T}}^0\circ\eta_{T}^{-1})(t,x) &= -\frac{1}{g_{T}(t,x)} \frac{\pd_{x^i}g_{T}(t,x)}{\pd_{t} g_{T}(t,x)} \,,\\
(\pd_{0}{\eta_{T}}^i\circ\eta_{T}^{-1})(t,x) &= -x^i \,,&
(\pd_j{\eta_{T}}^i\circ\eta_{T}^{-1})(t,x) &= \frac{1}{g_{T}(t,x)} \eucmet^i_j \,,
\end{align}
are obtained. For a first application, we obtain the transformations of partial derivatives
\begin{align}
\label{PDHSC0}
(\pd_{0} u \circ \eta_{T})(s,y) &= \frac{\ee^s}{y^{k}\pd_{y^{k}} h(y) - h(y)} \big( \pd_s + y^{k} \pd_{y^{k}} \big) v(s,y) \,, \\
\label{PDHSCj}
(\pd_j u \circ \eta_{T})(s,y) &= \ee^s \pd_{y^{j}} v(s,y) - \frac{\ee^s \pd_{y^{j}} h(y)}{y^{k}\pd_{y^{k}} h(y) - h(y)} \big( \pd_s + y^{k}\pd_{y^{k}} \big) v(s,y) \,,
\end{align}
for $u,v\in C^\infty(\RR^d)$ that are related by $v = u \circ \eta_{T}$. The determinant from the Jacobian is readily computed by adding a $y^i$-multiple of the $(i+1)$-st column to the first column,
\begin{equation*}
\det(\pd\eta_{T})(s,y) =
\ee^{-(d+1)s} \big( y^{k}\pd_{y^{k}}h(y) - h(y) \big) \,.
\end{equation*}
\subsection{Minkowski metric in hyperboloidal similarity coordinates}
\label{metHSC}
The components of the Minkowski metric and its inverse are expressed in hyperboloidal similarity coordinates by
\begin{equation*}
\g_{\mu\nu} =
\pd_{\mu} {\eta_{T}}^{\kappa} \minkmet_{\kappa\lambda} \pd_{\nu} {\eta_{T}}^{\lambda} \,, \quad \g^{\mu\nu} =
\big( \pd_{\kappa} {\eta_{T}^{-1}}^\mu \circ\eta_{T} \big) \minkmet^{\kappa\lambda} \big( \pd_{\lambda} {\eta_{T}^{-1}}^\nu \circ\eta_{T} \big)
\end{equation*}
which yields
\begin{align}
\g_{00}(s,y) &= \ee^{-2s}\left( -h(y)^{2} + |y|^{2} \right) \,, \\
\g_{i0}(s,y) &= \g_{0i}(s,y) = \ee^{-2s}\left( h(y) \pd_{y^i} h(y)- y_{i} \right) \,, \\
\g_{ij}(s,y) &= \ee^{-2s} \left( \eucmet_{ij} - \pd_{y^i}h(y)\pd_{y^{j}}h(y) \right) \,.
\end{align}
and
\begin{align}
\g^{00}(s,y) &= -\ee^{2s} \frac{1 - \pd_{y^{k}} h(y) \pd_{y_k} h(y)}{\left( y^{k}\pd_{y^{k}}h(y) - h(y)\right)^{2}} \,, \\
\g^{i0}(s,y) &= \g^{0i}(s,y) = \ee^{2s} \left( - \frac{1 - \pd_{y^{k}} h(y) \pd_{y_k} h(y)}{\left( y^{k}\pd_{y^{k}}h(y) - h(y)\right)^{2}}y^i - \frac{\pd_{y_{i}}h(y)}{y^{k}\pd_{y^{k}}h(y) - h(y)} \right) \,, \\
\g^{ij}(s,y) &= \ee^{2s} \left( \eucmet^{ij} - \frac{1 - \pd_{y^{k}} h(y) \pd_{y_k} h(y)}{\left( y^{k}\pd_{y^{k}}h(y) - h(y)\right)^{2}}y^iy^{j} - \frac{y^{j}\pd_{y_{i}}h(y)+y^i\pd_{y_j}h(y)}{y^{k}\pd_{y^{k}}h(y) - h(y)} \right) \,,
\end{align}
respectively. The metric determinant $\metdet{\g} = |\det( \pd\eta_{T} )|$ reads
\begin{equation*}
\metdet{\g}(s,y) =\ee^{-(d+1)s}\left(y^{k}\pd_{y^{k}}h(y)-h(y)\right) \,.
\end{equation*}
\subsection{Christoffel symbols of the metric}
\label{ChristoffelHSC}
The Christoffel symbols of the Minkowski metric transform in hyperboloidal similarity coordinates to
\begin{equation*}
\Gamma^{\lambda}{}_{\mu\nu}  = \frac{1}{2} \g^{\lambda\kappa} \big( \pd_{\mu}\g_{\nu\kappa} + \pd_{\nu}\g_{\kappa\mu} - \pd_{\kappa}\g_{\mu\nu} \big) = (\pd_{\mu}\pd_{\nu}{\eta_{T}}^{\kappa}) (\pd_{\kappa} {\eta_{T}^{-1}}^{\lambda} \circ\eta_{T}) \,.
\end{equation*}
This gives the expressions
\begin{align}
\Gamma^0{}_{00}(s,y) &= -1 \,,&
\Gamma^0{}_{i0}(s,y) &= \Gamma^0{}_{0i}(s,y) = 0 \,,&
\Gamma^0{}_{ij}(s,y) &= \frac{\pd_{y^i}\pd_{y^{j}}h(y)}{y^{k}\pd_{y^{k}}h(y) - h(y)} \,, \\
\Gamma^{k}{}_{00}(s,y) &= 0 \,,&
\Gamma^{k}{}_{i0}(s,y) &= \Gamma^{k}{}_{0i}(s,y) = -\eucmet_{i}^{k} \,,&
\Gamma^{k}{}_{ij}(s,y) &= \frac{\pd_{y^i}\pd_{y^{j}}h(y)}{y^{k}\pd_{y^{k}}h(y) - h(y)}y^{k} \,.
\end{align}
Also recall the identity
\begin{equation}
\label{contractedChristoffel}
\frac{1}{\metdet{\g}} \pd_{\mu} \big( \g^{\mu\nu} \metdet{\g} \big) = - \g^{\kappa\lambda} \Gamma^\nu{}_{\kappa\lambda} \,.
\end{equation}
\section{Equivalent descriptions for radial Sobolev norms}
Here, we review the auxiliary results from \cite[Appendix A]{MR4338226}. The following inequalities are versions of \emph{Hardy's inequality}.
\begin{lemma}
\label{Hardy}
\begin{enumerate}[itemsep=1em, topsep=1em]
\item\label{HardyBalls} Let $R,s\in\RR$ such that $R>0$ and $s<-1/2$. Then
\begin{equation*}
\big\| |\,.\,|^s f \big\|_{L^{2}(\BB_{R})} \lesssim \big\| |\,.\,|^{s+1} f' \big\|_{L^{2}(\BB_{R})}
\end{equation*}
for all $f\in C^\infty(\overline{\BB_{R}})$ with $\lim_{x\to 0} \left( |x|^{s+\frac{1}{2}} |f(x)| \right) = 0$.
\item\label{HardyR} Let $s\in\RR$ such that $s\neq -1/2$. Then
\begin{equation*}
\big\| |\,.\,|^s f \big\|_{L^{2}(\RR)} \lesssim \big\| |\,.\,|^{s+1} f' \big\|_{L^{2}(\RR)}
\end{equation*}
for all $f\in \mathcal{S}(\RR)$ with $\lim_{x\to 0} \left( |x|^{s+\frac{1}{2}} |f(x)| \right) = 0$.
\item\label{classicHardy} Let $d\in\NN$ and $s\in \big[ 0, \frac{d}{2} \big)$. Then
\begin{equation*}
\big\| |\,.\,|^{-s} f \big\|_{L^{2}(\RR^d)} \lesssim \big\| f \big\|_{\dot{H}^s(\RR^d)} 
\end{equation*}
for all $f\in\mathcal{S}(\RR^d)$.
\end{enumerate}
\end{lemma}
\begin{proof}
\begin{enumerate}[wide, itemsep=1em, topsep=1em]
\item Using integration by parts and Young's inequality, we have
\begin{align*}
&\quad\int_{\BB_{R}} |x|^{2s} f(x)^{2} \dd x \\&=
\int_{-R}^0 (-x)^{2s} f(x)^{2} \dd x + \int_{0}^R x^{2s} f(x)^{2} \dd x \\&=
-\frac{1}{2s+1} \lim_{x \nearrow 0} (-x)^{2s+1} f(x)^{2} +
\frac{1}{2s+1} R^{2s+1} f(-R)^{2} +
\int_{-R}^0 \frac{2}{2s+1} (-x)^{2s+1} f(x)f'(x) \dd x \\&\quad+
\frac{1}{2s+1} R^{2s+1} f(R)^{2} -
\frac{1}{2s+1} \lim_{x \searrow 0} x^{2s+1} f(x)^{2} -
\int_{0}^R \frac{2}{2s+1} x^{2s+1} f(x)f'(x) \dd x  \\& \leq
\frac{1}{2s+1} R^{2s+1} \left( f(R)^{2} + f(-R)^{2} \right) \\&\quad+
\int_{-R}^0 \frac{1}{2} (-x)^{2s} f(x)^{2} + \frac{2}{(2s+1)^{2}} (-x)^{2s+2} f'(x)^{2} \dd x \\&\quad+
\int_{0}^R \frac{1}{2} x^{2s} f(x)^{2} + \frac{2}{(2s+1)^{2}} x^{2s+2} f'(x)^{2} \dd x \\&=
\frac{1}{2s+1} R^{2s+1} \left( f(R)^{2} + f(-R)^{2} \right) +
\frac{1}{2} \int_{\BB_{R}} |x|^{2s} f(x)^{2} \dd x +
\frac{2}{(2s+1)^{2}} \int_{\BB_{R}} |x|^{2s+2} f'(x)^{2} \dd x \,.
\end{align*}
Since $2s+1 < 0$ we can estimate the remaining boundary term from above by $0$ and the result follows.
\item By letting $R\to \infty$ in the above calculation we infer the result.
\item This is the classical Hardy's inequality from harmonic analysis, see \cite[Lemma A.2]{MR2233925}.
\qedhere
\end{enumerate}
\end{proof}
We rely on a convenient description of the Sobolev norm for radial functions in odd space dimensions in terms of a weighted Sobolev norm for the radial representative.
\begin{lemma}
\label{oddSobolev}
Let $d,k\in\NN_{0}$, $d\geq 1$ odd. We have
\begin{equation*}
\|f\|_{H^{k}(\RR^d)} \simeq \big\| |\,.\,|^{\frac{d-1}{2}} \widehat{f} \big\|_{H^{k}(\RR)}
\end{equation*}
for all radial $f\in \mathcal{S}(\RR^d)$, where $f(\,.\,) = \widehat{f}(|\,.\,|)$.
\end{lemma}
\begin{proof}
\begin{enumerate}[wide, itemsep=1em, topsep=1em]
\item[``$\gtrsim$'':] After a change of variables we see
\begin{equation*}
\| f \|_{L^{2}(\RR^d)}^{2} = \int_{\RR^d} |f(x)|^{2} \dd x = \int_{\RR^d} |\widehat{f}(|x|)|^{2} \dd x \simeq \int_\RR |\widehat{f}(r)|^{2} |r|^{d-1} \dd r = \big\| |\,.\,|^{\frac{d-1}{2}} \widehat{f} \big\|_{L^{2}(\RR)}^{2} \,.
\end{equation*}
Now exploit radiality for $\widehat{f}(r) = f(re_{1})$ to see $\widehat{f}^{(k)}(r) = (\pd_{1}^{k} f)(re_{1})$. Since $d$ is odd, a derivative $\pd_{r}^\ell$ hitting $r^{\frac{d-1}{2}}$ contributes to the Sobolev norm only when $\ell \leq \frac{d-1}{2}$, i.e. $\ell < \frac{d}{2}$. In this case Hardy's inequality yields
\begin{equation*}
\big\|  |\,.\,|^{-\ell} |\,.\,|^{\frac{d-1}{2}} \widehat{f}^{(k-\ell)} \big\|_{L^{2}(\RR)} \simeq \big\|  |\,.\,|^{-\ell} \pd_{1}^{k-\ell} f \big\|_{L^{2}(\RR^d)} \lesssim \big\| \pd_{1}^{k-\ell} f \big\|_{\dot{H}^{\ell}(\RR^d)} \lesssim \| f \|_{\dot{H}^{k}(\RR^d)} \,.
\end{equation*}
Using this in the classic definition of the homogeneous Sobolev norm we get one inequality.
\item[``$\lesssim$'':] To see the other inequality, we introduce the Fourier transform $\FT$ on $\mathcal{S}(\RR^d)$ in \Cref{FourierTransform}. If $f$ is radial then $\FT f$ is also radial and
\begin{equation*}
\widehat{\FT f}(\rho) = (2\pi)^{\frac{d}{2}} \rho^{-\frac{d}{2}+1} \int_{0}^\infty \widehat{f}(r) J_{\frac{d}{2}-1}(\rho r) r^{\frac{d}{2}} \dd r \,,
\end{equation*}
cf. \cite{MR1970295} for a reference. The Bessel function of nonnegative half-integer order is given explicitly by
\begin{align*}
J_{\frac{d}{2}-1}(z) &= \sqrt{\frac{2}{\pi}} z^{\frac{d-3}{2} + \frac{1}{2}} \left( - \frac{1}{z} \frac{\od}{\od z}\right)^{\frac{d-3}{2}} \frac{\sin(z)}{z} \\&=
\sqrt{\frac{2}{\pi}} z^{-\frac{1}{2}} \left( P_{\frac{d-3}{2}}(z^{-1}) \sin(z) -  Q_{\frac{d-5}{2}}(z^{-1}) \cos(z) \right) \,,
\end{align*}
where $P_n$, $Q_n$ are polynomials of degree $n$ with $P_n(-z) = (-1)^n P_n(z)$ and $Q_n(-z) = (-1)^n Q_n(z)$, see \cite{teschl1998topicsRA}. It follows crucially from the parity of $P_n,Q_n$ and oddness of $d$ that
\begin{align*}
\widehat{\FT f}(\rho) &=
(2\pi)^{\frac{d-1}{2}} \rho^{-\frac{d-1}{2}} \int_\RR \ee^{-\ii\rho r} \widehat{f}(r) r^{\frac{d-1}{2}} \Big( \ii P_{\frac{d-3}{2}} \big( (\rho r)^{-1} \big) - Q_{\frac{d-5}{2}} \big( (\rho r)^{-1} \big) \Big) \dd r \\&=
(2\pi)^{\frac{d-1}{2}} \sum_{j=0}^{\frac{d-3}{2}} p_j \rho^{-\frac{d-1}{2}-j} \FT\big( (\,.\,)^{\frac{d-1}{2} - j} \widehat{f}  \big)(\rho)
\end{align*}
where $\ii P_{\frac{d-3}{2}}(z) - Q_{\frac{d-5}{2}}(z) \eqqcolon \sum_{j=0}^{\frac{d-3}{2}} p_j z^{j}$. So, placing this in the homogeneous Sobolev norm and applying \Cref{Hardy} successively gives
\begin{align*}
\| f \|_{\dot{H}^{k}(\RR^d)} &= \big\| |\,.\,|^{k} \FT f \big\|_{L^{2}(\RR^d)} \simeq
\big\| |\,.\,|^{\frac{d-1}{2}} |\,.\,|^{k} \widehat{\FT f} \big\|_{L^{2}(\RR)} \\&\lesssim
\sum_{j=0}^{\frac{d-3}{2}} \big\| |\,.\,|^{k - j} \FT \big( (\,.\,)^{\frac{d-1}{2} - j} \widehat{f} \big) \big\|_{L^{2}(\RR)} \\&\lesssim
\sum_{j=0}^{\frac{d-3}{2}} \big\| |\,.\,|^{k} \FT \big( (\,.\,)^{\frac{d-1}{2} - j} \widehat{f} \big)^{(j)} \big\|_{L^{2}(\RR)} \\&\simeq
\big\| |\,.\,|^{k} \FT \big( (\,.\,)^{\frac{d-1}{2}} \widehat{f} \big) \big\|_{L^{2}(\RR)} \\&\simeq
\big\| |\,.\,|^{\frac{d-1}{2}} \widehat{f} \big\|_{\dot{H}^{k}(\RR)} \,.
\end{align*}
This yields the other inequality.\qedhere
\end{enumerate}
\end{proof}
We construct types of Sobolev extensions in order to transfer this result to open balls.
\begin{lemma}
\label{SobolevExtension}
Let $R>0$ and $k\in\NN_{0}$. There is a map $\mathcal{E}: C^{k}(\overline{\BB_{R}}) \rightarrow C^{k}(\RR)$ such that
\begin{equation*}
(\mathcal{E} f)|_{\BB_{R}} = f \,, \qquad \supp(\mathcal{E} f) \subseteq \BB_{2R} \,, \qquad \| \mathcal{E} f \|_{H^{k}(\RR\setminus\BB_{R})} \lesssim \|f\|_{H^{k}(\BB_{R}\setminus\BB_{R/2})} \,,
\end{equation*}
for all $f\in C^{k}(\overline{\BB_{R}})$.
\end{lemma}
\begin{proof}
The idea is to extend functions outside of $\BB_{R}$ by a shifted version of them that carries the mass on $\BB_{R}\setminus\BB_{R/2}$ and then add the Taylor polynomial of a controlled function that ensures differentiability at the endpoints, cf. \cite[Lemma A.4]{MR4338226}, \cite[Lemma B.2]{MR3742520}. Concretely, let $\chi\in C^\infty(\RR)$ be a smooth cutoff with $\chi = 1$ on $\overline{\BB_{1}}$ and $\chi = 0$ on $\RR\setminus\BB_{3/2}$ and put
\begin{equation*}
(\mathcal{E}f)(x) =
\begin{cases}
0\,, & x\in (-\infty,-2R) \,, \\
\displaystyle \chi(\tfrac{x}{R}) \Big( - f(-2R-x) + \sum_{j=0}^{\left\lfloor \tfrac{k}{2} \right\rfloor} \frac{2f^{(2j)}(-R)}{(2j)!} (x+R)^{2j} \Big) \,, & x\in[-2R,-R) \,, \\
f(x) \,, & x\in[-R,R] \,, \\
\displaystyle \chi(\tfrac{x}{R}) \Big( - f(2R-x) + \sum_{j=0}^{\left\lfloor \tfrac{k}{2} \right\rfloor} \frac{2f^{(2j)}(R)}{(2j)!} (x-R)^{2j} \Big) \,, & x\in(R,2R] \,, \\
0 \,, & x\in (2R,\infty) \,,
\end{cases}
\end{equation*}
for $f\in C^{k}(\overline{\BB_{R}})$. By construction we have $\lim_{x\to \pm R}(\mathcal{E}f)^{(j)}(x) = f^{(j)}(\pm R)$. Thus $\mathcal{E}f \in C^{k}(\RR)$. Moreover, $(\mathcal{E}f)|_{\overline{\BB_{R}}} = f$ and $\supp(\mathcal{E}f) \subset \BB_{2R}$ are evident. The control of the polynomial may be obtained from the fundamental theorem of calculus and the mean value theorem, so $|f^{(j)}(\pm R)| \lesssim \| f \|_{H^{k}(\BB_{R}\setminus\BB_{R/2})}$ for all $j=0,1,\ldots,k$. Hence,
\begin{align*}
\| \mathcal{E}f \|_{H^{k}(\RR\setminus\BB_{R})} &\lesssim \| \chi(\tfrac{\,.\,}{R}) f(-2R-\,.\,) \|_{H^{k}(\RR\setminus\BB_{R})} + \| \chi(\tfrac{\,.\,}{R}) f(2R-\,.\,) \|_{H^{k}(\RR\setminus\BB_{R})} + \| f \|_{H^{k}(\BB_{R}\setminus\BB_{R/2})} \\&\lesssim
\| f(-2R-\,.\,) \|_{H^{k}(\BB_{3R/2} \setminus\BB_{R})} + \| f(2R-\,.\,) \|_{H^{k}(\BB_{3R/2} \setminus\BB_{R})} + \| f \|_{H^{k}(\BB_{R}\setminus\BB_{R/2})} \\&\lesssim
\| f \|_{H^{k}(\BB_{R}\setminus\BB_{R/2})}  \,.
\qedhere
\end{align*}
\end{proof}
\Cref{oddSobolev} also holds on open balls with the aid of these extension operators.
\begin{lemma}
\label{oddSobolevBalls}
Let $R>0$, $d,k\in\NN_{0}$, $d\geq 1$ odd. We have
\begin{equation*}
\|f\|_{H^{k}(\BB^{d}_{R})} \simeq \big\| |\,.\,|^{\frac{d-1}{2}} \widehat{f} \big\|_{H^{k}(\BB_{R})}
\end{equation*}
for all $f\in C^\infty_\mathrm{rad}(\overline{\BB^{d}_{R}})$.
\end{lemma}
\begin{proof}
Let $\mathcal{E}$ be the extension operator from \Cref{SobolevExtension}. In combination with \Cref{oddSobolev} we get
\begin{align*}
\| f \|_{H^{k}(\BB^{d}_{R})} &\leq \big\| (\mathcal{E} \widehat{f} )(|\,.\,|) \big\|_{H^{k}(\RR^d)} \simeq
\big\| |\,.\,|^{\frac{d-1}{2}} (\mathcal{E} \widehat{f} ) \big\|_{H^{k}(\RR)} \\&\simeq
\big\| |\,.\,|^{\frac{d-1}{2}} \widehat{f} \big\|_{H^{k}(\BB_{R})} + \big\| |\,.\,|^{\frac{d-1}{2}} (\mathcal{E} \widehat{f} ) \big\|_{H^{k}(\RR\setminus\BB_{R})} \\&\lesssim
\big\| |\,.\,|^{\frac{d-1}{2}} \widehat{f} \big\|_{H^{k}(\BB_{R})} + \| \widehat{f} \|_{H^{k}(\BB_{R}\setminus\BB_{R/2})}\\&\lesssim
\big\| |\,.\,|^{\frac{d-1}{2}} \widehat{f} \big\|_{H^{k}(\BB_{R})} 
\end{align*}
and
\begin{align*}
\big\| |\,.\,|^{\frac{d-1}{2}} \widehat{f} \big\|_{H^{k}(\BB_{R})} &\leq
\big\| |\,.\,|^{\frac{d-1}{2}} \mathcal{E} \widehat{f} \big\|_{H^{k}(\RR)} \simeq
\big\| (\mathcal{E} \widehat{f})(|\,.\,|) \big\|_{H^{k}(\RR^d)} \\&\simeq
\|f \|_{H^{k}(\BB^{d}_{R})} + \big\| (\mathcal{E} \widehat{f})(|\,.\,|) \big\|_{H^{k}(\RR^d\setminus\BB^{d}_{R})} \\&\lesssim
\|f \|_{H^{k}(\BB^{d}_{R})} + \| \mathcal{E} \widehat{f} \|_{H^{k}(\RR\setminus\BB_{R})} \\&\lesssim
\|f \|_{H^{k}(\BB^{d}_{R})} + \| \widehat{f} \|_{H^{k}(\BB_{R} \setminus \BB_{R/2})} \\&\lesssim
\|f \|_{H^{k}(\BB^{d}_{R})}
\end{align*}
where we exploited radiality in the classic definition of the Sobolev norm and that all powers are bounded on $\BB_{2R}^d\setminus\BB_{R}^d$.
\end{proof}
\section{Bounds on inhomogeneities}
Also, we need to collect some estimates for integral operators of the following type.
\begin{lemma}
\label{integralSobolev}
Let $R>0$, $k,m,n\in\NN_{0}$ such that $n+1-m\geq 0$. Let $\varphi\in C^\infty(\overline{\BB_{R}})$ and define
\begin{equation*}
Tf(x) = x^{-m} \int_{0}^x y^n \varphi(y) f(y) \dd y \,, \quad x\in\overline{\BB_{R}}\setminus\{0\} \,,
\end{equation*}
for $f\in C^\infty(\overline{\BB_{R}})$.
\begin{enumerate}[itemsep=1em, topsep=1em]
\item Then $Tf$ extends to an element in $C^\infty(\overline{\BB_{R}})$ for all $f\in C^\infty(\overline{\BB_{R}})$.
\item Moreover, $T: \dom(T) \subset H^{k}(\BB_{R}) \rightarrow  H^{k}(\BB_{R})$, densely defined on $\dom(T) = C^\infty(\overline{\BB_{R}})$ by $Tf$, is a linear operator with
\begin{equation*}
\|Tf\|_{H^{k}(\BB_{R})} \lesssim \|f\|_{H^{k}(\BB_{R})}
\end{equation*}
for all $f\in C^\infty(\overline{\BB_{R}})$, that extends to a bounded linear operator on $H^{k}(\BB_{R})$.
\end{enumerate}
\end{lemma}
\begin{proof}
\begin{enumerate}[wide, itemsep=1em, topsep=1em]
\item By the fundamental theorem of calculus we have
\begin{equation*}
Tf(x) =  x^{-m} \int_{0}^x y^n \varphi(y) f(y) \dd y =  x^{n+1-m} \int_{0}^1 t^n \varphi(tx) f(tx) \dd t
\end{equation*}
for $x\in\overline{\BB_{R}}\setminus\{0\}$, which shows that $Tf\in C^\infty(\overline{\BB_{R}})$ for all $f\in C^\infty(\overline{\BB_{R}})$.
\item By the Leibniz rule
\begin{equation*}
\left| Tf^{(j)}(x) \right| \lesssim \sum_{i=0}^{j} \int_{0}^1 \left|f^{(i)}(tx) \right| \dd t \lesssim \sum_{i=0}^{j} x^{-1} \int_{0}^x \left| f^{(i)}(y) \right| \dd y
\end{equation*}
for all $j=0,1,\ldots,k$. Using \cref{HardyBalls} of \Cref{Hardy} we get
\begin{equation*}
\| (Tf)^{(j)} \|_{L^{2}(\BB_{R})} \lesssim \sum_{i=0}^{j} \| f^{(i)} \|_{L^{2}(\BB_{R})} \lesssim \| f \|_{H^{k}(\BB_{R})}
\end{equation*}
for all $j = 0,1,\ldots,k$, which proves the bound.
\qedhere
\end{enumerate}
\end{proof}

\end{document}